\documentclass[letterpaper,reqno,10pt]{amsart}
\usepackage{indentfirst} 
\usepackage{graphicx}
\usepackage{amssymb}
\usepackage{bm,amsbsy}
\usepackage{mathrsfs,mathtools}
\usepackage{amsmath}
\usepackage{amsbsy}
\usepackage{accents}
\usepackage{dsfont}
\usepackage{amsthm,stmaryrd}
\usepackage{thmtools}
\usepackage{enumitem}
\usepackage{bbm}     
\usepackage{slashed}       
\usepackage[colorlinks = true, linkcolor = black, citecolor = blue ]{hyperref}
\usepackage[usenames,dvipsnames]{xcolor} 
\usepackage{tikz}
\usepackage{etoolbox}

\makeatletter
\newcommand \Dotfill {\leavevmode \cleaders \hb@xt@ .77em{\hss .\hss }\hfill \kern \z@}
\makeatother

\usepackage{xpatch}
\makeatletter
\xpatchcmd{\@tocline}
  {\hfil\hbox to\@pnumwidth{\@tocpagenum{#7}}\par}
  {%
    \ifnum#1=1
      \hfill 
    \else
      \Dotfill 
    \fi
    \hbox to\@pnumwidth{\@tocpagenum{#7}}\par
  }
  {}
  {}
\makeatother

\makeatletter
\def\l@section{\@tocline{1}{8pt}{0pc}{0pc}{\bfseries}}
 \def\l@subsection{\@tocline{2}{0pt}{1.7pc}{6pc}{}}
\def\l@subsubsection{\@tocline{3}{0pt}{8pc}{8pc}{}}
 \makeatother

\usepackage[left=2.5cm, right=2.5cm, bottom=3cm]{geometry}

\theoremstyle{plain}
\newtheorem{theorem}{Theorem}[section]

\theoremstyle{plain}
\newtheorem{proposition}{Proposition}[section]
\newtheorem{lemma}[proposition]{Lemma}
\newtheorem{corollary}[proposition]{Corollary}
\newtheorem{definition}[proposition]{Definition}
\newtheorem{remark}[proposition]{Remark}

\newcommand{\isEquivTo}[1]{\underset{#1}{\sim}}

\newcommand{\vertiii}[1]{{\left\vert\kern-0.25ex\left\vert\kern-0.25ex\left\vert #1 
    \right\vert\kern-0.25ex\right\vert\kern-0.25ex\right\vert}}

\def\R{{\mathbb{R}}}
\def\z{{\zeta}}
\def\T{{\mathbb{X}}}
\def\H{{\mathcal{H}^+}}

\def\dr{{\mathrm{d}}}

\begin{document}

\title[Relatively non-degenerate estimates for massless Vlasov fields on Schwarzschild]{Relatively non-degenerate integrated decay estimates for massless Vlasov fields on Schwarzschild spacetimes}

\author[L\'eo Bigorgne]{L\'eo Bigorgne} \address{Univ Rennes, CNRS, IRMAR - UMR 6625, F-35000 Rennes, France.}
\email{leo.bigorgne@univ-rennes.fr}

\author[Renato Velozo Ruiz]{Renato Velozo Ruiz} \address{Department of Mathematics, University of Toronto, 40 St. George Street, Toronto, ON, Canada.}
\email{renato.velozo.ruiz@utoronto.ca}

\begin{abstract}
In this article, we make use of a weight function capturing the concentration phenomenon of unstable future-trapped causal geodesics. A projection $V_+$, on the tangent space of the null-shell, of the associated symplectic gradient turns out to enjoy good commutation properties with the massless Vlasov operator. This implies that $V_+f$ remains bounded, for any smooth solution $f$ to the massless Vlasov equation.

By identifying a well-chosen modification of $V_+$, we are able to construct a $W_{x,p}^{1,1}$ weighted norm for which any smooth solution to the massless Vlasov equation verifies an integrated local energy decay estimate without relative degeneration. Together with the $r^p$-weighted energy method of Dafermos--Rodnianski, we establish time decay for the energy norm. This norm allows for the control of the energy-momentum tensor $\mathrm{T}[f]$ as well as all its first order derivatives.

The method developed in this paper is in particular compatible with the approach of \cite{Mavq, DHRTq} used to study quasi-linear wave equations on black hole spacetimes. 

\end{abstract}


\maketitle

\tableofcontents

\section{Introduction}

In this paper, we study the boundedness and decay properties of massless collisionless systems on the exterior $(\mathcal{S},g_M)$ of a Schwarzschild black hole background of mass $M>0$. We describe the evolution of such systems of free falling particles with zero rest mass by a distribution function satisfying a transport equation along the null geodesic flow. More precisely, we study the solutions $f$ to the \emph{massless Vlasov equation on a fixed Schwarzschild spacetime}  
\begin{align}\label{vlasov_eqn_massless_intro}
\T_{g_M}(f)=0,
\end{align}
where $\T_{g_M}$ is the geodesic spray vector field. The distribution function $f \colon \mathcal{P }\to \R_+$ is a nonnegative function defined on the \emph{null-shell}   
\begin{equation}\label{def_first_intro_null_shell}
\mathcal{P} \coloneqq \Big\{ (x,p)\in T^*\mathcal{S} \; \big| \; g_x^{-1}(p,p)=0, \; \text{ $p$ is future-directed} \Big\}.
\end{equation}
The analysis of the asymptotic behaviour of the solutions to \eqref{vlasov_eqn_massless_intro} is motivated by the study of self-gravitating massless systems in relativistic kinetic theory \cite{A11,AGS}. In the framework of general relativity, they are modeled by the solutions $(\mathcal{M},g,f)$ of the \emph{massless Einstein--Vlasov system} \eqref{EV}
\begin{align} \tag{EV}
\begin{aligned}\label{EV}
\mathrm{Ric}(g)-\frac{1}{2}\mathrm{R}(g)\cdot g&=8\pi \mathrm{T}[f],\\
\T_g(f)&=0,
\end{aligned}
\end{align}
where $\mathrm{Ric}(g)$, $\mathrm{R}(g)$, and $\T_g$ are respectively the Ricci curvature tensor, the scalar curvature and the geodesic spray of the Lorentzian manifold $(\mathcal{M},g)$. The \emph{energy-momentum tensor} $\mathrm{T}[f]$ of the Vlasov field $f$ is given by
\begin{align} \label{energy_momentum_local_geometry_intro}
 \mathrm{T} [f]_{ \mu \nu } \coloneqq \int_{\mathcal{P}_x} f(x,p) p_\mu p_\nu \dr \mu_{\mathcal{P}_x},
\end{align}
where $\dr \mu_{\mathcal{P}_x}$ is the volume form induced by the metric $g$ on the fibers $\mathcal{P}_x$ of the null-shell. We refer to \cite{AR08,CB09,A11} for more information about this geometric PDE system. When the distribution function $f$ vanishes identically, the massless Einstein--Vlasov system reduces to the \emph{vacuum Einstein equations} $\mathrm{Ric}(g)\!=\!0$. Consequently, the members of the one-parameter family of Schwarzschild black holes can be viewed as stationary spherically symmetric solutions to \eqref{EV}. In the perspective of studying the stability of the Schwarzschild spacetimes as solutions to the massless Einstein--Vlasov system, we investigate the asymptotic properties of massless Vlasov fields on the exterior of a fixed Schwarzschild background. 

In the last decades, an intensive amount of research \cite{DR09, BS03, AB, TT11, DRSR, Aret11, Mos16, MS24} has been carried out to show integrated local energy decay estimates (ILEDs) for wave equations on the exterior of black hole backgrounds. These works are motivated by the stability problem of black hole spacetimes for the Einstein vacuum equations. ILEDs for wave equations on black holes capture key geometric properties of these spacetimes, including the well-known redshift effect and the existence of trapped null geodesics. In view of this line of research, a program of deriving ILEDs for massless Vlasov fields on black hole exteriors was initiated in \cite{ABJ,L23}. In addition, in the latter work, inverse polynomial decay of a non-degenerate energy flux has been shown by adapting the $r^p$-weighted energy method of Dafermos--Rodnianski \cite{DR10}. However, the first and high order radial derivatives of the energy-momentum tensor are not controlled in these works. This non-trivial task has been difficult to show due to the convoluted structure of the system of commuted equations for solutions of the massless Vlasov equation \eqref{vlasov_eqn_massless_intro}. Most of the derivatives in the span of $\{\partial_{r^*} , \partial_{p_{r^*}}\}$ grow exponentially on bounded regions of $\mathcal{S}$. We deal with this issue by using well-chosen vector fields that arise from the expansion and contraction properties of the null geodesic flow. In this article, we develop a commutation approach for massless Vlasov fields on a Schwarzschild black hole background. As a result, we control the radial derivative of the energy-momentum tensor $\mathrm{T}[f]$.

Recent works on decay estimates for wave equations on black hole backgrounds by Holzegel--Kauffman \cite{GK23}, and Mavrogiannis \cite{Mavr23}, have shown novel \emph{relatively non-degenerate integrated (local) energy decay estimates}. In these integrated estimates, the bulk term is not degenerate with respect to its boundary term.\footnote{The integrated energy estimate shown in \cite{GK23} has an unavoidable degeneracy at spatial infinity $r=+\infty$. This is why these integrated estimates are \textit{local}.} A key ingredient in these estimates is a commutation vector field $G \in \Gamma (T\mathcal{S})$, given below in \eqref{defG}, such that one can control, up to lower order terms, non-degenerately the spacetime integral of the first order derivatives of $G$ acting on the scalar field. Specifically, \cite{GK23} and \cite{Mavr23} studied linear wave equations on the exterior of Schwarzschild and Schwarzschild--de Sitter spacetimes, respectively. See the extensions of these works to the full subextremal Kerr family \cite{HKa23} and, conditionally to mode stability, to subextremal Kerr--de Sitter spacetimes\footnote{In particular, the results in \cite{mavrogiannis_2022} hold for very slowly rotating Kerr--de Sitter black holes.} \cite{mavrogiannis_2022}.

In this article, we develop a commutation approach to show decay for massless Vlasov fields on the exterior of Schwarzschild spacetimes. We make use of a weight function $\varphi_-\colon \mathcal{P}\to \R$ that captures the concentration phenomenon of future-trapped null geodesics in $\mathcal{P}$. The projection $V_+\in \Gamma(T\mathcal{P})$ of its associated symplectic gradient, parallel to a direction transverse to $T\mathcal{P}$, turns out to enjoy good commutation properties with the massless Vlasov operator. By identifying a well-chosen modification $\pmb{V}_+^{\mathrm{mod}}$, we construct a suitable $W_{x,p}^{1,1}$ weighted norm for which any finite energy solution to the massless Vlasov equation verifies an ILED without relative degeneration. Together with the $r^p$-weighted energy method, we establish time decay for the energy flux. This approach is compatible with previous works \cite{Mavq,DHRTq,DHRTq2} used to study quasi-linear wave equations on black hole spacetimes. By compatible we mean that
\begin{itemize}
\item As for the solutions to wave equations in these works, our method to derive decay estimates for the solutions to the massless Vlasov equation relies on the boundedness of an energy flux, an ILED (which as in \cite{Mavq} is relatively non-degenerate), and a hierarchy of $r^p$-weighted energy estimates (as in \cite{DHRTq, DHRTq2}).
\item The vector field used as a multiplier and capturing the redshift effect, used to obtain improved estimates near the future event horizon in this paper, is the same than the one used in the analysis of wave equations on Schwarzschild spacetimes.
\item Let $\Pi \colon \mathcal{P} \to \mathcal{S}$ be the canonical projection from the null-shell to the base manifold $\mathcal{S}$ and $\mathbf{C}_{\mathrm{Vlasov}}$ be the set of the commutators used in this paper to commute the massless Vlasov equation. Then the span of $\mathrm{d} \Pi \big( \mathbf{C}_{\mathrm{Vlasov}} \big) $ equals $T\mathcal{S}$, which allows to estimate all the derivatives of $\mathrm{T}[f]$ (the source term in the Einstein equations for \eqref{EV}).
\item Moreover, \cite{Mavq} makes crucial use of energy norms in which the top-order derivatives include the vector field $G$, introduced below in \eqref{defG}, which is colinear to $\mathrm{d} \Pi (V_+^{\mathrm{mod}})$.
\end{itemize}

Let us briefly mention some advantages of non-relatively degenerate ILEDs. First, they allow one to derive exponential decay for the energy flux through purely energy-based methods (in contrast, a relatively degenerate ILED yields only superpolynomial decay \cite{DRSdS, L23}). Moreover, polynomial decay follows easily from such estimates together with the $r^p$-weighted energy method (see Section \ref{Sec7}). By contrast, in \cite{L23}, a significantly more involved hierarchy of weighted energy norms had to be introduced in order to obtain the $\tau^{-p}$ decay of the energy flux. These features should lead to much less technical nonlinear analysis (see already \cite{Mavq}). Finally, non-relatively degenerate ILEDs also make it possible to obtain sharp decay estimates.

\subsection{State of the art on decay for massless Vlasov fields on black hole exteriors}

Decay estimates for solutions to the massless Vlasov equation on black hole exteriors have been obtained using either \emph{weighted energy methods} or \emph{phase space volume estimates}. Let us discuss these contributions:

\subsubsection{Based on weighted energy methods} First, Andersson-Blue-Joudioux \cite{ABJ} studied the solutions to the massless Vlasov equation on very slowly rotating Kerr black holes. They proved boundedness for a weighted energy flux and an ILED, which degenerates at the future event horizon $\H$ and on the trapped set. Note that they do not obtain pointwise decay estimates for $\mathrm{T}[f]$. Later, the first author \cite{L23} adapted the $r^p$-weighted energy method \cite{DR10} for massless Vlasov fields on Schwarzschild to derive inverse polynomial decay of non-degenerate energy fluxes, through suitable spacelike-null foliations, and also for momentum averages. For this, an ILED (\underline{with} a relative degeneration) allowing for a control of the solutions near $\H$ and the photon sphere was obtained. We point out that $\partial_r \mathrm{T}[f]$ was not controlled and that the method used to derive pointwise decay estimates relied crucially on the exact spherical symmetry of Schwarzschild black holes. We note that for strongly decaying data, this article shows superpolynomial decay of $\mathrm{T}[f]$.

\subsubsection{Based on phase space volume estimates} Recently, the second author \cite{VR23} proved that Schwarzschild is non-linearly stable for the massless Einstein--Vlasov system in spherical symmetry. This result uses the normal hyperbolicity of the trapped set for the null geodesic flow. When restricted to Schwarzschild, this hyperbolic behaviour is captured by weight functions $\varphi_{\pm}\colon \mathcal{P}\to \R$ that define the subsets $\{\varphi_{+}=0\}$ and $\{\varphi_{-}=0\}$, where future-trapped and past-trapped orbits are contained, respectively. Assuming compact support for the initial data, $\mathrm{T}[f]$ was proved to decay exponentially in the bounded region of the exterior of the black hole. Derivatives of the distribution function were estimated by studying Jacobi fields on $\mathcal{P}$ in terms of the Sasaki metric. Suitable derivatives of the energy-momentum tensor were also shown to decay.

Around the same time, Weissenbacher \cite{W23} proved decay estimates for the components of $\mathrm{T}[f]$ for solutions to the massless Vlasov equation on Reissner--Nordström spacetimes for compactly supported initial data. In the extremal case, if the initial support of $f$ or its first order derivatives intersect a neighborhood of the future event horizon, the transversal derivatives of certain components of $\mathrm{T}[f]$ do not decay along $\H$.

\subsection{Vlasov equations on backgrounds with hyperbolic flows}

A well-known property of Schwarzschild is the existence of \emph{trapped null geodesics} at $\{r=3M\}$. These orbits are \emph{unstable}, in the sense that the trapped set, where trapped geodesics are located, is \emph{normally hyperbolic}.\footnote{As proved in \cite[Section~$2$]{WZ11}, the trapped set of Schwarzschild spacetime is \textit{eventually absolutely $r$-normally hyperbolic} for all $r$.} We discuss previous works on decay for Vlasov fields on backgrounds with hyperbolic flows.

In the exterior of Schwarzschild, the radial geodesic equation defines a non-linear ODE system with a hyperbolic fixed point corresponding to the trapped set. The linearisation of this non-linear ODE with respect to its fixed point, defines a Hamiltonian flow on $\R_x\times \R_p$ with Hamiltonian $H(x,p)=\frac{1}{2}p^2-\frac{1}{2}x^2$ (up to normalisation). Unstable trapping holds for this flow, and with this motivation, small data global existence for the Vlasov--Poisson system with the external potential $\frac{-|x|^2}{2}$ was shown in \cite{VV24} in dimension $n\geq 2$. Modified vector field techniques were used in dimension 2. See \cite{BVV23} for the modified scattering dynamics of this system in dimension 2. 

Recently \cite{VVel23} studied decay estimates for Vlasov fields on non-trapping asymptotically hyperbolic manifolds. By a commuting vector field approach, the spatial density and its first order derivatives were shown to decay exponentially, when the initial data is away from the zero velocity set. For general compactly supported data on hyperbolic space, the decay is only inverse polynomial.

\subsection{The main results}

Before stating our main results, we briefly introduce some notations. See Section \ref{Sec2} for a more thorough presentation. We will use the notation $A \lesssim B$ to specify that there exists a constant $C>0$, depending only on the black hole mass $M$ and possibly other parameters quantifying the decay rate of the initial data, such that $A \leq CB$. If $A \lesssim B $ and $B \lesssim A$ we will write $A \sim B$.

\subsubsection{The coordinate system induced by the tortoise coordinates}

In the strict exterior $\mathring{\mathcal{S}}$ of the Schwarzschild black hole $(\mathcal{S},g_M)$ of mass $M>0$, the metric reads
$$g_M=-\Omega^2(r)\dr t^2+\frac{1}{\Omega^2(r)}\dr r^2+r^2 \dr  \theta^2+r^2\sin^2(\theta)  \dr \phi^2, \qquad \quad \Omega^2(r)\coloneqq1-\frac{2M}{r}, $$ 
where $(t,r,\theta , \phi ) \in \R \times (2M,+ \infty ) \times (0,\pi) \times (0,2 \pi)$. In this article, we will mostly work with the tortoise coordinate system $(t,r^*,\theta , \phi )$, where $\frac{\dr r^*}{\dr r}=\Omega^{-2}$ and $r^*(3M)=0$. It induces the coordinate system $(t,r^*,\theta,\phi,p_t,p_{r^*},p_{\theta}, p_{\phi})$ on the cotangent bundle $T^* \mathring{\mathcal{S}}$. Then, the null-shell $\mathcal{P}$ can be parametrised by $(t,r^*,\theta,\phi,p_{r^*},p_{\theta}, p_{\phi})$, where $p_t$ is given by
 \begin{equation*}
p_t = - \bigg| \,p_{r^*}^2 + \Omega^2 \frac{|\slashed{p}|^2}{r^2}  \bigg|^{\frac{1}{2}}, \qquad \qquad |\slashed{p}| \coloneqq \bigg|p_{\theta}^2+\frac{ p_{\phi}^2}{\sin^2 (\theta ) }\bigg|^{\frac{1}{2}}.
\end{equation*}
The expression of the geodesic spray $\T_{g_M}$ is written below in \eqref{eq:defT}.

\subsubsection{Weight functions and commutation vector fields}

We now define the quantities that we will use to set up the main weighted energy norm of the article. 

Let $N\in \Gamma(T\mathcal{S})$ be the future-directed timelike vector field
$$ N \coloneqq   \partial_t+  \chi_N (r)  \frac{r^2}{M^2\Omega^2} (\partial_t-\partial_{r^*}),$$
where $\chi_N \in C^\infty (\R)$ is a cutoff function such that $\chi_N (r)=1$ for $r \leq 2.5M$ and $\chi_N (r) =0$ for $r \geq 2.7M$. We define the \emph{weight function associated to} $N$, by the contraction 
$$p_N\coloneqq p(N)=p_t+\chi_N (r) r^2 \, \frac{p_t-p_{r^*}}{M^2\Omega^2} .$$
Note that contrary to $|p_t|$, $|p_N|$ controls all the components of $p$ near $\H$ (see Section \ref{Rqintrored}). Let $G\in \Gamma(T\mathcal{S})$ be the spacelike vector field 
\begin{equation}\label{defG}
G\coloneqq \frac{r}{\Omega}\Big(1+\frac{6M}{r}\Big)^{\frac{1}{2}}\Big(1-\frac{3M}{r}\Big)\partial_t + \frac{r }{\Omega} \partial_{r^*} .
\end{equation}
We define the \emph{weight function $\varphi_- (x,p)$ associated to} $G$ by
\begin{equation}\label{eq:defvarphiminusintro}
\varphi_-(x,p)\coloneqq p(G)
\end{equation} 
as well as the rescaled weight $\pmb{\varphi}_-\coloneqq\Omega^{-1}\varphi_-$. To the authors knowledge, the weight \(\varphi_-\) (with a different normalisation) was introduced for the first time in the study of trapping on Kerr black holes in \cite{D15}. The vector field \(G\) instead, appeared first in the study of decay for the wave equation on Schwarzschild in \cite{GK23}.

Next, we define the vector fields that we will use to commute the Vlasov equation \eqref{vlasov_eqn_massless_intro}. They arise as projections on $T\mathcal{P}$ of the symplectic gradients of suitable weight functions (see Section \ref{Sec5} for more details).

Let $\T_{|\slashed{p}|} \in \Gamma(TT^*\mathcal{S})$ be the symplectic gradient of weight function $|\slashed{p}|$. The vector field $\T_{|\slashed{p}|}$ is tangent to $T\mathcal{P}$ and it can be expressed in terms of the complete lifts of the rotational vector fields. 

Let $V_+\in \Gamma(T\mathcal{P})$ be the projection on $T \mathcal{P}$ parallel to $\partial_{p_t}$ of the symplectic gradient of $\varphi_-$. Let us also consider the vector fields $\pmb{V}_+ , \pmb{V}_+^{\mathrm{mod}} \in \Gamma(T \mathcal{P})$, obtained as rescalings and modifications of $V_+$, given by  
$$\pmb{V}_{\! +}\coloneqq \Omega^{-1} V_+, \qquad \qquad \pmb{V}_+^{\mathrm{mod}} \coloneqq \pmb{V}_{\! +}+\Phi r \pmb{\varphi}_- \partial_{p_{r^*}},$$
where $\Phi(x,p)$ is determined by a suitable transport equation along the geodesic flow (see Definition \ref{DefVminusmod}). The modified vector field $\pmb{V}_+^{\mathrm{mod}}$ plays a central role in the energy estimates we perform. 

\subsubsection{The energy fluxes}

We will study the solutions to \eqref{vlasov_eqn_massless_intro} using a spherically symmetric spacelike-null foliation $(\Sigma_\tau)_{\tau \geq 0}$ crossing $\H$ and terminating at future null infinity $\mathcal{I}^+$. More precisely, for a fixed constant $R_0 >3M$ and a well-chosen $u_0 \in \R$,
$$ \Sigma_\tau \coloneqq \big\{ t^* = \tau , \, r \leq R_0 \big\} \sqcup \big\{ t-r^*=\tau+u_0, \, r > R_0  \big\}, \qquad \qquad t^*=t+2M\log (r-2M).$$
In this framework, we pose a sufficiently regular initial data on $\pi^{-1}(\Sigma_0) \subset \mathcal{P}$, where $\pi\colon \mathcal{P}\to \mathcal{S}$ is the canonical projection. For any distribution function $g : \mathcal{P} \to \R$ and all $\tau \geq 0$, we define the norm
$$ \mathbb{E} [g](\tau) \coloneqq \int_{\pi^{-1}(\Sigma_\tau)}  \big| p_{n_{\Sigma_\tau}} g \big|   \dr \mu_{\pi^{-1}(\Sigma_\tau)},$$ 
with respect to the induced volume form $\dr \mu_{\pi^{-1}(\Sigma_\tau)}$ on $\pi^{-1}(\Sigma_\tau)$, where $n_{\Sigma_\tau}$ is the corresponding normal to $\Sigma_\tau$ and $p_{n_{\Sigma_\tau}}=p(n_{\Sigma_\tau})$ its contraction with $p$. We will study the evolution in time of massless Vlasov fields by using the $W^{1,1}_{x,p}$ weighted energy flux
\begin{align*}
 \mathcal{E}[g]& \coloneqq \mathbb{E} \big[  p_N g \big]+\mathbb{E} \big[  \pmb{\varphi}_- \partial_t g \big] +\mathbb{E} \big[  \pmb{\varphi}_- \T_{|\slashed{p}|} g \big] + \mathbb{E} \big[ p_N \, \pmb{V}_+^{  \mathrm{mod}} g \big] \\
 & \quad \, +\mathbb{E} \big[ |p_N|^{\frac{3}{4}} |r^{-1} \pmb{\varphi}_-|^{\frac{1}{4}}  \pmb{V}_{\! +} g \big]  +\mathbb{E} \big[\Omega^{\frac{1}{2}}|p_t|^{\frac{3}{4}} |r^{-1} \pmb{\varphi}_-|^{\frac{5}{4}} \partial_{p_{r^*}} g \big] .
\end{align*}

\subsubsection{Statement of the main results}

We are now ready to state the main result of this article and its consequences.

\begin{theorem}[ILED without relative degeneration]\label{thm_main_intro}
Let $f $ be a sufficiently regular solution to the massless Vlasov equation \eqref{vlasov_eqn_massless_intro} such that $\mathcal{E}[f](0) <+\infty$. Then, there exists $C>0$, depending only on $M$, such that
$$ \sup_{\tau \geq 0} \,  \mathcal{E}[f](\tau ) +\int_{\tau=0}^{+\infty} \mathcal{E}\big[r^{-1} \log^{-2}(2+r)f \big](\tau ) \dr \tau \leq C \mathcal{E}[f](0 ).$$
\end{theorem}

As a corollary of Theorem \ref{thm_main_intro}, we show inverse polynomial decay of the energy norm $\mathcal{E}[f]$ by using the $r^p$-weighted energy method. For this, we use a suitable hierarchy $\mathcal{E}^p[f]$ of weighted energy fluxes with $p\in \mathbb{N}$. See Section \ref{subsect_decay_fluxes} for the precise form of $\mathcal{E}^p[f]$. We also prove pointwise decay for the components of $\mathrm{T}[f]$.

\begin{corollary}[Inverse polynomial decay of the energy-flux]\label{Corintro}
Let $p\in \mathbb{N}$ and $f $ be a solution to the massless Vlasov equation such that $\mathcal{E}^p[f](0) < +\infty$. Then, there holds
\begin{equation}
 \forall \, \tau \geq 0, \qquad \mathcal{E}[f](\tau) \lesssim \frac{1}{\langle \tau \rangle^{p}} \mathcal{E}^{p}[f](0),
\end{equation}
and 
\begin{equation}
 \forall \, \tau \geq 0, \qquad \sup _{\Sigma_\tau}  \, r^2 \mathrm{T}_{NN}\big[f\big]  \lesssim  \frac{1}{\langle \tau \rangle^{p}} \mathcal{E}^{p}_{4} [f](0),
\end{equation}
where $\mathcal{E}^{p}_{4} [f](0)$ is an appropriate fourth-order energy of the initial data. 
\end{corollary}
\begin{remark}
As in \cite{L23,VR23, W23}, we also obtained improved decay estimates for the good null components of $\mathrm{T}[f]$. See Corollary \ref{Cordecay} for more information.
\end{remark}
\begin{remark}
By working with the foliation $(\{\overline{t}=\tau\})_{\tau \geq 0}$, where $(\overline{t},r,\theta , \phi)$ are the hyperboloidal coordinates defined in Section \ref{Subsec75}, we could obtain pointwise decay estimates by merely assuming the finiteness of a third order energy norm.
\end{remark}

Additionally, we show exponential decay of the energy norm $\mathcal{E}[f]$. For this, we use a suitable exponentially $r$-weighted energy flux $\mathcal{E}^b_{\mathrm{exp}}[f]$, which is introduced in Section \ref{subsect_decay_fluxes} and where $b>0$ is related to the base of the exponential weight. 

\begin{corollary}[Exponential decay of the energy-flux]\label{Corintro2}
Let $b>0$ and $f $ be a solution to the massless Vlasov equation such that $\mathcal{E}^b_{\mathrm{exp}}[f](0) <+\infty$. Then, there exists $b_0 \in(0,b]$, depending on $M$, such that
$$ \forall \, \tau \geq 0, \qquad \mathcal{E}[f](\tau) \lesssim e^{-b_0 \tau} \mathcal{E}_{\mathrm{exp}}^b[f](0),$$
and 
\begin{equation}
 \forall \, \tau \geq 0, \qquad \sup _{\Sigma_\tau} \, r^2 \mathrm{T}_{NN}\big[f\big]  \lesssim e^{-b_0 \tau}\mathcal{E}^{b}_{\mathrm{exp},4} [f](0),
\end{equation}
in terms of an appropriate fourth-order energy $\mathcal{E}^{b}_{\mathrm{exp},4} [f](0)$ of the initial data. 
\end{corollary}

Theorem \ref{thm_main_intro} is the first ILED without relative degeneration proved for massless Vlasov fields on the exterior of black hole spacetimes. In particular, it improves the results of \cite{ABJ, L23} and allows to derive exponential decay through energy methods. The control of $\pmb{V}_+^{\mathrm{mod}} f$ allows to control $\partial_r \mathrm{T}[f]$, which improves the results obtained in \cite{ABJ, L23, W23}. Finally, Corollaries \ref{Corintro} and \ref{Corintro2} are the first sharp pointwise decay estimates obtained for \eqref{vlasov_eqn_massless_intro}.

The vector field approach we develop is also suitable to upgrade the previous results to higher order statements. However, for the sake of clarity we do not pursue this here.

\subsection{Difficulties of the problem}
The main goal of this paper consists in developing a commutation approach which allows us to address the next two problems:
\begin{enumerate}[label = (\alph*)]
\item Prove an ILED without relative degeneration for finite energy solutions $f$ to the massless Vlasov equation \eqref{vlasov_eqn_massless_intro}. One can easily prove (see for instance Proposition \ref{ProremainderILED}) an inequality of the form
\begin{equation}\label{ILEDintro}
 \forall \, \tau_2 \geq \tau_1 \geq 0, \qquad  \mathfrak{E}[ f](\tau_2)+\int_{\tau = \tau_1}^{\tau_2}  \mathfrak{E} \bigg[ \frac{(r-3M)^2}{r^4}  f\bigg](\tau) \dr \tau \lesssim \mathfrak{E}[ f](\tau_1), \tag{deg-ILED}
 \end{equation}
where $\mathfrak{E}[f]$ is a non-degenerate energy flux such as $\mathbb{E}[p_Nf]$. There are two degeneracies in the flux of the bulk term compared to the flux in the RHS. One at infinity, which is not problematic since the decay estimates will be obtained by using the $r^p$-weighted energy method that exploits the asymptotic flatness of Schwarzschild spacetime.\footnote{ This degeneracy can be removed for exponentially decaying initial data. See Proposition \ref{Proexpodec}.} The other degeneracy occurs at the photon sphere $\{r=3M\}$, which is related to the difficulties associated to the existence of trapped null geodesics, and has to be removed in order to derive decay estimates.

We would like to prove an ILED without relative degeneration, that is, an estimate such as
\begin{equation}\label{ILEDintro2}
 \forall \, \tau_2 \geq \tau_1 \geq 0, \qquad  \mathfrak{F}[ f](\tau_2)+\int_{\tau = \tau_1}^{\tau_2}  \mathfrak{F} \bigg[ \frac{1}{r^2}  f\bigg](\tau) \dr \tau \lesssim \mathfrak{F}[ f](\tau_1), \tag{ILED-wrd}
 \end{equation}
for an energy flux $\mathfrak{F}[f]$ controlling $\mathbb{E}[p_Nf]$. Contrary to \eqref{ILEDintro}, apart from the factor vanishing at infinity as $r^{-2}$, the flux in the time integral in the LHS is equivalent to the one in the RHS. In conjunction with a well-adapted $r^p$-weighted energy method, \eqref{ILEDintro2} will allow us to derive decay for the energy flux $\mathfrak{F}[f]$.
\item Prove boundedness for all the derivatives of the energy-momentum tensor $\mathrm{T}[f]$. In order to control 
$$ \mathcal{L}_{\partial_t}\mathrm{T}[f], \qquad \qquad \mathcal{L}_{\mathbf{\Omega}_i}\mathrm{T}[f],$$
where $(\mathbf{\Omega}_i)_{1 \leq i \leq 3}$ are the generators of $SO_3(\R)$, one can exploit the time and spherical symmetries of Schwarzschild spacetime. However, for the radial derivative, the difficulty is that contrary to the case of Minkowski spacetime, we do not have $\partial_r = \frac{x^i}{r} \partial_{x^i}$ with $\partial_{x^i}$ being generators of space translation symmetries letting the Minkowski metric invariant. 

In the case of the wave equation $\Box_g \psi =0$ on Schwarzschild or Kerr spacetimes, where the analogous problem consists in proving boundedness for $\psi$ in $\dot{H}^2(\Sigma_\tau)$, one can proceed as follows. First, one can control $\partial_t \psi$ in $\dot{H}^1(\Sigma_\tau)$ by exploiting $[\Box_g , \partial_t ]=0$. Then, the other second order derivatives can be controlled through a manifestation of the redshift effect and elliptic estimates. In the case of the massless Vlasov equation, one could hope to take advantage of the hypoelliptic properties of the transport operator $\T_g$ through the so-called \textit{averaging lemmata} \cite{Ag84, averaginglemmas}. Nevertheless, the optimal regularity $\dot{H}^{1/2}(\Sigma_\tau)$ one could expect from them, together with bounds on $f$ and $\partial_t f$, is far from $W^{1,1}(\Sigma_\tau)$. It is even too weak to get pointwise estimates through Sobolev embeddings. 

\begin{remark}
In Schwarzschild one can simply use the symmetries of spacetime and control $\partial_r^2 \psi$ through $\Box_g \psi =0$. The analogous strategy for the massless Vlasov equation \eqref{vlasov_eqn_massless_intro} would merely allow us to control $\partial_{r^*} \mathrm{T}\big[\frac{p_{r^*}}{p_t}f \big]$.
\end{remark}
\end{enumerate}

It turns out that we will solve these two problems together. Concerning the issue of the degeneracy at the photon sphere for \eqref{ILEDintro}, we recall:
\begin{itemize}
\item From \cite[Proposition 3.12]{L23}, the estimate \eqref{ILEDintro2} does not hold for $\mathfrak{F}[f]=\mathbb{E}[p_Nf]$.
\item One can prove an estimate similar to \eqref{ILEDintro2} with $\mathfrak{F}[f]=\mathbb{E}[p_Nf]$, but where the RHS is replaced by a flux $\overline{\mathfrak{F}}[f]$, where $\mathfrak{F} \lesssim \overline{\mathfrak{F}}$ even though $\mathfrak{F}$ and $\overline{\mathfrak{F}}$ are \underline{not} equivalent. See \cite[Proposition 3.2]{L23} or Proposition \ref{ProILED} below.
\end{itemize} 

\subsection{Ideas and strategy of the proof}\label{Subsecideas}

We now present the key ideas of the proof of Theorem \ref{thm_main_intro}, which address the problems $(a)$ and $(b)$. The main part of the analysis will consist in performing an estimate similar to \eqref{ILEDintro2} for the $\dot{W}^{1,1}_{x,p}$ weighted semi-norm $\mathcal{E}[f]-\mathbb{E}[p_Nf]$, which will allow us to solve the problem $(b)$. Then, by a local Sobolev inequality and \eqref{ILEDintro}, we will derive the estimate \eqref{ILEDintro2} for $\mathfrak{F}[f]= \mathcal{E}[f]$.


\subsubsection{Step $0$: Linearisation around the photon sphere.} \label{linearised_ph_sph} The linearisation of $\T_g$ in the $(t,r^*,p_{r^*})$ variables around the fixed point $(r-3M=0,\,p_{r^*}=0)$ yields, up to a constant\footnote{The actual linearisation yields the operator $\partial_t +p \partial_x+\frac{1}{27M^2}x \partial_p$, where the additional constant is the square of the Lyapunov exponent $\frac{1}{3\sqrt{3}M}$ associated to the unstable trapping at the photon sphere. }, the operator
$$ \T_{\mathrm{lin}}\coloneqq\partial_t +p \partial_x+x \partial_p, \qquad \quad (t,x,p) \in \R_t \times \R_x \times \R_p,$$
which has been studied in \cite{VV24}. The characteristics of $\T_{\mathrm{lin}}$ either escape to infinity or converge to $(x=0,\, p=0)$. Moreover, the weights $s\coloneqq e^t(p-x)$ and $u\coloneqq e^{-t}(p+x)$ are conserved along the flow of $\T_{\mathrm{lin}}$ and are closely related to the trapping phenomenon. In particular, $p-x$ decays exponentially along the flow of $\T_{\mathrm{lin}}$ and measures how close is the trajectory starting at $(x,p)$ at $t=0$ to be past-trapped. 
The derivatives of a solution $g$ to $\T_{\mathrm{lin}}( g)=0$ can be well-understood by the analysis of the symplectic gradients $U$ and $S$ of $s$ and $u$, respectively. Alternatively, one can study the symplectic gradients $W_+$ and $W_-$ of $\psi_-\coloneqq p-x$ and $\psi_+\coloneqq p+x$, respectively. These vector fields are given by
$$ U \coloneqq e^tW_+, \qquad S \coloneqq e^{-t}W_-, \qquad W_+=\partial_x+ \partial_p, \qquad W_- = \partial_x - \partial_p$$
and satisfy the commutation properties
$$ \big[\T_{\mathrm{lin}}, U \big]= \big[\T_{\mathrm{lin}} , S \big]=0, \qquad \qquad \big[ \T_{\mathrm{lin}},W_+ \big]=-W_+, \qquad \qquad \big[ \T_{\mathrm{lin}},W_- \big]=W_-.$$
One can see in particular that:
\begin{itemize}
\item $Ug$ is conserved along the flow of $\T_{\mathrm{lin}}$. Equivalently, $W_+g$ decays as $e^{-t}$.
\item By integration by parts $\big|\partial_x \int_p g \dr p\big| \lesssim e^{-2t}$, whereas $\big| \int_p g \dr p\big| \lesssim e^{-t}$ by a suitable change of variables.
\item For any $Z_{a,b}=a\partial_x+b\partial_p$ with $a \neq b$, then $Z_{a,b}g$ grows exponentially along the flow of $\T_{\mathrm{lin}}$.
\end{itemize}
Thus, despite that derivatives of macroscopic quantities associated to $g$ decay exponentially, there is merely one (direction of) microscopic derivative $(a\partial_x+b\partial_p)g$ which decays exponentially, whereas the others \emph{grow exponentially}. In the case of Schwarzschild, this suggests that one way to derive decay estimates for $\partial_r \mathrm{T}[f]$ is to identify a good vector field on the null-shell which is analogous to either $U$ or $W_+$.

\subsubsection{Step $1$: Finding a good weight function.} The weight $\varphi_-$ introduced in \eqref{eq:defvarphiminusintro}, which measures how close is a null geodesic to be past-trapped, is exponentially decaying along the flow on the bounded region of spacetime. Specifically, it satisfies
\begin{equation}\label{equa:phiintro}
 \T_g (\varphi_-)=-\frac{|p_t|}{r^{\frac{1}{2}} |r+6M|^{\frac{1}{2}}\Omega^2} \varphi_-.
\end{equation}
By integrating this relation, one obtains a conserved quantity analogous to $s$. Let, for $(x,p) \in \mathcal{P}$, $\tau \mapsto \Phi_{\tau}(x,p)$ be the flow map of $\T_g$ parametrised by $t^*$ with data $\Phi_{t^*(x)}(x,p)=(x,p)$. Then
$$ \mathfrak{s}(x,p)\coloneqq \varphi_- (x,p) e^{\alpha(x,p)}, \qquad \alpha(x,p)\coloneqq  \int_{s=0}^{t^*(x)} \frac{|p_t|}{r^{\frac{1}{2}} |r+6M|^{\frac{1}{2}}\Omega^2} \circ \Phi_{s}(x,p) \dr s ,  $$
is conserved along the null geodesic flow. From \eqref{equa:phiintro}, we will show that the \emph{degenerate} norm $\mathbb{E}[\varphi_- f]$ verifies an \eqref{ILEDintro2}. See Proposition \ref{ProILED000} for more information. 

\subsubsection{Step $2$: Identifying a vector field commuting well with $\T_g$.}\label{Subsubsecstart} By a standard abuse of terminology, we denote by $\T_w$ the symplectic gradient of a function $w(x,p)$ defined on an open subset of $T^* \mathcal{S}$. It turns out, in view of the analysis in Appendix \ref{SecCotan}, that $\varphi_-$ can be extended to the set
$$\mathcal{P}_{\mathrm{causal}}\coloneqq \big\{ (x,p) \in T^* \mathcal{S} \, : \, g^{-1}_x(p,p) \leq 0, \, \,p_t <0 \big\},$$ containing all causal geodesics. Its extension, denoted $\varphi_-^{m(x,p)}(x,p)$, where $m(x,p)=|-g_x^{-1}(p,p)|^{1/2}$ and $\varphi_-^0=\varphi_-$, verifies an identity similar to \eqref{equa:phiintro}. Consequently, one can also extend $\alpha$ and $\mathfrak{s}$, so that this last quantity is conserved along causal geodesics. Therefore, its symplectic gradient $\T_{\mathfrak{s}}$ verifies
$$ \big[ \T_g, \T_{\mathfrak{s}} \big]=0, \qquad  \quad \T_{\mathfrak{s}}=\T_{\varphi_-^m} e^{\alpha(x,p)}+ \varphi_-^m (x,p) e^{\alpha(x,p)} \T_\alpha \qquad \quad \text{on\,\, $\accentset{\circ}{\mathcal{P}}_{\mathrm{causal}}$.}$$
Compared to the linearised case previously discussed, $\mathfrak{s}$ and $\T_{\mathfrak{s}}$ should be compared with $s$ and $U$. In view of the regularity of the geodesic flow in Schwarzschild and these last quantities, one can smoothly extend $\T_{\mathfrak{s}}$, $\T_{\varphi_-^m}$, and $\T_\alpha$ up to $\mathcal{P}$. We can then obtain boundedness for $\T_{\mathfrak{s}} f$, where $f$ is a solution to the massless Vlasov equation, that is $f$ is defined on $\mathcal{P}$ and $\T_g (f)=0$. However, the function $\alpha(x,p)$ is not given as an explicit function of the coordinates, which makes the expression of $\T_\alpha$ particularly involved. Deriving estimates for $\partial_r \mathrm{T}[f]$ from $\T_{\mathfrak{s}} f$ is then a non-trivial task.

With this in mind, we decided to look for a vector field analogous to $W_+$ instead. As $e^{-\alpha (x,p)}\T_{\mathfrak{s}}$ suffers from the same issue as $\T_{\mathfrak{s}}$, a good candidate for that could be $\T_{\varphi_-^m}$ which has the advantage of having the simple expression
\begin{equation}\label{eq:expTphi}
 \T_{\varphi_-^m} = \frac{|r+6M|^{\frac{1}{2}}}{r^{\frac{1}{2}}\Omega}(r-3M) \partial_t+\frac{r}{\Omega} \partial_{r^*}-\bigg( \frac{r-3M}{r\Omega}p_{r^*}+\frac{r^{\frac{1}{2}}}{|r+6M|^{\frac{1}{2}} \Omega}p_t \bigg) \partial_{p_{r^*}}+\pmb{x}_g \T_g \quad \text{on \,$\mathcal{P}$,}
 \end{equation}
in the coordinate system $(t,r^*,\theta, \phi, p_t,p_{r^*},p_\theta , p_\phi )$ on $T^* \mathcal{S}$ and where $\pmb{x}_g(r,p_{r^*},p_t)$ is a smooth function. We stress that the last term, which is tangent to $\mathcal{P}$, is irrelevant for the study of solutions to the massless Vlasov equation since $\pmb{x}_g \T_g(f)=0$. Nevertheless, as $\T_{\varphi_-^m} \notin T \mathcal{P}$, this vector field cannot be used to study the solutions to the massless Vlasov equation. Instead, we will make a non-canonical choice, that we will justify in Remark \ref{Rqproj}, and work with 
$$V_+ \coloneqq \mathrm{Proj}_{\parallel \partial_{p_t}}\big( \T_{\varphi_-^m} \big)-\pmb{x}_g \T_g, \qquad \qquad \mathrm{Proj}_{\parallel \partial_{p_t}} :  \bigcup_{(x,p) \in \mathcal{P}} T_{(x,p)} T^* \mathcal{S} \to T \mathcal{P}, $$
where $\mathrm{Proj}_{\parallel \partial_{p_t}} $ is the projection parallel to $\partial_{p_t}$. Then, in the coordinate system $(t,r^*,\theta, \phi,p_{r^*},p_\theta , p_\phi )$ on $\mathcal{P}$, the vector field $V_+$ has formally the same expression as $\T_{\varphi_-^m}-\pmb{x}_g \T_g$, allowing for the (schematic) estimate 
\begin{equation}\label{eq:Gderiv}
 \big| G \mathrm{T}[f] \big| \lesssim \big|\mathrm{T}[V_+ f]\big|+\big|\mathrm{T}[f] \big|.
 \end{equation}
Moreover, $V_+$ has good commutation properties with $\T_g$,
\begin{align}\label{eq:ComVplusinto}
[\T_g,V_+] &=  - \frac{|p_t|}{r^{\frac{1}{2}}|r+6M|^{\frac{1}{2}}\Omega^2} V_+- \frac{r+3M}{r^{\frac{3}{2}}|r+6M|^{\frac{3}{2}}}|p_t|\varphi_- \partial_{p_{r^*}}  + 2\Omega \T_g.
\end{align}
However, compared with $[\T_{\mathrm{lin}},W_+]=-W_+$, there are two additional error terms. The last one vanishes when applied to $f$, but the second one is in fact problematic in the perspective of proving an ILED without relative degeneration. Indeed, by analogy with the linearised problem, we schematically have around the photon sphere
\begin{equation}\label{eq:heuri} 
2\varphi_- p_t \partial_{p_{r^*}} f \sim \psi_- W_+g- \psi_- W_-g .
\end{equation}
Recall then that $\psi_-$ and $W_+ g$ decay as $e^{-t}$ but $W_- g$ grows as $e^t$. Hence, $ \psi_- W_+g- \psi_- W_-g $ is of order $1$ and, roughly speaking, its asymptotic behaviour is then comparable with the one of $g$. It turns out that one may indeed prove boundedness for $\mathfrak{G}[f]=\mathbb{E} [ r^{-1}\varphi_- p_t \partial_{p_{r^*}} f ]+\mathbb{E} [ p_t V_+ f ]+\mathbb{E}[\varphi_- \partial_t f]$ but, exactly as for $\mathbb{E}[p_N f]$, one cannot obtain \eqref{ILEDintro2} with $\mathfrak{F}[f]=\mathfrak{G}[f]$. 

\begin{remark}\label{Rqproj}
As explained in Appendix \ref{Subsecsympgrad}, if $W$ is a vector field transverse to $T \mathcal{P}$, we have 
$$V_+-\mathrm{Proj}_{\parallel W}(\T_{\varphi_-^m})-\pmb{x}_g \T_g=\varphi_- Z,$$ where $Z$ is a smooth vector field on $\{ r>2M\}$. If $W$ is chosen to be $\partial_{p_{t^*}}$, the derivatives according to $p_{t^*}$ in the coordinate system of $T^* \mathcal{S}$ induced by $(t^*,r,\theta,\phi)$, it turns out that $Z$ is collinear to $\partial_{p_{r^*}}$. One can check using Proposition \ref{Propartialprstar} that for any vector field of the form $\widetilde{V}_+=V_++\varphi_- z\big(r,\frac{p_{r^*}}{p_t} \big) \partial_{p_{r^*}}$, where $z$ is a smooth function, the commutator $[\T_g,\widetilde{V}_+]$ verifies the next properties:
\begin{itemize}
\item The first term is the same as in the RHS of \eqref{eq:ComVplusinto}, with $V_+$ replaced by $\widetilde{V}_+$.
\item There is an error term of the form $b\big(r,\frac{p_{r^*}}{p_t} \big)|p_t|\varphi_-  \partial_{p_{r^*}}$, similar to the second one in the RHS of \eqref{eq:ComVplusinto}.
\item There are extra error terms, proportional to $\varphi_- \partial_t$ and $\varphi_- \widetilde{V}_+$, which can then be handled.
\end{itemize}
We could then use $\widetilde{V}_+$ instead of $V_+$, but choosing $V_+$ provides a more convenient commutation formula.
\end{remark}

\subsubsection{Step $3$. Finding an improved commutator.} We recall the two next properties:
\begin{itemize}
\item $V_++\varphi_- \T_\alpha$ satisfies good commutation properties with $\T_g$ but the coefficients of $\T_\alpha$ are not explicit functions of the coordinates. Even worse, they are not even invariant by the flow of $\partial_t$. Thus, the components of $V_++\varphi_- \T_\alpha$ in the basis $\{\partial_t,\partial_{r^*},\dots , \partial_{p_\phi}\}$ cannot be easily estimated and could even vanish, so that estimating $\partial_r\mathrm{T}[f]$ through $V_++\varphi_- \T_\alpha$ looks complicated. 
\item $V_+$ is given through the simple expression \eqref{eq:expTphi} but $[\T_g,V_+]$ contains an error term which cannot be controlled sufficiently well for our purposes. 
\end{itemize}
However, we can observe that the difference $\varphi_- \T_\alpha$ of these two vector fields has the weight $\varphi_-$, which motivates us to look for a correction $V_{\, +}^{\mathrm{mod}}$ of $V_+$ verifying the following properties:
 \begin{itemize}
 \item $V_{\, +}^{\mathrm{mod}} = V_++\Lambda  \varphi_- \partial_{p_{r^*}}$, where $\Lambda$ is a suitable function that can be controlled in $W^{1,\infty}_{x,p}$. In such a way, by performing integration by parts, one can prove that \eqref{eq:Gderiv} holds true as well when $V_+$ is formally replaced by $V_{\, +}^{\mathrm{mod}}$.
 \item There exists a vector field $Z$ such that
 $$ \color{white} \square \qquad \color{black} [\T_g,V_{\, +}^{\mathrm{mod}}]= - \frac{|p_t|}{r^{\frac{1}{2}}|r+6M|^{\frac{1}{2}}\Omega^2} V_{\, +}^{\mathrm{mod}}+\frac{\Omega}{ r^{\frac{7}{4}}}|\varphi_-|^{\frac{1}{4}}|p_N|^{\frac{3}{4}} Z+2\Omega \T_g, \qquad \quad \sup_{\tau \geq 0} \mathbb{E}[p_N Zf](\tau) < +\infty .$$
Then, thanks to the factor $|\varphi_-|^{1/4}$, we will be able to prove an ILED without relative degeneration  \color{white}f\color{black} for\footnote{Note that $\varphi_-$ grows as $r$, which explains why the derivative is in fact weighted by $r^{-1}\varphi_-$.} $|r^{-1}\varphi_-|^{1/4}|p_N|^{3/4} Z f$, and then for the flux $\mathbb{E}[p_N \,V_{\, +}^{\mathrm{mod}} f]$.
 \end{itemize}
 \begin{remark}
 The analysis in \textit{Step 2} suggests that the asymptotic behavior of both $V_+f$ and $\varphi_- \partial_{p_{r^*}}f$ are comparable to the one of $g$, which satisfies $\mathbb{X}_{\mathrm{lin}}(g)=0$. In order to construct an improved commutator from $V_+$, it is then consistent to look for a correction $V_++C$, where the asymptotic behaviour of $Cf$ is also comparable to $g$. As the bad error term in $[\mathbb{X}_g,V_+]$ is proportional to $\varphi_- \partial_{p_{r^*}}$, one can expect to look for an improved commutator of the form $V_+^{\mathrm{mod}}= V_++\Lambda \varphi_- \partial_{p_{r^*}}$.
 \end{remark}
 
 Finally, in order to fully capture the redshift effect and obtain a stronger control of the solutions near the future event horizon $\mathcal{H}^+$, we will work with a rescaled version of these vector fields, $\pmb{V}_{\! +}\coloneqq\Omega^{-1}V_+$ and $\pmb{V}_{\, +}^{ \mathrm{mod}}\coloneqq \Omega^{-1}V_{\, +}^{ \mathrm{mod}}$, which are related to the weight $\pmb{\varphi}_-=\Omega^{-1} \varphi_-$ discussed in Section \ref{Rqintrored}.

\subsubsection{Step 4: Proving an ILED without relative degeneration}

Once $\pmb{V}_+^{ \mathrm{mod}}$ is identified, the main difficulty is to control appropriately $|r^{-1}\varphi_-|^{1/4}|p_N|^{3/4} Z f $. For this, the goal consists in proving boundedness and an ILED without relative degeneration for
$$ \pmb{\varphi}_- \partial_t f , \qquad \qquad \mathrm{D}_{p_{r^*}}f:= \Omega^{\frac{1}{2}}|p_t|^{\frac{3}{4}} |r^{-1} \pmb{\varphi}_-|^{\frac{5}{4}} \partial_{p_{r^*}} f  , \qquad \qquad \mathrm{D}_{\pmb{V}_{\! +}}f:=|p_N|^{\frac{3}{4}} |r^{-1} \pmb{\varphi}_-|^{\frac{1}{4}}  \pmb{V}_{\! +} f .$$
The first derivative can be easily treated since $\T_g(\partial_t f)=0$. Moreover, we can expect to be able to control well the last two derivatives since they are weighted by a sufficiently large power of $\varphi_-$. However, it turns out that the system of the commuted equations is \textit{not triangular}. More precisely, we have
\begin{align*}
\T_g \big(r^{\frac{1}{4}} \mathrm{D}_{\pmb{V}_{\!+}} f \big) & \leq - \mathbf{b}_{\mathrm{good}}^1(x,p) r^{\frac{1}{4}}\mathrm{D}_{\pmb{V}_{\!+}} f + |p_N|r^{-3} \big|r^{\frac{9}{4}} \mathrm{D}_{p_{r^*}}f \big| , \\
 \T_g \big( r^{\frac{9}{4}} \mathrm{D}_{p_{r^*}}f \big)& \leq -\mathbf{b}_{\mathrm{good}}^2(x,p) r^{\frac{9}{4}} \mathrm{D}_{p_{r^*}}f +C|p_t/p_N|^{\frac{3}{4}}|\pmb{\varphi}_-| \big|r^{\frac{1}{4}} \mathrm{D}_{\pmb{V}_{\!+}} f \big| +\textrm{good term},
 \end{align*}
with $C>0$ and where the good term is a combination of $\pmb{\varphi}_- \partial_t f$ and $\pmb{\varphi}_- \T_{|\slashed{p}|}f$. Moreover, $|\pmb{\varphi}_-| \lesssim r|p_N|$ and $\mathbf{b}_{\mathrm{good}}^1(x,p) , \,\mathbf{b}_{\mathrm{good}}^2(x,p) \gtrsim r^{-1}|p_N|$. One can observe that there are two problems.
\begin{itemize}
\item First, $r^{\frac{1}{4}} \mathrm{D}_{\pmb{V}_{\!+}} f$ and $r^{\frac{9}{4}} \mathrm{D}_{p_{r^*}}f$ carry respectively extra $r^{\frac{1}{4}}$ and $r^{\frac{9}{4}}$ factors compared to the quantities we would like to control. 
\item The functions $\mathbf{b}_{\mathrm{good}}^1$ and $\mathbf{b}_{\mathrm{good}}^2$ are not strong enough to absorb the bad error terms. There does not exist $c>0$ such that $ c\mathbf{b}_{\mathrm{good}}^1(x,p) > Cr^{-2 }|p_t/p_N|^{\frac{3}{4}}|\pmb{\varphi}_-|$ and $ \mathbf{b}_{\mathrm{good}}^2(x,p) >  c r^{-1} |p_N|$.
\end{itemize}

We deal with the first issue by splitting the null-shell $\mathcal{P}$ into the three regions $\{ p_{r^*} \leq 0, \; r \geq R \}$, $\{ p_{r^*} \geq 0, \; r \geq R \}$ and $\{ r \leq R \}$, where $R > 3M$ is sufficiently large\footnote{ Considering this splitting of the null-shell is the only way we found to deal with the two problems mentioned above. Let us however mention that $\mathrm{D}_{\pmb{V}_{\!+}}f$ and $ \mathrm{D}_{p_{r^*}}f$ carry weights related to the $r^p$ energy method (see Section 1.5.6 below), so that it is natural to at least consider separatly the subregions $\{ r \leq R \}$ and $\{ r \geq R \}$.}.
\begin{enumerate}[label = (\alph*)]
\item Incoming particles in the far-away region $\{ p_{r^*} \leq 0, \; r \geq R \}$. This domain can be treated independently of the rest of the null-shell. The reason behind that is that future-directed null geodesics never enter this region. The difficulties related to trapping do not appear here and, if $R$ is chosen large enough, we are in fact on a perturbation of Minkowski. We are then able to control the derivatives of the distribution function by a method inspired by the flat case. 
\item Escaping particles in the far-away region $\{ p_{r^*} \geq 0, \; r \geq R \}$. The key observation is that $\T_g(r^{-1})= -p_{r^*} r^{-2} \leq 0$, so we can rescale the vector fields by generating error terms with a good sign.
\item The bounded region $\{ r \leq R \}$, where the first problem is irrelevant.
\end{enumerate}
We address the second issue by taking advantage of the weight function $\zeta(x,p)$ introduced in Definition \ref{Defz}, which verifies $\zeta \sim 1$ and $\T_g (\zeta) \lesssim -r^{-2}|p_{r^*}|-r^{-3}|r-3M| |p_t|$. This weight can be exploited in order to deal with the bounded region, when away from a neighborhood of $\mathcal{H}^+$ and the photon sphere. In fact, it allows to handle the error term $C|p_t/p_N|^{\frac{3}{4}}|\pmb{\varphi}_-| \big|r^{\frac{1}{4}} \mathrm{D}_{\pmb{V}_{\!+}} f \big|$, which degenerates at $r=2M$ and at $r=3M$, on $\{ r \leq R \}$.

\subsubsection{About the redshift effect and the $r^p$-method}\label{Rqintrored}

We note that $\varphi_-$ can also be used to capture the redshift effect \cite{DR09} and the $r^p$-weighted phenomenon \cite{DR10}. This explains why, as it can be observed in \eqref{equa:phiintro}, the weight $|\varphi_-|$ is a Lyapunov function in the full exterior of Schwarzschild and not merely around $\{r=3M\}$. Let us discuss these properties:
\begin{itemize}
\item The redshift effect is the phenomenon that allows for massless Vlasov fields supported in a neighborhood of $\mathcal{H}^+$ to decay exponentially. In any region $\{r \geq 2M+\delta\}$ with $\delta>0$, the vector field $\partial_t$ is strictly timelike, and $|p_t|$ controls any component of $p$. However, since $\partial_t$ is null on $\mathcal{H}^+$, this property does not hold up to $\H$. This is why we control $p_N f$, where
$$ p_N \sim p_t+\frac{p_u}{\Omega^2} \qquad\qquad \text{near \,$r=2M$,}$$
instead of $p_t f$. At a first glance, this could seem problematic since the stationarity of Schwarzschild merely allows for a direct control of $p_t f$. However, this can be shown as $\frac{p_u}{\Omega^2}$ eventually decays exponentially along any null geodesic that enters the black hole region. 
\item The $r^p$-weighted phenomenon is related to the fact that most of the particles are escaping in the region $r > 3M$ after a large time. In fact, $r^p |p_v|$ eventually decays for all $0 \leq p <2$, along any null geodesic terminating at future null infinity $\mathcal{I}^+$.
\end{itemize}

The redshift effect and the $r^p$-weighted phenomenon are fully captured by the identities
\begin{align*}
\T_g \bigg(  \frac{r^2}{\Omega^2} |p_u| \bigg)  = 2\frac{r-3M}{\Omega^4}|p_u|^2  , \qquad \qquad  \T_g \bigg(  \frac{r^2}{\Omega^2} |p_v| \bigg)  = -2\frac{r-3M}{\Omega^4}|p_v|^2  ,
 \end{align*}
respectively. Moreover, the function $\varphi_-$ is related to these two weight functions through
\begin{align*}
\Big|\varphi_-- \frac{2r}{\Omega}|p_u| \Big| \lesssim \Omega |p_t| \quad \text{for \,$r \leq 2.5M$}, \qquad \qquad \Big| \varphi_-+\frac{2r}{\Omega}|p_v|\Big| \lesssim \frac{|p_t|}{r} \quad \text{for \,$r \geq 4M$}.
 \end{align*}
 These estimates justify why we will often work with the rescaled quantity $\pmb{\varphi}_- \coloneqq \Omega^{-1} \varphi_-$.

\subsection{The non-linear stability of the Minkowski space}

The global asymptotic stability of Minkowski spacetime for the massless Einstein-Vlasov \eqref{EV} system was first obtained by Taylor \cite{T}, when the initial distribution function is compactly supported. Under this assumption, the problem becomes a small data semi-global existence result in the wave zone. The estimates for the distribution function were performed by studying Jacobi fields on $\mathcal{P}$. Later, Bigorgne--Fajman--Joudioux--Smulevici--Thaller \cite{BFJST} proved the non-linear stability of Minkowski for \eqref{EV} without assuming any compact support for the initial data. The estimates for the distribution function here were obtained by using weighted commuting vector field techniques. The case of spherically symmetric perturbations was previously addressed by Dafermos \cite{D06}.

\subsection{The massive Vlasov equation on black hole exteriors}

The energy-momentum tensor $\mathrm{T}[f]$ for massive Vlasov fields on Schwarzschild spacetime in general \emph{does not decay}. The massive Vlasov equation admits many non-trivial finite energy stationary solutions. One can circumvent this obstruction to decay, by considering massive Vlasov fields supported in the closure of the largest domain of the mass-shell where timelike geodesics either cross $\mathcal{H}^+$, or escape to infinity. For this class of distribution functions, quantitative decay estimates for the energy-momentum tensor have been shown by the second author \cite{V24}. On the other hand, in the region where Vlasov fields do not decay, a phase mixing result without a rate of convergence has been proved by Rioseco-Sarbach \cite{RS20}. In the perspective of addressing the problem of quantitative phase mixing, Chaturvedi--Luk \cite{CL24} have recently shown these estimates for a linear Vlasov equation under an external Kepler potential. Even further, they have obtained in spherical symmetry a long-time nonlinear phase mixing result for the Vlasov--Poisson system under an external Kepler potential. 

In a different research line, Kehle--Unger \cite{KU24} constructed one-parameter families of smooth spherically symmetric solutions of the Einstein--Maxwell--Vlasov system, interpolating between dispersion and collapse. The solution corresponding to the threshold of gravitational collapse turns out to be an extremal black hole.

\subsection{Structure of the article}
We end with an outline of the remainder of the paper.

\begin{itemize}
\item \textbf{Section \ref{Sec2}.} We recall the framework to study massless Vlasov fields from the point of view of the initial value problem. We recall the general form of the energy flux of a distribution function.
\item \textbf{Section \ref{Sec3}.} We introduce a class of weight functions, and show their basic properties. In particular, we set the weights $\varphi_{\pm}$ that capture the expansion and concentration properties of the geodesic flow.
\item \textbf{Section \ref{Sec4}.} We show zeroth order energy boundedness and ILEDs.
\item \textbf{Section \ref{Sec5}.} We introduce a class of vector fields that are used to commute the Vlasov equation. In particular, we introduce the symplectic gradients $V_{\pm}$, and the modified vector field $\pmb{V}_+^{\mathrm{mod}}$.
\item \textbf{Section \ref{SecILEDwrd}.} We prove an ILED without relative degeneration for the first order energy norm $\mathcal{E}[f]$.
\item \textbf{Section \ref{Sec7}.} We show time decay for the first order energy-norm $\mathcal{E}[f]$, by using the $r^p$- energy method.
\item \textbf{Section \ref{Sec8}.} We prove pointwise decay estimates for the components of the energy-momentum tensor $\mathrm{T}[f]$ using Sobolev inequalities.
\item \textbf{Appendix \ref{Append}.} We show pointwise bounds for the correction term of the modified vector field $\pmb{V}_{\, +}^{\mathrm{mod}}$, including its first order derivatives. 
\item \textbf{Appendix \ref{SecCotan}.} We discuss the commutator associated to the conserved quantity arising from trapping by using the symplectic structure of the cotangent bundle.  
\end{itemize}

\subsection{Acknowledgements}
LB conducted this work within the France 2030 framework programme, the Centre Henri Lebesgue ANR-11-LABX-0020-01. RVR thanks Jacques Smulevici for several stimulating discussions. RVR also thanks Georgios Mavrogiannis for insightful discussions about \cite{Mavr23, Mavq}.

\section{Preliminaries}\label{Sec2}

\subsection{Schwarzschild spacetimes and their properties}

 The Schwarzschild black holes $(\mathcal{S},g_M)_{M>0}$ is a one-parameter family of spherically symmetric and stationary Lorentzian manifolds. From now on, we fix a mass $M>0$ and we drop the subscript $M$ of the metric, so that we write $g$ for $g_M$. . 

\subsubsection{Coordinate systems}

Let us define the Schwarzschild metric in terms of $(t^*,r,\theta,\phi)$ coordinates. We equip $\R \times \R_+^* \times \mathbb{S}^2$ with the metric
 $$g=-\Omega^2(r) |\dr t^*|^2+\frac{4M}{r} \dr t^* \dr r+\Big(1+\frac{2M}{r}\Big)\dr r^2+r^2g_{\mathbb{S}^2}, \qquad \qquad \Omega^2(r)\coloneqq1-\frac{2M}{r},$$
 where $g_{ \mathbb{S}^2} = \dr \theta^2+\sin^2(\theta ) \dr\phi^2$ is the round metric on the unit sphere $\mathbb{S}^2$. This coordinate system has the usual degeneration of the spherical coordinates $(\theta , \phi) \in (0,\pi) \times (0,2\pi)$. The \textit{exterior region} of the black hole $\mathcal{S}$ and the \textit{future event horizon} $\mathcal{H}^+$ are given by
 $$ \mathcal{S} \coloneqq \R \times [2M,\infty) \times \mathbb{S}^2, \qquad \qquad \mathcal{H}^+ \coloneqq \R \times \{2M\} \times \mathbb{S}^2.$$
The subset $\{ r < 2M \}$, which corresponds to the interior of the black hole, will not be studied in this article.

The vector field $\partial^*_{t^*}$, the derivative with respect to $t^*$ in these coordinates, is timelike on $\mathring{\mathcal{S}}$ and null on $\mathcal{H}^+$. The future event horizon is then a null hypersurface normal to $\partial^*_{t^*}$. We fix the time orientation in the Lorentzian manifold $(\R \times \R_+^* \times \mathbb{S}^2,g)$ by requiring $\partial^*_{t^*}$ to be future-directed in $\mathcal{S}$.

We will mostly work in the tortoise coordinate system $(t,r^*,\theta , \phi) \in \R \times \R \times (0,\pi)\times (0,2\pi)$, which covers $\mathring{\mathcal{S}} = \{ r >2M \}$ and where $(t,r^*)$ are given by
$$t \coloneqq t^*-2M\log (r-2M), \qquad \qquad r^*(r)\coloneqq r-3M+2M\log (r-2M)-2M \log(M),$$
so that $\dr r^*= \Omega^{-2} \dr r$ and $r^*(3M)=0$. The metric then reads
\begin{equation}\label{eq:defmetric}
g=-\Omega^2(r)\mathrm{d} t^2+\Omega^2(r)|\mathrm{d} r^*|^2+r^2 g_{\mathbb{S}^2}.
\end{equation}
We will denote by $\partial_t$, $\partial_{r^*}$, $\partial_\theta$ and $\partial_\phi$ the derivatives with respect to $t$, $r^*$, $\theta$ and $\phi$ in this coordinate system and we consistently define $\partial_r \coloneqq \Omega^{-2} \partial_{r^*}$. 

Finally, we also introduce the outgoing and ingoing Eddington--Finkelstein null coordinates. They are respectively given by
$$u=t-r^*, \qquad \quad v=t+r^*,$$
so that we define
$$  \partial_u \coloneqq \frac{1}{2} \big( \partial_t - \partial_{r^*} \big), \qquad \quad  \partial_v \coloneqq \frac{1}{2} \big( \partial_t + \partial_{r^*} \big) .$$
As one can check, by exploiting for instance the coordinate system $(t^*,r,\theta,\phi)$, $\Omega^{-2} \partial_u$ can be extended as a smooth vector field up to $\mathcal{H}^+$. In the double null coordinate system $(u,v,\theta,\phi)$, the metric takes the form $g=-\Omega^2 \dr u \dr v +r^2 g_{\mathbb{S}^2}$. We note that the level sets of $u$ and $v$ are respectively outgoing and ingoing null cones. In view of the relation between $t^*$ and $v$, we can abusively view the future event horizon $\mathcal{H}^+$ as $\{(u=+\infty, v,\theta,\phi)\}$. Finally, \textit{future null infinity} $\mathcal{I}^{+}$, which can be rigorously defined in a conformal compactification of Schwarzschild spacetime, can be viewed as the limit of $\{v=v_0\}$ as $v_0 \to +\infty$. It then abusively corresponds to $\{(u,v=+\infty,\theta,\phi)\}$.

We refer to \cite[Chapter~$13$]{O83} for more information about the family of Schwarzschild black holes.

\subsubsection{Killing fields of Schwarzschild spacetime}

As it can be checked in \eqref{eq:defmetric}, $\partial_t$ is a timelike Killing vector field for $r>2M$ and $(\mathring{\mathcal{S}},g)$ is a static spacetime. Schwarzschild black hole is also spherically symmetric since
\begin{equation}\label{killing_sph_symm}
\mathbf{\Omega}_1=-\sin \phi\,\partial_{\theta}-\cos \phi\cot\theta\,\partial_{\phi}, \qquad   \mathbf{\Omega}_2=\cos \phi \,\partial_{\theta}-\sin\phi \cot \theta \,\partial_{\phi},\qquad \mathbf{\Omega}_3=\partial_{\phi},
\end{equation}
are Killing vector fields generating an action by isometries of $SO_3(\R)$.

\subsubsection{The timelike vector field $N$} We will use a timelike future-directed vector field $N \in \Gamma(T\mathcal{S})$ to control sufficiently well massless Vlasov fields. It is defined by
$$ N \coloneqq   \partial_t+  \chi_N (r)  \frac{2r^2}{M^2\Omega^2} \partial_u,$$
where $\chi_N \in C^\infty (\R)$ is a cutoff function such that $\chi_N (r)=1$ for all $r \leq 2.5M$ and $\chi_N (r) =0$ for all $r \geq 2.7M$. In particular, $N=\partial_t$ on $\{ r \geq 2.7M \}$ and, contrary to $\partial_t$, $N$ is timelike on $\{ r \geq 2M \}$ since 
$$g(N,N)=-\Omega^2(r)-2M^{-2}r^2 \chi_N (r).$$

\subsubsection{The hypersurfaces $\Sigma_\tau$}\label{Subsubsecvolume}

Let us define the spacelike-null foliation $(\Sigma_\tau)_{\tau \geq 0}$ that we will use in order to study the solutions to $\T_g (f)=0$. We set the constants 
$$R_0>3M, \qquad \quad t_0 \coloneqq -2M \log (R_0-M), \qquad \quad  u_0\coloneqq t_0-r^*(R_0).$$ 

\begin{definition}
Let, for all $\tau\in \R_+$, $\Sigma_\tau$ be the hypersurface 
$$ \Sigma_\tau \coloneqq \{ t^* = \tau , \, r \leq R_0 \} \sqcup \{ u=\tau+u_0, \, r > R_0  \}.$$
\end{definition}
\begin{remark}
We have $\Sigma_\tau=\varphi_{\tau}(\Sigma_0)$ for all $\tau\geq 0$, where $\tau \mapsto \varphi_\tau$ is the flow-map generated by the Killing field $\partial_t$. Moreover, $\Sigma_\tau$ is composed by a piece of the spacelike hypersurface $\{ t^* = \tau \}$ and the piece of the outgoing null cone $\{u = \tau +u_0 \}$ located in the future of $\{t^* = \tau \}$. They intersect at the sphere $\{(t_0+\tau,r^*(R_0))\}\times \mathbb{S}^2$.
\end{remark}

Let us also introduce the following notation. For $0\leq \tau_1<\tau_2\leq +\infty$, we denote the sets 
$$ \mathcal{R}^{\tau_2}_{\tau_1} \coloneqq \bigcup_{\tau_1\leq\tau\leq \tau_2} \Sigma_\tau,\qquad \qquad \mathcal{R} \coloneqq \mathcal{R}^{+\infty}_0.$$ 

These subsets are represented in the following piece of the Penrose diagram of the exterior of Schwarzschild spacetime.
\begin{figure}[!ht]
\begin{center}
\begin{tikzpicture}[scale=1]
\fill[color=gray!14] (-7.13,0.87) .. controls  (-6.13,-0.13) and   (-2.84,0).. (-2.82,0)--(-1.41,1.41)--(-4,4);

\node (II)   at (-4,0)   {};
\path  
  (II) +(90:4)  coordinate[label=90:$i^+$]  (IItop)
       +(180:4) coordinate (IIleft)
       +(0:4)   coordinate[label=0:$i^0$]  (IIright)
       ;
\draw[thick] (IIleft) -- 
        node[midway, above=1pt, sloped] {$\mathcal{H}^+$ \, \small{($u=+\infty$)}}
      (IItop);
 \draw[dashed]     (IItop)--    node[midway, above=1pt, sloped] {$\mathcal{I}^+$ \, \small{($v=+\infty$)}}     (IIright);

\fill[color=gray!34] (-6.37,1.63) .. controls (-5.87,1.13)  and (-4,1.7)..  (-3,1.4)--(-2.2,2.2) -- (-2.88,2.88)-- (-3.29,2.47) .. controls (-4,3.1) and (-5.29,2.11)  ..  (-5.59,2.41)  ;
\draw[dashed] (0,0)--(-1,-1); 
\draw[thick] (-8,0)--(-7,-1); 
\draw[very thin,blue]  (-7.13,0.87) .. controls  (-6.13,-0.13) and   (-2.84,0).. node[below]{\small{$\Sigma_0$}}   (-2.82,0)--(-1.41,1.41)    ;
\draw (-6.37,1.63) .. controls (-5.87,1.13)  and (-4,1.7).. node[below]{\small{$\Sigma_{\tau}$}}  (-3,1.4)--(-2.2,2.2) ;
\draw (-5.59,2.41) .. controls (-5.29,2.11)   and (-4,3.1) .. node[above]{\small{$\Sigma_{\tau_2}$}} (-3.29,2.47)--(-2.88,2.88) ;

\draw (-4,2) node{\small{$\mathcal{R}_{\tau}^{\tau_2}$}};
\end{tikzpicture}
\caption{The foliation $(\Sigma_{\tau})_{\tau \geq 0}$. }\label{fig2}
\end{center}
\end{figure}
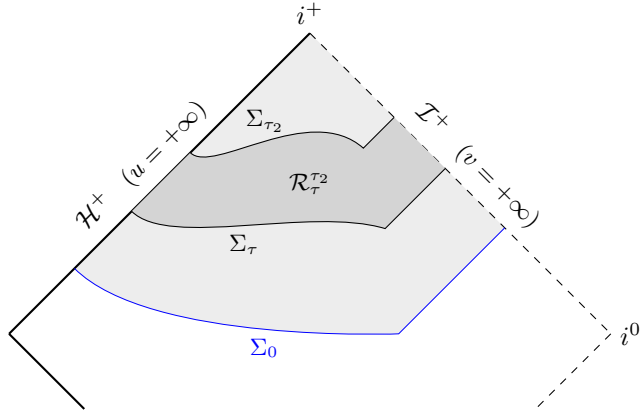

\subsubsection{Volume forms}

The volume form in the spacetime region $\mathcal{R}$ is given by 
$$\dr \mu_{\mathcal{R}}=\Omega^2(r)r^2 \dr t\wedge \dr r^* \wedge \dr\mu_{\mathbb{S}^2}=r^2 \dr t\wedge \dr r \wedge \dr\mu_{\mathbb{S}^2},$$
where $\dr \mu_{\mathbb{S}^2} =  \sin(\theta) \dr \theta \wedge \dr \phi$ is the volume form on the unit round sphere $\mathbb{S}^2$. We define the future-directed normal $n_{\Sigma_\tau}$ to $\Sigma_\tau$ as
\begin{equation*}
 n_{\Sigma_\tau}\big|_{\{r<R_0\}}\coloneqq \Big(1+\frac{2M}{r} \Big)^{-\frac{1}{2}}\partial_v+\Big(1+\frac{2M}{r}\Big)^{\frac{1}{2}}\frac{1}{\Omega^2(r)}\partial_u,\qquad \quad n_{\Sigma_\tau}\big|_{\{r\geq R_0\}}\coloneqq\partial_v.
\end{equation*} 
Note that $n_{\Sigma_\tau}$ is unitary on $\{ r < R_0 \}$. Pulling back the spacetime volume form into the hypersurfaces $\Sigma_\tau$, in accordance with the choice of $n_{\Sigma_\tau}$ as normal vector field, we obtain the volume forms
 \begin{equation*}
 \dr \mu_{\Sigma_\tau}\big|_{\{r<R_0\}}=\Big(1+\frac{2M}{r}\Big)^{\frac{1}{2}} \,r^2 \dr r \wedge \dr\mu_{\mathbb{S}^2},\qquad \quad \dr \mu_{\Sigma_\tau}\big|_{\{r\geq R_0\}}=r^2 \dr v \wedge \dr\mu_{\mathbb{S}^2},
 \end{equation*}

The null hypersurface $\H$ is equipped with the volume form and the normal vector $$ \dr \mu_{\mathcal{H}^+}=r^2 \dr v \wedge \dr\mu_{\mathbb{S}^2},\qquad\quad n_{\H}=\partial_v.$$ On the other hand, even though future null infinity $\mathcal{I}^+$ is not part of Schwarzschild spacetime, we can view this set as the level set $\{v=+\infty\}$ equipped with the volume form and the normal vector $$\dr \mu_{\mathcal{I}^+}= \dr u \wedge \dr\mu_{\mathbb{S}^2}\qquad\quad n_{\mathcal{I}^+}=\partial_u.$$
Finally, we will often use the following consequence of the coarea formula. For all measurable function $h_+ \colon \mathcal{R} \to \R_+$ and all $\tau_2 \geq \tau_1 \geq 0$,
\begin{equation}\label{eq:coarea}
 \int_{\mathcal{R}_{\tau_1}^{\tau_2}} h_+ \dr \mu_{\mathcal{R}} \sim \int_{\tau = \tau_1}^{\tau_2} \int_{\Sigma_\tau} h_+ \dr \mu_{\Sigma_\tau} \dr \tau \tag{coarea}.
 \end{equation}

\subsection{The cotangent bundle}\label{Subseccotan} We recall geometric properties of the cotangent bundle
$$T^* \mathcal{S} \coloneqq \{ (x,p) \, |  \, x \in \mathcal{S}, \, p \in T^*_x\mathcal{S} \}$$
which will be useful for our study of massless Vlasov fields. 

\begin{definition}\label{Defcoord}
Let $\mathscr{C}=(x^{\mu})_{0 \leq \mu \leq 3}$ be a local coordinate system on $\mathcal{S}$ defined on an open subset $\mathcal{U}$.
\begin{itemize}
\item The coordinates $(p_\mu)_{0 \leq \mu \leq 3}$ on the fibers $T_x^* \mathcal{S}$ are referred to as the \emph{conjugate momenta} to $\mathscr{C}$ if $$ \forall \, (x,p) \in \mathcal{U} \times T_x^* \mathcal{S}, \qquad \qquad p=p_\mu \dr x^\mu .$$
\item Then, $(x^\mu,p_\mu)_{0 \leq \mu \leq 3}$ is a local coordinate system on $T^* \mathcal{S}$, referred to as the one \emph{induced} by $\mathscr{C}$. We say that such coordinates are \emph{canonical}.
\end{itemize}
\end{definition}

Our analysis relies on the canonical symplectic structure of the cotangent bundle. The associated antisymmetric bilinear form arises as the exterior differential of the Poincaré $1$-form, which reads $p_\mu \dr x^\mu$ in a canonical coordinate system.

\begin{definition}
Let $(x^\mu,p_\mu)_{0 \leq \mu \leq 3}$ be a canonical coordinate system on $T^* \mathcal{S}$. Then,
$$\Omega_s \coloneqq \dr p_\mu \wedge \dr x^\mu $$
defines a symplectic form. We associate to any function $w (x,p)$, defined on an open subset of $T^* \mathcal{S}$, its \emph{symplectic gradient} $\T_w$. It is uniquely determined by
$$ \dr w = \Omega_s \big(\cdot , \T_w \big),$$
so that
$$ \T_w = \partial_{p_\mu}(w) \partial_{x^\mu}-\partial_{x^\mu}(w) \partial_{p_\mu}.$$
\end{definition}

In particular, the geodesic spray $\T_g$ is the symplectic gradient of the one-particle Hamiltonian
 $$ H:(x,p) \mapsto \frac{1}{2}g^{-1}_x(p,p),$$
 that is $\T_g := \T_H$. The images by the canonical projection $(x,p) \mapsto x$ of the integral curves of $\T_g$,
$$
\frac{\dr x^{\alpha}}{\dr s}=p_{\beta}g^{\alpha\beta}, \qquad  \frac{\dr p_{\alpha}}{\dr s}=-\frac{1}{2}\frac{\partial g^{\beta\gamma}}{\partial x^{\alpha}}p_{\beta}p_{ \gamma},
$$ 
are the geodesics of $(\mathcal{S},g)$. In the coordinate system $(t,r^*,\theta,\phi,p_t,p_{r^*},p_\theta,p_\phi)$, induced by $(t,r^*,\theta,\phi)$, the geodesic spray is given by
\begin{equation*}
\mathbb{X}_{g}=-\frac{p_t}{\Omega^2} \partial_t + \frac{p_{r^*}}{\Omega^2} \partial_{r^*}+\frac{p_{\theta}}{r^2} \partial_{\theta} + \frac{p_{\phi}}{r^2 \sin^2 ( \theta )} \partial_{\phi} -\bigg(\frac{M}{r^2\Omega^2}(|p_t|^2-|p_{r^*}|^2)-\frac{\Omega^2}{r^3}|\slashed{p}|^2 \bigg) \partial_{p_{r^*} }+\frac{\cot (\theta )}{r^2 \sin^2 (\theta)}  |p_{\phi}|^2 \partial_{p_{\theta} } .
 \end{equation*} 
The following result is central to our approach to derive decay estimates for massless Vlasov fields and their derivatives.

\begin{proposition}
Let $\mathcal{O} \subset T^* \mathcal{S}$ be an open set and $\mathfrak{c} : \mathcal{O} \to \R$ be a conserved quantity along the geodesic flow. Then, $\T_g (\mathfrak{c})=0$ and $[\T_g, \T_{\mathfrak{c}}]=0$. 
\end{proposition}
\begin{proof}
The relation $\T_g (\mathfrak{c})=0$ exactly means that $\mathfrak{c}$ is conserved along the geodesics of Schwarzschild. For the second identity, one can check that, for any functions $\mathfrak{a}$ and $\mathfrak{b}$ defined on $\mathcal{O}$,
$$ [\T_{\mathfrak{a}}, \T_{\mathfrak{b}}]= \T_{\Omega_s(\T_{\mathfrak{a}},\T_{\mathfrak{b}})}, \qquad \qquad \T_a (\mathfrak{b})=\dr \mathfrak{b} \cdot \T_{\mathfrak{a}} = \Omega_s(\T_{\mathfrak{a}},\T_{\mathfrak{b}}).$$
Then, we apply these identites to $\mathfrak{a}=H$ and $\mathfrak{b}=\mathfrak{c}$.
\end{proof}

The metric of Schwarzschild spacetime gives rise to a natural metric on $T^* \mathcal{S}$, the Sasaki metric (see for instance \cite[Section~$II.D.$]{AGS}). It induces the following volume forms $\dr \mu_{T^* \mathcal{S}}$ and $\dr \mu_{T^*_x \mathcal{S}}$ on $T^* \mathcal{S}$ and $T^*_x \mathcal{S}$, written in a canonical coordinate system $(x^\mu,p_\mu)_{0 \leq \mu \leq 3}$,
\begin{align*}
 \dr \mu_{T^*_x \mathcal{S}} & \coloneqq \big|\det g^{-1}_x \big|^{\frac{1}{2}} \dr p^0 \wedge \dr p^1 \wedge \dr p^2 \wedge \dr p^3, \\
\dr \mu_{T^* \mathcal{S}}& \coloneqq - \dr x^0 \wedge \dr x^1 \wedge \dr x^2 \wedge \dr x^3 \wedge  \dr p^0 \wedge \dr p^1 \wedge \dr p^2 \wedge \dr p^3 .
 \end{align*}
We note that $\dr \mu_{T^* \mathcal{S}}= - \dr \mu_{\mathcal{R}} \wedge \dr \mu_{T^*_x \mathcal{S}}$. Moreover, this volume form is invariant with respect to the geodesic flow, that is $\mathcal{L}_{\T_g}\dr \mu_{T^* \mathcal{S}}=0$. 

\subsection{The null-shell}\label{Subsecnullshell}

Since we study in this article ensembles of massless particles, we will in fact be interested in a subset of the cotangent bundle, where future-directed null geodesics lie.

\begin{definition}
The null-shell is the subset $\mathcal{P}\subset T^*\mathcal{S}$ given by
 $$\mathcal{P} \coloneqq \Big\{(x,p)\in T^*\mathcal{S} \; \big| \; g^{-1}_x(p,p)=0, \; p(\partial_t ) <0 \Big\}.$$
The \emph{null-shell relation} $g^{-1}_x(p,p)=0$ implies that the covector $p$ is null. The condition $p(\partial_t)<0$ implies that $p$ is future-directed. The projection map 
$$ \pi \colon \mathcal{P} \to \mathcal{S}, \qquad \qquad \pi (x,p)=x,$$
will be used throughout this paper. We denote the fiber of $x \in \mathcal{S}$ by 
$$ \mathcal{P}_x := \pi^{-1}(x).$$
\end{definition}
\begin{remark}
The vector field $\partial_t$ is the derivative with respect to $t$ in the coordinate system $(t,r^*,\theta, \phi)$.
\end{remark}

The set $\mathcal{P}$, which is a connected component of a level set of $H$, is then a smooth seven dimensional manifold. If $(x^\mu,p_\mu)_{0 \leq \mu \leq 3}$ is a canonical coordinate system on $T^* \mathcal{S}$ such that $x^0$ is a temporal function on $\mathcal{S}$, then the null-shell relation implies that $p_0$ is a function of the other coordinates on $\mathcal{P}$. Hence, $(x^\mu,p_i)_{0 \leq \mu \leq 3, \, 1 \leq i \leq 3}$ are smooth coordinates on $\mathcal{P}$.

\begin{definition}
Let $\mathscr{C}=(x^{\mu})_{0 \leq \mu \leq 3}$ be a local coordinate system on $\mathcal{S}$ such that $x^0$ is a temporal function function. We will say that the coordinates $(x^\mu,p_i)_{0 \leq \mu \leq 3, \, 1 \leq i \leq 3}$ on $\mathcal{P}$ are induced by $\mathscr{C}$.
\end{definition}

In this article, we will mainly use the coordinate system $\widehat{\mathscr{C}} \coloneqq (t,r^*,\theta , \phi, p_{r^*},p_\theta , p_\phi)$ on $\mathcal{P}$ induced by $(t,r^*,\theta , \phi)$. Then, the null-shell relation implies
\begin{equation}\label{eq:defConsQua}
p_t = - \bigg| \,p_{r^*}^2 + \Omega^2 \frac{|\slashed{p}|^2}{r^2} \, \bigg|^{\frac{1}{2}}, \qquad \qquad |\slashed{p}| \coloneqq \bigg|p_{\theta}^2+\frac{ p_{\phi}^2}{\sin^2 \theta }\bigg|^{\frac{1}{2}}.
\end{equation}
In the coordinate system induced by the null coordinates $(u,v,\theta , \phi)$, we have
$$ p= p_u \dr u +p_v \dr v +p_\theta \dr \theta + p_\phi \dr \phi, \qquad \qquad p_u = \frac{p_t-p_{r^*}}{2}, \quad p_v = \frac{p_t+p_{r^*}}{2}, \quad 4p_up_v = \frac{\Omega^2}{r^2}|\slashed{p}|^2 .$$
We note in particular that $p_u \leq 0$ and $p_v \leq 0$. Since $\partial_t$ is not uniformly timelike in the exterior of the black hole, $|p_t|$ does not control all the components of $p$. More precisely and as it can be observed in \eqref{eq:defConsQua}, it does not control $|\slashed{p}|$ and $\Omega^{-2}|p_u|$ near $\H$. For this reason we will use
\begin{equation}\label{eq:defpN}
p_N\coloneqq p(N)= p_t+ \chi_N (r)  \frac{2r^2}{M^2\Omega^2} p_u,
\end{equation}
 which does control all the components of $p$. 

\begin{lemma}\label{LempN}
We have $p_N=p_t$ on $\{ r \geq 2.7M \}$ and $|p_N| \sim |p_t|+\frac{|p_u|}{\Omega^2} $ on $\mathcal{S}$.
\end{lemma}

We also need to describe $p_{n_{\Sigma_\tau}}$, which will naturaly appear through applications of the divergence theorem.

\begin{lemma}\label{LempSig}
We have $p_{n_{\Sigma_\tau}}=p_u$ on $\{ r > R_0 \}$ and $|p_{n_{\Sigma_\tau}}| \sim |p_N|$ on $\{ r \leq R_0 \}$.
\end{lemma}

As $2H=g_x^{-1}(p,p)$ is conserved along the geodesics, $\T_g$ is tangent to $\mathcal{P}$. The geodesic spray is given in the coordinate system $\widehat{ \mathscr{C}}$ by
\begin{equation}\label{eq:defT}
\T_g = -\frac{p_t}{\Omega^2} \partial_t + \frac{p_{r^*}}{\Omega^2} \partial_{r^*}+\frac{p_{\theta}}{r^2} \partial_{\theta} + \frac{p_{\phi}}{r^2 \sin^2  (\theta )} \partial_{\phi} + \frac{r-3M}{r^4} |\slashed{p}|^2 \partial_{p_{r^*} }+\frac{\cot ( \theta )}{r^2 \sin^2 (\theta)}  |p_{\phi}|^2 \partial_{p_{\theta} }.
\end{equation}
In particular, in this paper, we study solutions to $\T_g(f)=0$.

Since $\mathcal{P}$ is null for the Sasaki metric, one needs to be careful when defining a suitable volume form $\dr \mu_\mathcal{P}$ on the null-shell. Since $\dr H$ is a natural normal to $\mathcal{P}$ (and $\mathcal{P}_x$), the following choice is usually made.
\begin{definition}
For all $x \in \mathcal{S}$, let $\dr \mu_{\mathcal{P}_x}$ be the unique volume form on $\mathcal{P}_x$ satisfying
$$ \dr \mu_{T^*_x\mathcal{S}} =  \dr H \wedge \dr \mu_{\mathcal{P}_x}.$$
We define further the volume forms $\dr \mu_{\mathcal{P}}$ on $\mathcal{P}$ and $\dr \mu_{\pi^{-1}(\Sigma_\tau)}$ on $\pi^{-1}(\Sigma_\tau)$ as 
$$\dr \mu_{\mathcal{P}} =\dr \mu_{\mathcal{R}} \wedge \dr \mu_{\mathcal{P}_x}, \qquad \qquad \dr \mu_{\pi^{-1}(\Sigma_\tau)} = \dr \mu_{\Sigma_\tau} \wedge \dr \mu_{\mathcal{P}_x}.$$
\end{definition} 
\begin{remark}\label{RkLiouvi}
Liouville's theorem states that $\mathcal{L}_{\T_g} \dr \mu_{\mathcal{P}}=0$ (see \cite[Theorem~$2$]{AGS}).
\end{remark}
Then, in the coordinate system $\widehat{\mathscr{C}}$, we have
\begin{align*}
\dr \mu_{\mathcal{P}_x} & = \frac{\dr p_{r^*} \wedge \dr p_\theta \wedge \dr p_\phi}{r^2 \sin(\theta) |p_t|}.
\end{align*}

\subsection{Physical observables}

We are now able to define momentum averages of a sufficiently regular distribution function $f \colon \mathcal{P} \to \R$. We start by introducing the source term in the Einstein equations \eqref{EV}. The \emph{energy-momentum tensor} $\mathrm{T}[f]$ of $f$ is a symmetric $(0,2)$--tensor field on $\mathcal{S}$, given in a canonical coordinate system by 
$$
\forall \, x \in \mathcal{S}, \qquad \qquad \mathrm{T}[f]_{\mu\nu}(x) \coloneqq \int_{\mathcal{P}_x} f(x,p) p_{\mu} p_{\nu} \dr \mu_{\mathcal{P}_x}.
$$
The \emph{particle current density} $\mathrm{N}[f]$ of $f$ is the $1$-form $\mathrm{N}[f]$ given by
$$
\forall \, x \in \mathcal{S}, \qquad \qquad \mathrm{N}[f]_{\mu}(x) \coloneqq \int_{\mathcal{P}_x} f(x,p) p_{\mu} \dr \mu_{\mathcal{P}_x}.
$$
As a consequence of Liouville's theorem, one obtains \cite[Theorem~$3$]{AGS}, which in particular implies
\begin{equation}\label{eq:divN}
\nabla^{\mu}\mathrm{N}[f]_\mu =\int_{\mathcal{P}_x} \T_g(f) \dr \mu_{\mathcal{P}_x}, \qquad \qquad  \nabla^{\mu}\mathrm{T}[f]_{\mu\nu} = \int_{\mathcal{P}_x}\T_g(f)p_{\nu}\dr \mu_{\mathcal{P}_x}  .
\end{equation}
In particular, if $f$ is a solution to the massless Vlasov equation $\T_g(f)$, both $\mathrm{N}[f]$ and $\mathrm{T}[f]$ are divergence free. The relation $\nabla^\mu \mathrm{N}[f]_\mu =0$ is consistent with the conservation of the total number of particles.

\subsection{The energy space}\label{subsect_energy_space}

Let us now set the energy flux of a distribution function motivated by the divergence property of the particle current \eqref{eq:divN}.

\begin{definition}
Let $f \colon \mathcal{P} \to \R$ be a distribution function. For all $\tau \geq 0$, we define the \emph{energy flux} $\mathbb{E} [f](\tau)$ through the hypersurface $\Sigma_\tau$ by
$$ \mathbb{E} [f](\tau) \coloneqq \int_{\pi^{-1}(\Sigma_\tau)}  \big| p_{n_{\Sigma_\tau}} f \big|   \dr \mu_{\pi^{-1}(\Sigma_\tau)}.$$ 
\end{definition}
\begin{remark}
For Vlasov matter and a vector field $X$, the energy current $J_{\mu}^X[f]\coloneqq \mathrm{T}[f]_{\mu \nu} X^{\nu}$ is equal to $\mathrm{N}[fp(X)]$. We can then work with the particle current $\mathrm{N}[\cdot ]$ and the energy flux $\mathbb{E} [\cdot]$ as long as we consider quantities of the form $f\omega$, where $\omega \colon \mathcal{P}\to \R$ is a suitable weight function. 
\end{remark}

We prove an energy identity, which is a conservation law for solutions to the massless Vlasov equation \eqref{vlasov_eqn_massless_intro}.

\begin{proposition}\label{Proenergy0}
Let $f \colon \mathcal{P} \to \R$ be a distribution function such that $\mathbb{E}[f](0)<+\infty$. Then, for all $ \tau \geq 0$, we have
\begin{align*}
\mathbb{E}[f](\tau)+ \! \int_{\pi^{-1 } ( \mathcal{H}^+ \cap \{t^* \geq 0 \} )} \big| p_t  f \big| \dr \mu_{ \pi^{-1 } ( \mathcal{H}^+ )}+ \! \int_{\pi^{-1 } ( \mathcal{I}^+ \cap \{u \geq u_0 \} )} \! \big| p_u f \big| \dr \mu_{ \pi^{-1 } ( \mathcal{I}^+ )}\! = \mathbb{E}[f](0)+ \! \int_{\pi^{-1} ( \mathcal{R}_0^\tau ) } \! \T_g (|f|)   \dr \mu_{\mathcal{P}}.
\end{align*}
\end{proposition}
\begin{proof}
Recall the relation \eqref{eq:divN} as well as the identities collected in Section \ref{Subsubsecvolume} concerning the volume forms and normals of the hypersurfaces considered here. The result follows by applying the divergence theorem to $\mathrm{N}_{\mu}\big[|f| \big]$ in $ \mathcal{R}_0^\tau $ and by noting that $p_t=p_v$ on $\mathcal{H}^+$.
\end{proof}

Throughout this paper, we will rather use the next energy inequality, which is a direct consequence of the previous proposition.
\begin{proposition}\label{Proenergy}
Let $f \colon \mathcal{P} \to \R$ be a distribution function such that $\mathbb{E}[f](0)<\infty$. Then, for all $ \tau \geq 0$, 
\begin{align*}
\mathbb{E}[f](\tau) \leq \mathbb{E}[f](0)+\int_{\pi^{-1} ( \mathcal{R}_0^\tau ) }  \T_g (|f|)   \dr \mu_{\mathcal{P}}.
\end{align*}
\end{proposition}
\begin{remark}
In the upcoming applications, if $f$ is not nonnegative, we will use that $\T_g(|f|)=\T_g(f)\frac{f}{|f|}$.
\end{remark}

\section{Weight functions}\label{Sec3}

Let us introduce a well-chosen class of weight functions that will be used in the article.

\subsection{The conserved quantities}

We recall that for any null geodesic $\gamma$ and any conformal Killing vector field $K$, the quantity $g(\dot{\gamma},K)$ is conserved along $\gamma$. As a consequence, the contraction $p(K)$ of $p$ with $K$ is a solution to the massless Vlasov equation. In our case, we have the next properties.

\begin{lemma}\label{Lemconsqua}
The symmetries of Schwarzschild spacetime induce the following conserved weights along the geodesic flow:
\begin{itemize}
\item The \textit{particle energy} $-p_t$ (sometimes denoted $E$).
\item The \textit{azimuthal angular momentum} $p_\phi$ (sometimes denoted $\ell_z$).
\item The \textit{total angular momentum} $|\slashed{p}|$ (sometimes denoted $\ell$), defined as
$$|\slashed{p}| \coloneqq \bigg| p_\theta^2+ \frac{p_\phi^2}{\sin^2 ( \theta )} \bigg|^{\frac{1}{2}}= \bigg|\sum_{1 \leq i \leq 3} |p(\mathbf{\Omega}_i)|^2\bigg|^{\frac{1}{2}}.$$
\end{itemize}
In particular, we have $\T_g (p_t)=\T_g(p_\phi)=\T_g (|\slashed{p}|)=0$.
\end{lemma}
\begin{remark}
The conserved quantity $|\slashed{p}|^2$ corresponds to the Carter constant. Moreover, $|\slashed{p}|^2$ is equal to $Q^{\alpha \beta}p_\alpha p_\beta$, where $Q$ is the $(0,2)$-Killing tensor field
\begin{equation}\label{defQ}
Q\coloneqq \partial_\theta \otimes \partial_\theta+\frac{1}{\sin^2 (\theta ) } \partial_\phi \otimes \partial_\phi = \sum_{1 \leq i \leq 3} \mathbf{\Omega}_i \otimes \mathbf{\Omega}_i.
\end{equation}
In Schwarzschild spacetime, $Q$ can be decomposed as a linear combination of tensorial products of Killing vector fields as in \eqref{defQ}, contrary to the case of rotating Kerr black holes.
\end{remark}

\subsection{Redshift weight functions}\label{Subsecredshift}

Let us introduce suitable weight functions to capture quantitatively the redshift effect for solutions of the massless Vlasov equation.

\subsubsection{The weight associated to the redshift vector field}
We will capture the redshift effect near the future event horizon by exploiting the vector field $\frac{r^2}{\Omega^2}\partial_u$, which is transverse to $\H$, through the weight function
$$\frac{r^2}{\Omega^2}p_u=p\Big(\frac{r^2}{\Omega^2}\partial_u\Big).$$ 

\begin{lemma}\label{LemRS}
There holds
\begin{align*}
\T_g \bigg(  \frac{r^2}{\Omega^2} |p_u| \bigg) & = 2\frac{r-3M}{\Omega^4}|p_u|^2  .
 \end{align*}
\end{lemma}
\begin{proof}
First, recall $2|p_u|=|p_t|+p_{r^*}$ and $\T_g(p_t)=0$. Then, we compute 
\begin{align*}
 \T_g \bigg(  \frac{2r^2}{\Omega^2} |p_u| \bigg)  &  = \frac{2p_{r^*}}{\Omega^2} \partial_{r^*} \bigg(  \frac{r^2}{\Omega^2}  \bigg) |p_u| +\frac{r-3M}{r^2 \Omega^2} |\slashed{p}|^2 \partial_{p_{r^*}} (p_{r^*})  = 4\frac{r-3M}{\Omega^4}p_{r^*}|p_u| + \frac{r-3M}{r^2\Omega^2} |\slashed{p}|^2  .
  \end{align*}
We conclude the proof by using $p_{r^*}=p_v-p_u$, and the null-shell relation $\Omega^2 |\slashed{p}|^2=4r^2p_vp_u$. 
\end{proof}

\subsubsection{The auxiliary redshift weights}

In order to perform energy estimates, we introduce the weight function $\xi(x,p)$ that we will later use as a multiplier. Let $\xi \colon \mathcal{P}\to \R$ be the weight function defined by
$$ \xi(x,p)\coloneqq p_t+ \epsilon \chi_N (r)  \frac{2r^2}{\Omega^2} p_u  +\eta p_{r^*} \big[ \log^{-1}(2+r) -\log^{-1}(2+3M) \big]   , $$
where $\eta \coloneqq \frac{1}{2}\log(2+2M)$, and $\epsilon >0$ is a sufficiently small constant. This auxiliary weight satisfies an improved decay property along the geodesic flow, as the following proposition shows.  

\begin{proposition}\label{Promultiplierm}
The weight function $\xi$ satisfies $|\xi| \sim |p_N| \sim |p_t|+\frac{|p_u|}{\Omega^2}$. Moreover, we have
$$ \T_g \big( |\xi| \big) \lesssim -\frac{|p_{r^*}|^2}{r \log^2(2+r)}-\Big( 1-\frac{3M}{r} \Big)^2 \frac{|\slashed{p}|^2}{r^3} -\Big( 1-\frac{3M}{r} \Big)^2\frac{|p_u|^2}{r^2\Omega^4}  .$$
\end{proposition}
\begin{proof}
The first estimates follow from $\big|\eta p_{r^*} \big[ \log^{-1}(2+r) -\log^{-1}(2+3M) \big]\big| \leq \frac{1}{2}|p_t|$ and Lemma \ref{LempN}. For the second estimate, we use the properties below:
\begin{itemize}
\item $\T_g(p_t)=0$.
\item According to Lemma \ref{LemRS}, we have
\begin{align*}
 \T_g \Big(  \epsilon \chi_N (r)  \frac{2r^2}{\Omega^2} |p_u| \Big) & = 4\epsilon \chi_N(r)\frac{r-3M}{\Omega^4}|p_u|^2 +\epsilon p_{r^*} \chi'_N(r)  \frac{2r^2}{\Omega^2} |p_u| .
 \end{align*}
 Hence, by support considerations,
 $$  \T_g \Big(  \epsilon \chi_N (r)  \frac{2r^2}{\Omega^2} |p_u| \Big) + 2M \epsilon  \frac{|p_u|^2}{\Omega^4} \mathds{1}_{r \leq 2.5M} \lesssim \epsilon |p_t|^2 \, \mathds{1}_{2.5M \leq r \leq 2.7M}  .$$
\item Moreover, we have
\begin{align}
\nonumber \T_g \Big( \eta p_{r^*} \big[ \log^{-1}&(2+r) -\log^{-1}(2+3M) \big] \Big)\\ 
& =-\frac{ \eta |p_{r^*}|^2}{(2+r) \log^2(2+r)}+\eta \frac{r-3M}{r}\big[ \log^{-1}(2+r) -\log^{-1}(2+3M) \big] \frac{|\slashed{p}|^2}{r^3} \nonumber\\
 & \lesssim -\frac{ \eta |p_{r^*}|^2}{r \log^2(2+r)}- \eta \Big( 1-\frac{3M}{r} \Big)^2 \frac{|\slashed{p}|^2}{r^3}. \label{eq:truc}
 \end{align}
 \end{itemize}
 To conclude the proof, it remains to remark
 $$  |p_t|^2 \, \mathds{1}_{2.5M \leq r \leq 2.7M} + \Big( 1-\frac{3M}{r} \Big)^2 \frac{|p_u|^2}{r^2\Omega^4} \, \mathds{1}_{r \geq 2.5M} \lesssim \frac{  |p_{r^*}|^2}{r \log^2(2+r)}+ \Big( 1-\frac{3M}{r} \Big)^2 \frac{|\slashed{p}|^2}{r^3},$$ and to choose $\epsilon>0 $ small enough.
\end{proof}

Finally, in Section \ref{Subsec73}, it will be convenient to simply use $p_N$ and apply the next result. It can be proved by similar but simpler considerations than Proposition \ref{Promultiplierm}.

\begin{proposition}\label{Prorsweight}
There holds
$$ \T_g \big(|p_N| \big) +4|r-3M|\frac{|p_u|^2}{M^2\Omega^4} \mathds{1}_{r \leq 2.5M} \lesssim |p_t|^2 \, \mathds{1}_{2.5M \leq r \leq 2.7M}. $$
\end{proposition}

\subsection{Trapping weight functions}\label{Subsecvarphi}

Let us introduce suitable weight functions to capture quantitatively the trapping effect of null geodesics in Schwarzschild spacetime. 

We first recall some basic terminology about trapped orbits. We say that a null geodesic $\gamma$ is \emph{trapped}, if it is contained in the \emph{trapped set} 
$$\Gamma=\big\{ \, (x,p)\in \mathcal{P} \, \big| \, r-3M=p_{r^*}=0 \, \big\}.$$ 
We also say that a null geodesic $\gamma$ is \emph{future-trapped} if $r(\gamma(s))\to 3M$ as $s\to+\infty$. Similarly, we say that a null geodesic $\gamma$ is \emph{past-trapped} if $r(\gamma(s))\to 3M$ as $s\to-\infty$.

\subsubsection{The trapping weight function $\varphi_-$}

Let us set the weight function that we use to capture the concentration of future-trapped geodesics. For this, we recall the vector field $G$ introduced in \eqref{defG}.

\begin{definition}\label{Defvarphimin}
Let $\varphi_{-} \colon \mathcal{P}\to \R$ be the weight function
$$ \varphi_{-}(x,p)\coloneqq p(G)=  \frac{r }{\Omega} p_{r^*} +\frac{r}{\Omega}\Big(1+\frac{6M}{r}\Big)^{\frac{1}{2}}\Big(1-\frac{3M}{r}\Big)p_t. $$ 
We also introduce the rescaled weight 
\begin{equation}\label{eq:defovphi}
 \pmb{\varphi}_- \coloneqq \Omega^{-1}\varphi_-.
 \end{equation}
\end{definition}

\begin{remark}\label{RkvarphiHplus}
The weight function $\varphi_-$ is not $C^1_{x,p}$ up to $\H$, moreover it vanishes there. For these reasons, we introduce the stronger weight $\pmb{\varphi}_-$, which also captures the redshift effect. Indeed,
\begin{itemize}
\item near $\mathcal{H}^+$, and more generaly for $r \leq 3M$, we observe that
\begin{align*}
 \frac{r}{\Omega}p_{r^*}+\frac{|r+6M|^{\frac{1}{2}}}{r^{\frac{1}{2}}\Omega}(r-3M)p_t & =\frac{r}{\Omega}(-p_t+p_{r^*})+\frac{r^{\frac{3}{2}}+|r+6M|^{\frac{1}{2}}(r-3M)}{r^{\frac{1}{2}} \Omega}p_t \\
 & = \frac{2r}{\Omega}|p_u|- \frac{27M^2r^{\frac{1}{2}}\Omega}{r^{\frac{3}{2}}-|r+6M|^{\frac{1}{2}}(r-3M)} |p_t|.
 \end{align*}
 We can then write
 $$ \frac{2M}{r^2}\frac{|p_u|}{\Omega^2}= \frac{M}{r^3}\pmb{\varphi}_-  + \frac{27M^3r^{\frac{1}{2}}}{r^{\frac{9}{2}}-r^3|r+6M|^{\frac{1}{2}}(r-3M)} |p_t|.   $$
 \item Near $\mathcal{I}^+$, and for $r \geq 3M$, we have
 \begin{align*}
 \pmb{\varphi}_-=\frac{r}{\Omega^2}p_{r^*}+\frac{|r+6M|^{\frac{1}{2}}}{r^{\frac{1}{2}}\Omega^2}(r-3M)p_t &  = -\frac{2r}{\Omega^2}|p_v|+ \frac{27M^2r^{\frac{1}{2}}}{r^{\frac{3}{2}}+|r+6M|^{\frac{1}{2}}(r-3M)} |p_t|.
 \end{align*}
\end{itemize}
\end{remark}

We now investigate the behaviour of $\varphi_-$ along the null geodesic flow.

\begin{lemma}\label{Lemphiminus}
There holds
$$ \T_g(\varphi_-)= - \frac{ |p_t|}{r^{\frac{1}{2}}|r+6M|^{\frac{1}{2}} \Omega^2}\varphi_{-} .$$ 
Moreover, the rescaled weight $\pmb{\varphi}_-$ verifies
$$
 \T_g(\pmb{\varphi}_-)  = - \pmb{a}(r,p_{r^*},p_t) \pmb{\varphi}_-, $$
 where $\pmb{a}(r,p_{r^*},p_t)$ is defined as
 $$ 
 \pmb{a}(r,p_{r^*},p_t)\coloneqq\frac{r^2+2Mr+3M^2}{r|r+6M|^{\frac{1}{2}}(r^{\frac{3}{2}}+M|r+6M|^{\frac{1}{2}}) }|p_t|  +\frac{2M|p_u|}{r^2\Omega^2}   .
$$
\end{lemma}

\begin{proof}
The first equality follows from $\T_g(p_t)=0$,
\begin{align*}
 \partial_r \bigg( (r-3M) \frac{|r+6M|^{\frac{1}{2}}}{r^{\frac{1}{2}}\Omega} \bigg) &=  \frac{|r+6M|^{\frac{1}{2}}}{|r-2M|^{\frac{1}{2}}}+\frac{r-3M}{2|r+6M|^{\frac{1}{2}} |r-2M|^{\frac{1}{2}}}-\frac{|r+6M|^{\frac{1}{2}}(r-3M)}{2|r-2M|^{\frac{3}{2}}} \\
 & = \frac{2(r+6M)(r-2M)+(r-3M)(r-2M-r-6M)}{2|r+6M|^{\frac{1}{2}} |r-2M|^{\frac{3}{2}}}\\
 &=\frac{r^2}{|r+6M|^{\frac{1}{2}} |r-2M|^{\frac{3}{2}}},
 \end{align*}
 and
\begin{align*}
 p_{r^*}\partial_r \bigg(\frac{r }{\Omega}p_{r^*} \bigg)+\frac{r-3M}{r^4} |\slashed{p}|^2 \partial_{p_{r^*}} \bigg(\frac{r }{\Omega}p_{r^*} \bigg) & =  \frac{r^{\frac{1}{2}}(3r-6M-r)}{2|r-2M|^{\frac{3}{2}}}|p_{r^*}|^2 +\frac{r-3M}{r^3\Omega} |\slashed{p}|^2 =\frac{r-3M}{r\Omega^3} |p_t|^2.
\end{align*}
For the second equality, we write
\begin{align}
\nonumber \T_g \big(\Omega^{-1} \varphi_- \big) & = \Omega^{-1}\T_g \big( \varphi_- \big)+p_{r^*} \partial_r \big( \Omega^{-1} \big) \varphi_- \\
 & = \Omega^{-1}\T_g \big( \varphi_- \big)-\frac{Mp_t}{r^2 \Omega^2} \Omega^{-1}  \varphi_- +\frac{M(p_t-p_{r^*})}{r^2 \Omega^2} \Omega^{-1}  \varphi_- . \label{eq:kevatalenn1}
 \end{align}
\end{proof}

\begin{remark}
The weight $\varphi_-$ is a defining function for the submanifold $\{\varphi_-=0\}$ of the null-shell, where past-trapped orbits are located.  This property can be shown using the relations in Lemmata \ref{Lemphiminus} and \ref{lemmavarphiplus}.
\end{remark}

\subsubsection{The trapping weight function $\varphi_+$}
The expansion phenomenon of past-trapped null geodesics can be captured by the weight function $\varphi_{+} \colon \mathcal{P}\to \R$, which is given by
\begin{equation}\label{eq:defvarphiplus}
 \varphi_+ (x,p)\coloneqq\frac{r }{\Omega} p_{r^*} -\frac{r}{\Omega}\Big(1+\frac{6M}{r}\Big)^{\frac{1}{2}}\Big(1-\frac{3M}{r}\Big)p_t.
 \end{equation}
We note that $\varphi_+$ is singular on $\mathcal{H}^+$, however $\Omega \varphi_+$ is smooth up to $\mathcal{H}^+$. 

\begin{lemma}\label{lemmavarphiplus}
There holds
\begin{align*}
 \T_g(\varphi_+)&=  \frac{ |p_t|}{r^{\frac{1}{2}}|r+6M|^{\frac{1}{2}} \Omega^2}\varphi_{+} .
 \end{align*}
Moreover, we have $$ \varphi_+ \varphi_- =  27M^2|p_t|^2-|\slashed{p}|^2.$$ 
\end{lemma}

\begin{proof}
The first property can be proved in the same way as the one in Lemma \ref{Lemphiminus}. For the second property, we simply compute
$$ \varphi_+ \varphi_- =  \frac{r^2 }{\Omega^2} |p_{r^*}|^2 -\frac{r+6M}{r\Omega^2}(r-3M)^2|p_t|^2=-|\slashed{p}|^2+\frac{r^3-(r+6M)(r-3M)^2}{r\Omega^2}|p_t|^2=27M^2|p_t|^2-|\slashed{p}|^2.$$ 
\end{proof}

\begin{remark}
The weight $\varphi_+$ is a defining function for the submanifold $\{\varphi_+=0\}$ of $\mathcal{P}$, where future-trapped orbits are located.  This property can be shown by using the relations in the Lemmata \ref{Lemphiminus} and \ref{lemmavarphiplus}.
\end{remark}

We will rather work with the rescaled weight $\Omega\varphi_+$, which has the advantage of being regular up to $\mathcal{H}^+$ and of being bounded above by $2(r+6M)|p_t|$. The consequence is that $\frac{\Omega|\varphi_+|}{2(r+6M)}$ is a Lyapunov function for the null geodesic flow. 

\begin{lemma}\label{LemOmvarplus}
The rescaled weight function $\Omega \varphi_+ \colon \mathcal{P}\to \R$ verifies 
$$ \T_g \big(\Omega \varphi_+ \big)= \pmb{a}(r,p_{r^*},p_t)\Omega\varphi_+ .$$
\end{lemma}

\begin{proof}
By a direct computation using Lemma \ref{lemmavarphiplus}, we get
\begin{align*}
 \T_g \big(\Omega \varphi_+ \big)&=  \Omega \T_g \big( \varphi_+ \big)+p_{r^*} \partial_r \big( \Omega \big) \varphi_+= \frac{ |p_t|}{r^{\frac{1}{2}}|r+6M|^{\frac{1}{2}} \Omega^2} \Omega\varphi_{+}+\frac{Mp_t}{r^2 \Omega^2} \Omega  \varphi_+ -\frac{M(p_t-p_{r^*})}{r^2 \Omega^2} \Omega  \varphi_+ .
\end{align*}
It remains to recall from Lemma \ref{Lemphiminus} the definition of $\pmb{a}$. 
\end{proof}

In what follows, we will denote the sign function by $\mathrm{sgn}$, so that $\T_g(|g|)=\T_g(g) \cdot \mathrm{sgn}(g)$.

\begin{corollary}\label{Corlogvarphiplus}
We have
$$   \T_g \bigg( \bigg|\log \bigg( \frac{|\Omega\varphi_+|}{2(r+6M)|p_t|} \bigg) \bigg| \bigg) \lesssim- \frac{|p_N|}{r^2}   . $$
\end{corollary}

\begin{proof}
By the previous Lemma \ref{LemOmvarplus} and since $r^2+2Mr+3M^2 \geq r^2+Mr^{\frac{1}{2}}|r+6M|^{\frac{1}{2}}$, we have
\begin{align*}
   \T_g \bigg( \frac{|\Omega\varphi_+|}{r+6M} \bigg) & \geq \bigg( \frac{|p_t|}{r^{\frac{1}{2}}|r+6M|^{\frac{1}{2}} }+\frac{2M|p_u|}{r^2 \Omega^2}-\frac{p_{r^*}}{r+6M}  \bigg) \frac{|\Omega\varphi_+|}{r+6M} \gtrsim   \bigg( \frac{|p_t|}{r^2}+\frac{|p_u|}{r^2 \Omega^2} \bigg) \frac{|\Omega\varphi_+|}{r+6M}.
\end{align*}
 Next, since $\T(2|p_t|)=0$, we remark that
$$   \T_g \bigg( \bigg|\log \bigg( \frac{|\Omega\varphi_+|}{2(r+6M)|p_t|} \bigg) \bigg| \bigg) =  \T_g \bigg( \frac{|\Omega\varphi_+|}{r+6M} \bigg) \frac{(r+6M)}{|\Omega\varphi_+|} \mathrm{sgn} \bigg( \log \bigg( \frac{|\Omega\varphi_+|}{2(r+6M)|p_t|} \bigg)  \bigg) . $$
Finally, as $ |\Omega \varphi_+| < 2 (r+6M)|p_t|$, we have that the last factor on the RHS is identically equal to $-1$. We then obtain the result from the last estimate, and $|p_N| \sim |p_t| +|p_u|\Omega^{-2}$.
\end{proof}

\subsection{The $r^p$-weight functions}

We will capture the dispersion at infinity by using the weight functions\footnote{Note that the rescaled vector field $\frac{r^2}{\Omega^2}\partial_v$ is equal to $-\frac{1}{2}\partial_x$ in the coordinates $(u, x\coloneqq1/r,\theta,\phi)$.} 
$$\frac{r^2}{\Omega^2}p_v=p\Big(\frac{r^2}{\Omega^2}\partial_v\Big), \qquad \quad \frac{r}{\Omega^2}p_v=p\Big(\frac{r}{\Omega^2}\partial_v\Big).$$ 

\begin{lemma}\label{Lemrp}
There holds
\begin{align*}
\T_g \bigg(  \frac{r^2}{\Omega^2} |p_v| \bigg) & = -2\frac{r-3M}{\Omega^4}|p_v|^2  .
 \end{align*}
 More generally, for all $0 \leq p \leq 2$, we have
 \begin{align*}
\T_g \bigg(  \frac{r^p}{\Omega^2} |p_v| \bigg) & = -p\frac{r^{p-1}}{\Omega^4}|p_v|^2+6\frac{Mr^{p-1}}{r\Omega^4}|p_v|^2+(p-2)\frac{r^{p-1}}{\Omega^2}|p_u ||p_v|  .
 \end{align*}
\end{lemma}

\begin{proof}
First, we compute
\begin{align*}
 \T_g \bigg(  \frac{2r^2}{\Omega^2} |p_v| \bigg)  &  = \frac{2p_{r^*}}{\Omega^2} \partial_{r^*} \bigg(  \frac{r^2}{\Omega^2}  \bigg) |p_v| -\frac{r-3M}{r^2 \Omega^2} |\slashed{p}|^2 \partial_{p_{r^*}} (p_{r^*})  = 4\frac{r-3M}{\Omega^4}p_{r^*}|p_v| - \frac{r-3M}{r^2\Omega^2} |\slashed{p}|^2  .
  \end{align*}
We conclude the proof by using $p_{r^*}=p_v-p_u$, and the null-shell relation $\Omega^2 |\slashed{p}|^2=4r^2p_vp_u$. For the second identity, we use $\T_g(r^{p-2})=(p-2)p_{r^*}r^{p-3}$.
\end{proof}

We also show the following high-order form of the previous lemma. 

\begin{corollary}\label{Corrp}
Let $\overline{\chi} \in C^\infty (\R)$ be a cutoff function such that $\overline{\chi} (s) =0$ for $s \leq 4M$ and $\overline{\chi} (s) = 1$ for $s \geq 7M$. For any $p \in \mathbb{N}^*$, we have
\begin{align*}
\T_g \bigg( \overline{\chi}(r) \frac{r^{2p}}{\Omega^{2p}} |p_v|^p \bigg) & \lesssim_p -\frac{r^{2p-1}}{\Omega^{2p}}|p_v|^{p+1} \mathds{1}_{r \geq 7M}+|p_t|^{p+1} \mathds{1}_{4M \leq r \leq 7M} , \\
\T_g \bigg( \overline{\chi}(r) \frac{r^{2p-1}}{\Omega^{2p}} |p_v|^p \bigg) & \lesssim_p -\frac{r^{2p-2}}{\Omega^{2p}}|p_v|^{p}|p_t| \mathds{1}_{r \geq 7M}+|p_t|^{p+1} \mathds{1}_{4M \leq r \leq 7M} .
 \end{align*}
\end{corollary}
\begin{proof}
For $r \geq 7M$, we have
\begin{align*}
\T_g \bigg(  \frac{r}{\Omega^2} |p_v| \bigg) & = -\frac{r-6M}{r\Omega^4}|p_v|^2-\frac{1}{\Omega^2}|p_u ||p_v|  \lesssim -|p_t||p_v|.
 \end{align*}
 It then remains, in view of $\T_g(\overline{\chi}(r))=p_{r^*} \overline{\chi}'(r)$ and the support of $\overline{\chi}$, to combine this last estimate with Lemma \ref{Lemrp} and $|p_v| \leq |p_t|$.
\end{proof}

\section{Energy boundedness and integrated energy decay estimates}\label{Sec4}

Let us show the main zeroth order energy boundedness and integrated local energy decay estimates that we use. Recall that we often write \emph{ILED} in short for integrated local energy decay estimate.

\subsection{Degenerate ILEDs}

We start by proving a boundedness statement together with an ILED that only degenerates at the photon sphere and at spatial infinity.

\begin{proposition}\label{ProremainderILED}
Let $f$ be a solution to the massless Vlasov equation $\T_g(f)=0$. Then,
\begin{equation*}
\sup_{\tau \geq 0} \,\mathbb{E} \big[ p_N f \big](\tau)+\int_{\pi^{-1}(\mathcal{R})}  \frac{|r-3M|^2|p_N|^2}{r^3 \log^2(2+r)}|f|+ \frac{|p_{r^*}|^2 }{r \log^2(2+r)} |f|+\bigg|1-\frac{3M}{r} \bigg|^2 \frac{|\slashed{p}|^2}{r^3} |f| \dr \mu_{\mathcal{P}} \lesssim \mathbb{E} \big[ p_Nf \big](0). 
\end{equation*}
\end{proposition}

\begin{proof}
From Proposition \ref{Promultiplierm}, we recall the estimate for $\T_g(|\xi|)$ and that $|\xi|\sim |p_N|$. The result then follows from the energy estimate of Proposition \ref{Proenergy} applied to $\xi f$.
\end{proof}

A similar degenerate ILED for massless Vlasov fields was previously derived in \cite[Proposition 3.4]{L23}. In particular, the degeneracy at $\{p_{r^*}=0, \, r=3M \}$ can be weakened (see also Lemma \ref{LemPropz}). We observe that if we do not assume the finiteness of a stronger energy norm than $\mathbb{E}[p_N f](0)$, the degeneracy at the photon sphere is necessary as proved in \cite[Proposition 3.12]{L23}. Nonetheless, one can remove the degeneracy at $r=3M$ in the previous ILED by assuming stronger decay assumptions on the initial data.

\begin{proposition}\label{ProILED}
Let $f$ be a solution to the massless Vlasov equation. We have,
$$ \sup_{\tau \geq 0} \,  \mathbb{E} \bigg[ \bigg\langle \log \bigg( \frac{|\Omega \varphi_+ |}{r|p_t|} \bigg) \bigg\rangle p_N f \bigg](\tau)+ \int_{\pi^{-1}(\mathcal{R})}  \bigg(\frac{|p_N|^2}{r \log^2(2+r)}+ \frac{|\slashed{p}|^2}{r^3} \bigg) |f| \dr \mu_{\mathcal{P}} \lesssim  \mathbb{E} \bigg[ \bigg\langle \log \bigg( \frac{|\Omega \varphi_+ |}{r|p_t|} \bigg) \bigg\rangle p_N f \bigg](0).$$
\end{proposition}

\begin{proof}
By Corollary \ref{Corlogvarphiplus} and $\T_g(|\xi  f|)\leq 0$, we have 
$$ \T_g \bigg( \bigg| \log \bigg( \frac{|\Omega \varphi_+ |}{2(r+6M)|p_t|} \bigg) \bigg| |\xi f| \bigg)   \lesssim  - \frac{|p_N|}{r^2} |\xi f| \lesssim -\frac{|p_N|^2}{r^2}|f| .$$
The result then follows from the energy estimate of Proposition \ref{Proenergy} and from Proposition \ref{ProremainderILED}.
\end{proof}

For the previous two ILEDs, the bulk term is degenerate with respect to the boundary term. We now prove an ILED without relative degeneration for a degenerate weighted norm.

\begin{proposition}\label{ProILED000}
Let $f $ be a solution to the massless Vlasov equation $\T_g(f)=0$. There holds,
$$
\sup_{\tau \geq 0} \,\mathbb{E} \big[ \pmb{\varphi}_- f \big](\tau) + \int_{\pi^{-1}(\mathcal{R})}  \frac{|p_N|}{r} \big|\pmb{\varphi}_- f \big| \dr \mu_{\mathcal{P}}  \lesssim \mathbb{E} \big[ \pmb{\varphi}_- f \big](0).$$
\end{proposition}
\begin{proof}
It suffices to apply Proposition \ref{Proenergy} to $\pmb{\varphi}_- f$, and to compute $\T_g(\pmb{\varphi}_- f)=\T_g( \pmb{\varphi}_- )f$ using Lemma \ref{Lemphiminus}. Then, we use $p_N \sim p_t+\Omega^{-2} p_u$ and $p_N=p_t$ for $r \geq 2.7M$.
\end{proof}

\subsection{ILEDs for the $r^p$-method}

We show an ILED that will be used to implement the $r^p$-method. We denote by $\lceil \cdot \rceil$ the ceiling function.

\begin{proposition}\label{ProILEDrp}
Let $p \in \mathbb{N}^*$ and $f$ be a solution to the massless Vlasov equation. Then,
\begin{equation*}
\sup_{\tau \geq 0}\, \mathbb{E}\bigg[  \Big\langle r^{p} \Big|\frac{p_v}{p_t}\Big|^{\lceil \frac{p}{2} \rceil} \Big\rangle \pmb{\varphi}_- f \bigg] (\tau)+\int_{\tau=0}^{+\infty} \mathbb{E}\bigg[  \Big\langle r^{p-1} \Big|\frac{p_v}{p_t}\Big|^{\lceil \frac{p-1}{2} \rceil} \Big\rangle \pmb{\varphi}_- f \bigg] (\tau) \dr \tau \lesssim_p \mathbb{E}\bigg[ \Big\langle r^{p} \Big|\frac{p_v}{p_t}\Big|^{\lceil \frac{p}{2} \rceil} \Big\rangle \pmb{\varphi}_- f \bigg] (0) .
\end{equation*}
\end{proposition}
\begin{proof}
Note first, according to Corollary \ref{Corrp} and using $\T_g (|\pmb{\varphi}_-|) \leq 0$, that
$$ \T \Big( \overline{\chi}(r) \frac{ r^{p}}{\Omega^{2p}} \Big|\frac{p_v}{p_t}\Big|^{\lceil \frac{p}{2} \rceil} \big| \pmb{\varphi}_- f \big| \Big) \lesssim_p |p_t|\big| \pmb{\varphi}_- f \big| \mathds{1}_{4M \leq r \leq 7M}- r^{p-1} \Big|\frac{p_v}{p_t}\Big|^{\lceil \frac{p-1}{2} \rceil}|p_v|\big| \pmb{\varphi}_- f \big| \mathds{1}_{r \geq 7M}.$$
Let $A_p>0$ be a sufficiently large constant and recall the parameter $R_0 > 3M$ of the foliation $(\Sigma_\tau)_{\tau \geq 0}$. As $\T_g (\pmb{\varphi}_-) \lesssim -|p_N|r^{-1}\pmb{\varphi}_-$ and $|p_v| \lesssim |p_N|$, we have
$$ \T_g \bigg[ \bigg(A_p+  r^{p} \Big|\frac{p_v}{p_t}\Big|^{\lceil \frac{p}{2} \rceil} \bigg) \big| \pmb{\varphi}_- f \big| \bigg] \lesssim_p -|p_N|\big| \pmb{\varphi}_- f \big| \mathds{1}_{r \leq R_0}- r^{p-1} \Big|\frac{p_v}{p_t}\Big|^{\lceil \frac{p-1}{2} \rceil}|p_v|\big| \pmb{\varphi}_- f \big| \mathds{1}_{r \geq R_0}.$$
Recall from Lemma \ref{LempSig} that $|p_{n_{\Sigma_\tau}}| \sim |p_N|$ for $r \leq R_0$ and $|p_{n_{\Sigma_\tau}}| =|p_v|$ for $r \geq R_0$. Hence,
$$ \T_g \bigg[ \bigg(A_p+  r^{p} \Big|\frac{p_v}{p_t}\Big|^{\lceil \frac{p}{2} \rceil} \bigg) \big| \pmb{\varphi}_- f \big| \bigg] \lesssim_p -r^{p-1} \Big|\frac{p_v}{p_t}\Big|^{\lceil \frac{p-1}{2} \rceil}|p_{n_{\Sigma_\tau}}|\big| \pmb{\varphi}_- f \big|.$$
We then deduce from the energy estimate of Proposition \ref{Proenergy} and from \eqref{eq:coarea} that
$$ \sup_{\tau \geq 0}\, \mathbb{E}\bigg[  \Big\langle r^{p} \Big|\frac{p_v}{p_t}\Big|^{\lceil \frac{p}{2} \rceil} \Big\rangle \pmb{\varphi}_- f \bigg] (\tau)+\int_{\tau=0}^{+\infty} \mathbb{E}\bigg[   r^{p-1} \Big|\frac{p_v}{p_t}\Big|^{\lceil \frac{p-1}{2} \rceil}  \pmb{\varphi}_- f \bigg] (\tau) \dr \tau \lesssim_p \mathbb{E}\bigg[ \Big\langle r^{p} \Big|\frac{p_v}{p_t}\Big|^{\lceil \frac{p}{2} \rceil} \Big\rangle \pmb{\varphi}_- f \bigg] (0) .$$
The result then ensues by using also this estimate for $p=1$.
\end{proof}

\begin{remark}\label{Rkrpgen}
The ILED of Proposition \ref{ProILEDrp} can be extended to the more general case when $p\geq 0$ instead of $p\in \mathbb{N}^*$ (see \cite[Section 4]{L23}). \end{remark}

Finally, we prove a slight generalisation of Proposition \ref{ProILEDrp}. We will apply it to the system of the commuted Vlasov equation that we consider in Section \ref{SecILEDwrd}. The individual derivatives that we will consider do not satisfy the Vlasov equation. However, we will circumvent this difficulty by considering a well-chosen weighted combination of them.
 
\begin{proposition}\label{ProILEDrpbis}
Let $ n \geq 1$. Let further $g_k \colon \mathcal{P} \to \R$, with $1 \leq k \leq n$, be sufficiently regular distribution functions such that
\begin{equation}\label{eq:hypgk}
 \sum_{1 \leq k \leq n}  \T_g (|g_k|) \lesssim -\sum_{1 \leq k \leq n} |p_N| |g_k| \mathds{1}_{r \leq 7M}.
 \end{equation}
 Then, for any $p \in \mathbb{N}^*$, we have
$$ \sup_{\tau \geq 0} \sum_{1 \leq k \leq n} \mathbb{E}\Big[  \Big\langle r^{p} \Big|\frac{p_v}{p_t}\Big|^{\lceil \frac{p}{2} \rceil} \Big\rangle g_k \Big] (\tau)+\sum_{1 \leq k \leq n} \int_{\tau=0}^{+\infty} \mathbb{E}\Big[  \Big\langle r^{p-1} \Big|\frac{p_v}{p_t}\Big|^{\lceil \frac{p-1}{2} \rceil} \Big\rangle g_k \Big] (\tau) \dr \tau \lesssim_p \sum_{1 \leq k \leq n} \mathbb{E}\Big[ \Big\langle r^{p} \Big|\frac{p_v}{p_t}\Big|^{\lceil \frac{p}{2} \rceil} \Big\rangle g_k \Big] (0) .$$
\end{proposition}
\begin{proof}
It suffices to follow the proof of Proposition \ref{ProILEDrp} by formally replacing $|\pmb{\varphi}_- f|$ by $\sum_k  |g_k|$. 
\end{proof}

\section{Commutation vector fields}\label{Sec5}

Let us introduced a well-chosen class of commutation vector fields that will be used in the article. 

\subsection{The complete lifts of the Killing vector fields}\label{Subseclift}

For any (conformal) Killing vector field $X\in \Gamma(T\mathcal{S})$, the associated complete lift $\widehat{X}\in \Gamma(TT^*\mathcal{S})$ is tangent to $\mathcal{P}$, and commutes with the Vlasov operator $\T_g$. In fact, $\widehat{X} =\T_{p(X)}$, the symplectic gradient of the conserved quantity $p(X)$. This result holds in a general Lorentzian manifold. See \cite[Appendix~F]{AGS} and \cite[Appendix~C]{FJS17} for more information. We restrict our discussion to the case of a Schwarzschild background.

\begin{definition}
The complete lifts of the Killing vector fields of Schwarzschild spacetime are
\begin{align*}
\widehat{\mathbf{\Omega}}_1 \coloneqq &\,-\sin \phi \,\partial_{\theta}-\cot \theta\cos \phi \,\partial_{\phi}-\cos\phi\frac{p_{\phi}}{\sin^2\theta}\partial_{p_{\theta}}+(\cos \phi \,p_{\theta}-\sin\phi \cot \theta \,p_{\phi})\partial_{p_{\phi}},\\
\widehat{\mathbf{\Omega}}_2 \coloneqq &\,\cos\phi \,\partial_{\theta}-\cot \theta\sin\phi\partial_{\phi}-\sin\phi \frac{p_{\phi}}{\sin^2\theta}\partial_{p_{\theta}}+(\sin\phi \,p_{\theta}+\cos\phi \cot \theta \,p_{\phi})\partial_{p_{\phi}},\\
\widehat{\mathbf{\Omega}}_3 \coloneqq &\,\partial_{\phi},\qquad\qquad \widehat{\partial}_t=\partial_t.
\end{align*}
\end{definition}

The time derivative and the angular derivatives of the energy-momentum tensor $\mathrm{T}[f]$ of a massless Vlasov field $f$ can then be estimated using the next result.

\begin{proposition}\label{ProKilling}
For any $\widehat{Z} \in \big\{ \partial_t,\, \widehat{\mathbf{\Omega}}_1,\, \widehat{\mathbf{\Omega}}_2,\, \widehat{\mathbf{\Omega}}_3 \big\}$, we have $[\T_g,\widehat{Z} ]=0$. Moreover, there holds
$$ \mathcal{L}_Z \big( \mathrm{T}[f] \big) = \mathrm{T} \big[ \widehat{Z}f \big].$$
\end{proposition}

It will be convenient to work with the following commuting vector fields associated to the total angular momentum $|\slashed{p}|$. Recall from \eqref{defQ} the Carter tensor $Q$.

\begin{definition}
Let $\T_Q \in \Gamma(T\mathcal{P})$ be the projection of the symplectic gradient $\frac{1}{2}\T_{|\slashed{p}|^2}$. In local coordinates, 
$$ \T_Q \coloneqq p_{\theta} \partial_{\theta} + \frac{p_{\phi}}{ \sin^2\theta } \partial_{\phi} + \frac{\cot \theta }{ \sin^2 \theta}  p_{\phi}^2 \partial_{p_{\theta} }   .$$
In order to work with a vector field with the same homogeneity in $p$ as the other vector fields we will manipulate, we also introduce
 $$ \T_{|\slashed{p}|} = |\slashed{p}|^{-1} \T_Q.$$
\end{definition}

\begin{corollary}\label{CorTQ}
There holds $\T_Q=\sum_i p(\mathbf{\Omega}_i) \widehat{\mathbf{\Omega}}_i$. In particular, we have $[\T_g,\T_Q]=[\T_g,\T_{|\slashed{p}|}]=0$.
\end{corollary}

Note that the radial derivative does not generate a symmetry of Schwarzschild spacetime. Consequently, controlling $\mathcal{L}_{\partial_r} \mathrm{T}[f]$ requires a more thorough analysis.

\subsection{The radial scaling vector fields}

We investigate further the structure of the massless Vlasov operator. We focus on the properties of the time-radial part $\pmb{R} \in \Gamma(T\mathcal{P})$ and radial part $R\in \Gamma(T\mathcal{P})$ of $\T_g$, given by
 \begin{equation}\label{eq:defR}
  \pmb{R} \coloneqq -\frac{p_t}{\Omega^2} \partial_t+ \frac{p_{r^*}}{\Omega^2}\partial_{r^*}+\frac{r-3M}{r^4} |\slashed{p}|^2 \partial_{p_{r^*}},  \qquad \qquad R \coloneqq \frac{p_{r^*}}{\Omega^2}\partial_{r^*}+\frac{r-3M}{r^4} |\slashed{p}|^2 \partial_{p_{r^*}}.
  \end{equation}
In this relativistic setting, it is more natural to work with $\pmb{R}$. We note that contrary to $R$, the vector field $\pmb{R}$ is regular up to $\H$. Nonetheless, as $\pmb{R}-R$ is a multiple of $\partial_t$, which commutes with $\T_g$, focusing on $R$ will sometimes be convenient.

\begin{lemma}\label{LemDecompT}
The following decomposition holds
$$ \T_g= \pmb{R}+\frac{1}{r^2} \T_Q.$$
\end{lemma}

Note that if $f$ is solution to the massless Vlasov equation $\T_g(f)=0$, Lemma \ref{LemDecompT} together with Corollary \ref{CorTQ} allow to prove boundedness for $r^2 \pmb{R}f$. Unfortunately, this does not allow to get a satisfying estimate for $\mathcal{L}_{\partial_r}\mathrm{T}[f]$. In order to obtain more information on this last quantity, let us exhibit a scaling vector field in the span of $\{\partial_{r^*},\partial_{p_{r^*}}\}$ that commutes well with $R$. Let $L \in \Gamma(T \mathcal{P})$ be the vector field defined by
$$ L \coloneqq g(r) \partial_r + p_{r^*} \partial_{p_{r^*}}, \qquad \quad g(r) \coloneqq \frac{1}{27M^2}r(r-3M)(r+6M)  .$$

\begin{proposition}\label{ProL}
There holds
\begin{equation}\label{commut_scaling_1}
[R,L]=\frac{(r-3M)(r+5M)}{9M^2}R.
\end{equation}
\end{proposition}

\begin{proof}
We will prove that $L$ is the unique vector field of the form $g(r) \partial_r + p_{r^*} \partial_{p_{r^*}}$ that is not singular at the photon sphere and having a commutation of the form \eqref{commut_scaling_1} with $R$. For this, we first note that
\begin{align*}
  [R,L] &= p_{r^*} \partial_r g(r) \partial_r+\frac{r-3M}{r^4}|\slashed{p}|^2 \partial_{p^*}-\bigg(g(r) \partial_r \bigg(\frac{r-3M}{r^4} \bigg) |\slashed{p}|^2 \partial_{p_{r^*}}+p_{r^*}\partial_r \bigg) \\
  & = \big(\partial_r g(r)-1 \big) p_{r^*} \partial_r+ \bigg(\frac{r-3M}{r^4}  -g(r) \partial_r \bigg( \frac{r-3M}{r^4} \bigg) \bigg) |\slashed{p}|^2\partial_{p_{r^*}}.
  \end{align*}
Then, $[R,L]$ is collinear to $R$ if and only if $g$ verifies
$$ g(r)\partial_r \left( \frac{r-3M}{r^4} \right)-\frac{r-3M}{r^4}=\big(1-\partial_r g(r) \big)\frac{r-3M}{r^4}, $$ 
which is equivalent to
 $$\partial_r \left( g(r) \frac{r-3M}{r^4}\right)=2\frac{r-3M}{r^4}.$$
Integrating this last property and using that $g$ is continuous, we get
$$g(r)\frac{r-3M}{r^4} = 2\int_{3M}^r  \frac{s-3M}{s^4} \mathrm{d}s =-\frac{1}{r^2}+\frac{2M}{r^3}+\frac{1}{27M^2} = \frac{(r-3M)^2(r+6M)}{27M^2r^3}.$$ 
It gives us the expression of $g(r)$, from which we can compute $\partial_r g(r)-1$ and conclude the proof.
\end{proof}

It will be convenient to add a vector field proportional to $\partial_t$ to $L$ in order to improve its commutation properties with $\pmb{R}$. Let $\pmb{L} \in \Gamma(T \mathcal{P})$ be given by
 \begin{equation}\label{eq:defL}
  \pmb{L} \coloneqq -\frac{g(r)p_{r^*}}{\Omega^2p_t} \partial_t + L.
  \end{equation}

\begin{proposition}\label{ProbarY}
There holds
$$ [\T_g,\pmb{L}] = \frac{(r-3M)(r+5M)}{9M^2}\T_g-\frac{(r-3M)(r+3M)}{27M^2r^2}\T_Q+\frac{4(r-3M)(r+3M)p_u p_v}{27M^2 \Omega^2p_t} \partial_t.$$
\end{proposition}
\begin{proof}
Recall $\T_g=-\Omega^{-2}p_t\partial_t+R+r^{-2}\T_Q$ and note that $[\T_Q,L]=[\partial_t,L]=0$. Hence, by Proposition \ref{ProL},
\begin{align*}
 [\T_g,L] & = [R,L]+L \Big( \, \frac{p_t}{\Omega^2} \, \Big)\partial_t+2\frac{g(r)}{r^3} \T_Q \\
 & = \frac{(r-3M)(r+5M)}{9M^2}R+2\frac{g(r)}{r^3}\T_Q+\bigg( -g(r)\frac{r-3M}{r^4}\frac{|\slashed{p}|^2}{\Omega^2 p_t}-\frac{2Mg(r)p_t}{r^2\Omega^4}+\frac{|p_{r^*}|^2}{\Omega^2p_t}\bigg) \partial_t .
 \end{align*}
 Moreover, as $\T_g(p_t)=0$ and $[\T_g, \partial_t]=0$, we have
 $$ \Big[ \T_g, -\frac{g(r)p_{r^*}}{\Omega^2p_t} \partial_t \Big] = -\frac{g(r)(r-3M)|\slashed{p}|^2}{r^4\Omega^2p_t} \partial_t +\frac{2Mg(r)|p_{r^*}|^2}{r^2\Omega^4p_t} \partial_t-\frac{|p_{r^*}|^2}{ \Omega^2 p_t} \partial_r g(r) \partial_t.$$
We now write $R=\T_g+p_t \Omega^{-2} \partial_t-r^{-2} \T_Q$, and use the relations 
$$\partial_r g(r)-1=\frac{(r-3M)(r+5M)}{9M^2},\qquad \quad \Omega^2 |\slashed{p}|^2=r^2 \big( |p_t|^2-|p_{r^*}|^2 \big),$$ 
to get
\begin{align*}
 [\T_g,\pmb{L}] &= \frac{(r-3M)(r+5M)}{9M^2}\T_g-\frac{(r-3M)(r+3M)}{27M^2}\T_Q\\
 &\quad -\bigg( 2g(r)\frac{r-3M}{r^2 \Omega^4} +\frac{2Mg(r)}{r^2\Omega^4}-\frac{(r-3M)(r+5M)}{9M^2 \Omega^2} \bigg)\frac{|p_t|^2-|p_{r^*}|^2}{p_t} \partial_t,
\end{align*}
from which we deduce the result.
\end{proof}

\begin{remark}
The vector field $\pmb{L}$ has to be compared with a scaling vector field used in \cite{VV24} for the study of the linearised system $\mathbb{X}_{\mathrm{lin}}(f)=0$ (recall Subsection \ref{Subsecideas}) and an associated non-linear problem. It is closely related to the vector fields $V_{\pm}$ that we use to deal with trapping (see Proposition \ref{ProvarphiV} below).
\end{remark}

\subsection{Trapping vector fields}\label{SubsecVminus}

Let us introduce some projections on $T \mathcal{P}$ parallel to $\partial_{p_t}$ of symplectic gradients that arise from the unstable and stable trapping effects. We refer to Appendix \ref{SecCotan} for more details. 

We first introduce $V_+\in \Gamma(T\mathcal{P})$, the projection of $\T_{\varphi_-}$, which is given in coordinates by
 $$V_+\coloneqq  \partial_{p_t}( \varphi_-)\partial_t +\partial_{p_{r^*}} (\varphi_-) \partial_{r^*}-\partial_{r^*} (\varphi_-) \partial_{p_{r^*}},$$
where $\partial_{p_t} \varphi_-$ has to be understood as a quantity defined on $T^* \mathcal{S}$.

\begin{proposition}\label{ProdefVminus}
In local coordinates, we have
\begin{equation}\label{eq:defVminus}
 V_+  = \frac{|r+6M|^{\frac{1}{2}}}{r^{\frac{1}{2}}\Omega}(r-3M) \partial_t+\frac{r}{\Omega} \partial_{r^*}-\bigg( \frac{r-3M}{r\Omega}p_{r^*}+\frac{r^{\frac{1}{2}}}{|r+6M|^{\frac{1}{2}} \Omega}p_t \bigg) \partial_{p_{r^*}}.
 \end{equation}
\end{proposition}
\begin{proof}
The quantities $\partial_{p_t} \varphi_-$ and $\partial_{p_{r^*}} \varphi_-$ can be easily computed from the expression of $\varphi_-$, introduced in Definition \ref{Defvarphimin}. Then, we have 
\begin{align*}
\partial_{r^*}( \varphi_-) & = \Omega^2\frac{r^{\frac{1}{2}}(r-3M)}{|r-2M|^{\frac{3}{2}}}p_{r^*} +\Omega^2\frac{r^2}{|r+6M|^{\frac{1}{2}} |r-2M|^{\frac{3}{2}}}p_t=\frac{(r-3M)}{r\Omega}p_{r^*}+\frac{r^{\frac{1}{2}}}{|r+6M|^{\frac{1}{2}} \Omega}p_t.
\end{align*}
\end{proof}

\begin{remark}\label{RqVminus}
It will sometimes be convenient to use an alternative expression for $V_+$ by rewriting the last term of \eqref{eq:defVminus} using 
\begin{equation}\label{eq:equalityforVminus}
\frac{r-3M}{r\Omega}p_{r^*}+\frac{r^{\frac{1}{2}}}{|r+6M|^{\frac{1}{2}} \Omega}p_t= \frac{r-3M}{r^2} \varphi_- +\frac{27M^2 \Omega}{r^{\frac{3}{2}}|r+6M|^{\frac{1}{2}}}p_t.
\end{equation}
\end{remark}

We now compute two commutators that will be useful to determine $ [\T_g,V_+]$.

\begin{lemma}\label{LemComdrdpr}
There holds
\begin{align*}
[\T_g,\partial_{r^*}] &= \bigg(\frac{(r-3M)p_t}{r^2 \Omega^2}+\frac{(r-3M)|p_{r^*}|^2}{r^2 \Omega^2 p_t} \bigg) \partial_t-2\frac{r-3M}{r^2 \Omega^2}p_{r^*} \partial_{r^*}+ \Omega^2 \frac{r-6M}{r^5} |\slashed{p}|^2 \partial_{p_{r^*}}+2\frac{\Omega^2}{r}\T_g  , \\[2pt]
\big[ \T_g, \partial_{p_{r^*}} \big] & = \frac{p_{r^*}}{ \Omega^2p_t} \partial_t- \frac{1}{\Omega^2} \partial_{r^*}.
\end{align*}
\end{lemma}

\begin{proof}
We first recall the definition \eqref{eq:defR} of $\pmb{R}$ and Lemma \ref{LemDecompT}. We have
$$ \T_g = \pmb{R}+r^{-2} \T_Q, \qquad \qquad [\T_Q, \partial_{r^*}]=[\T_Q, \partial_{p_{r^*}}]=0.$$
The second commutator can then be obtained from \eqref{eq:defConsQua} and \eqref{eq:defT}. For the first identity, we compute
$$ \partial_{r^*} \big( \Omega^{-2} \big) = - \frac{2M}{r^2 \Omega^2}, \quad\qquad \partial_{r^*} \bigg( \frac{r-3M}{r^4} \bigg) = -3 \Omega^2 \frac{r-4M}{r^5}, \quad\qquad \partial_{r^*} \big( r^{-2} \big)= -2 \frac{\Omega^2}{r}r^{-2},$$
which can be obtained using $\Omega^2=1-\frac{2M}{r}$ and $\partial_{r^*} = \Omega^2 \partial_r$. Moreover, we also have
\begin{equation}\label{preuvecompartialr} 
\partial_{r^*} (p_t)= \partial_{r^*} \bigg( \frac{\Omega^2}{r^2} \bigg) \frac{|\slashed{p}|^2}{2p_t}= - \Omega^2 \frac{(r-3M)|\slashed{p}|^2}{r^4p_t}  ,
\end{equation}
 by using again \eqref{eq:defConsQua}. We then deduce that
\begin{align*}
[\T_g,\partial_{r^*}] &=   -\frac{(r-3M)|\slashed{p}|^2}{r^4p_t}  \partial_t   - \frac{2M}{r^2 \Omega^2}p_t \partial_t+ \frac{2M}{r^2 \Omega^2} p_{r^*} \partial_{r^*} +3 \Omega^2 \frac{r-4M}{r^5} |\slashed{p}|^2 \partial_{p_{r^*}} \\
& \quad +2\frac{\Omega^2}{r}\T_g+\frac{2}{r} p_t \partial_t-\frac{2}{r}p_{r^*} \partial_{r^*} -2\Omega^2 \frac{r-3M}{r^5}|\slashed{p}|^2 \partial_{p_{r^*}} .
\end{align*}
It remains to use the null-shell relation $\Omega^2|\slashed{p}|^2=r^2p_t^2-r^2p_{r^*}^2$.
\end{proof}

We can then compute $[\T_g,V_+]$. One can check that the coefficient of $V_+$ in the error term has a good sign. This suggests that $V_+$ is a good vector field for the analysis of massless Vlasov fields.

\begin{proposition}\label{ProComVminus}
There holds
\begin{align*}
[\T_g,V_+] &=  - \frac{|p_t|}{r^{\frac{1}{2}}|r+6M|^{\frac{1}{2}}\Omega^2} V_+- \frac{r+3M}{r^{\frac{3}{2}}|r+6M|^{\frac{3}{2}}}|p_t|\varphi_- \partial_{p_{r^*}}  + 2\Omega \T_g.
\end{align*}
\end{proposition}

\begin{proof}
We will work with the expression of $V_+$ given by Remark \ref{RqVminus}. Since $[\T_g,\partial_t]=0$, $\T_g(p_t)=0$ and $\partial_{r^*}=\Omega^2 \partial_r$, we have
\begin{align*}
[\T_g,V_+] &= p_{r^*}\partial_{r} \bigg( \frac{|r+6M|^{\frac{1}{2}}}{r^{\frac{1}{2}}\Omega}(r-3M) \bigg) \partial_t+ p_{r^*}\partial_{r} \bigg( \frac{r}{\Omega} \bigg) \partial_{r^*} -p_{r^*} \partial_r \bigg( \frac{r-3M}{r^2} \bigg) \varphi_- \partial_{p_{r^*}}- \frac{r-3M}{r^2} \T_g (\varphi_-) \partial_{p_{r^*}} \\
& \quad -p_{r^*} \partial_r \bigg( \frac{27M^2 \Omega}{|r+6M|^{\frac{1}{2}}r^{\frac{3}{2}}} \bigg) p_t \partial_{p_{r^*}}+\frac{r}{\Omega}[\T_g, \partial_{r^*} ]-\bigg( \frac{r-3M}{r^2} \varphi_- +\frac{27M^2 \Omega}{|r+6M|^{\frac{1}{2}}r^{\frac{3}{2}}}p_t \bigg) \big[ \T_g, \partial_{p_{r^*}} \big]  .
\end{align*}
From Lemma \ref{LemComdrdpr}, we can infer that
$$ [\T_g,V_+]= a_t \partial_t+a_{r^*} \partial_{r^*}+a_{p_{r^*}} \partial_{p_{r^*}}+ 2 \Omega \T_g,  $$
where $a_x$ are smooth functions of $r$, $p_t$, $p_{r^*}$ and $|\slashed{p}|^2$. Let us determine $a_{t}$. As $r^{\frac{1}{2}} \Omega= |r-2M|^{\frac{1}{2}}$, we have
\begin{align*}
 \partial_r \bigg( \frac{|r+6M|^{\frac{1}{2}}}{r^{\frac{1}{2}} \Omega}  (r-3M) \bigg) &=  \frac{|r+6M|^{\frac{1}{2}}}{|r-2M|^{\frac{1}{2}}}+\frac{r-3M}{2|r+6M|^{\frac{1}{2}} |r-2M|^{\frac{1}{2}}}-\frac{(r-3M)|r+6M|^{\frac{1}{2}}}{2|r-2M|^{\frac{3}{2}}} \\
 & = \frac{2(r+6M)(r-2M)+(r-3M)(r-2M-r-6M))}{2|r+6M|^{\frac{1}{2}} |r-2M|^{\frac{3}{2}}}\\
 &=\frac{r^2}{|r+6M|^{\frac{1}{2}} |r-2M|^{\frac{3}{2}}}.
 \end{align*}
 Thus, we get using first Lemma \ref{LemComdrdpr} and then relation \eqref{eq:equalityforVminus} that
 \begin{align*}
 a_{t} & = \frac{r^2}{|r+6M|^{\frac{1}{2}} |r-2M|^{\frac{3}{2}}}p_{r^*}+\frac{(r-3M)p_t}{r \Omega^3}+\frac{(r-3M)|p_{r^*}|^2}{r \Omega^3 p_t} -\bigg( \frac{r-3M}{r^2} \varphi_- +\frac{27M^2 \Omega}{|r+6M|^{\frac{1}{2}}r^{\frac{3}{2}}}p_t \bigg) \frac{p_{r^*}}{ \Omega^2p_t} \\
 & = \frac{(r-3M)p_t}{r \Omega^3}= - \frac{|p_t|}{r^{\frac{1}{2}}|r+6M|^{\frac{1}{2}}\Omega^2} \cdot \frac{|r+6M|^{\frac{1}{2}}}{r^{\frac{1}{2}}\Omega}(r-3M) .
 \end{align*}
Applying once again Lemma \ref{LemComdrdpr} and the relation \eqref{eq:equalityforVminus}, we obtain
\begin{align*}
a_{r^*} & =p_{r^*}\partial_{r} \bigg( \frac{r}{\Omega} \bigg) -2\frac{r-3M}{r \Omega^3}p_{r^*}+\frac{1}{\Omega^2}\bigg( \frac{r-3M}{r\Omega}p_{r^*}+\frac{r^{\frac{1}{2}}}{|r+6M|^{\frac{1}{2}} \Omega}p_t \bigg) = -\frac{|p_t|}{r^{\frac{1}{2}}|r+6M|^{\frac{1}{2}} \Omega^2}\cdot \frac{r}{\Omega}.
\end{align*}
We recall from Lemma \ref{Lemphiminus} the expression of $\T_g(\varphi_-)$. We then have to prove that
 $$ a_{p_{r^*}} +\frac{r-3M}{r^2} \T_g (\varphi_-) - \frac{|p_t|}{r^{\frac{1}{2}}|r+6M|^{\frac{1}{2}}\Omega^2}\cdot \frac{27M^2 \Omega}{|r+6M|^{\frac{1}{2}}r^{\frac{3}{2}}}p_t=- \frac{r+3M}{r^{\frac{3}{2}}|r+6M|^{\frac{3}{2}}}|p_t|\varphi_-.$$
 For this, we compute
 $$ \partial_{r} \bigg( \frac{r-3M}{r^2} \bigg) = -\frac{r-6M}{r^3}, \qquad\quad  \partial_r \bigg( \frac{27M^2 \Omega}{|r+6M|^{\frac{1}{2}}r^{\frac{3}{2}}} \bigg)= -\frac{54M^2(r^2+2Mr-12M^2)}{|r-2M|^{\frac{1}{2}} |r+6M|^{\frac{3}{2}} r^3}.$$
 Hence, by Lemma \ref{LemComdrdpr}, we have
 $$ a_{p_{r^*}} +\frac{r-3M}{r^2} \T_g (\varphi_-)= \frac{r-6M}{r^3}\varphi_- p_{r^*}+ \frac{54M^2(r^2+2Mr-12M^2)}{|r-2M|^{\frac{1}{2}} |r+6M|^{\frac{3}{2}} r^3}p_{r^*} p_t +\Omega \frac{r-6M}{r^4} |\slashed{p}|^2 .$$
 Moreover, by the null-shell relation there holds
 $$ \Omega \frac{r-6M}{r^4} |\slashed{p}|^2=  \frac{r-6M}{r^2 \Omega} (|p_t|^2-|p_{r^*}|^2)=\frac{r-6M}{r^2 \Omega} |p_t|^2-\frac{r-6M}{r^3 } \varphi_- p_{r^*}+\frac{(r-6M)(r-3M)|r+6M|^{\frac{1}{2}}}{r^{\frac{7}{2}} \Omega } p_t p_{r^*}. $$
We further remark that
\begin{align*}
  \frac{r-6M}{r^2 \Omega} |p_t|^2  -\frac{|p_t|}{r^{\frac{1}{2}}|r+6M|^{\frac{1}{2}}\Omega^2}\cdot \frac{27M^2 \Omega}{|r+6M|^{\frac{1}{2}}r^{\frac{3}{2}}}p_t &=      \frac{r^2-9M^2}{r^2(r+6M) \Omega} |p_t|^2   \\
  &=\frac{r+3M}{r^{\frac{3}{2}}|r+6M|^{\frac{3}{2}} } p_t \varphi_--  \frac{r+3M}{r^{\frac{1}{2}}|r+6M|^{\frac{3}{2}} \Omega } p_t p_{r^*}. 
\end{align*}
The result then ensues from
$$ \frac{(r-6M)(r-3M)|r+6M|^{\frac{1}{2}}}{r^{\frac{7}{2}} \Omega }-  \frac{r+3M}{r^{\frac{1}{2}}|r+6M|^{\frac{3}{2}} \Omega }=\frac{(r-6M)(r-3M)|r+6M|^2-(r+3M)r^3}{r^{\frac{7}{2}} |r+6M|^{\frac{3}{2}} \Omega}   ,  $$
and the relation $$(r-6M)(r-3M)|r+6M|^2-(r+3M)r^3=-54M^2r^2-108M^3r+54 \cdot 12 M^4.$$
\end{proof}

Let us also consider the projection on $T \mathcal{P}$ parallel to $\partial_{p_t}$ of the symplectic gradient $\T_{\varphi_+}$ associated to the weight $\varphi_+$. It is given by
 $$ V_-  =- \frac{|r+6M|^{\frac{1}{2}}}{r^{\frac{1}{2}}\Omega}(r-3M) \partial_t+\frac{r}{\Omega} \partial_{r^*}-\bigg( \frac{r-3M}{r\Omega}p_{r^*}-\frac{r^{\frac{1}{2}}}{|r+6M|^{\frac{1}{2}} \Omega}p_t \bigg) \partial_{p_{r^*}}.$$
Even if we will not use this vector field in the proof of Theorem \ref{thm_main_intro}, we state the next two propositions for completeness. 

\begin{proposition}\label{ProVplus}
We have
\begin{align*}
[\T_g,V_-] &=   \frac{|p_t|}{r^{\frac{1}{2}}|r+6M|^{\frac{1}{2}}\Omega^2} V_-+ \frac{r+3M}{r^{\frac{3}{2}}|r+6M|^{\frac{3}{2}}}|p_t|\varphi_+ \partial_{p_{r^*}}  + 2\Omega \T_g.
\end{align*}
\end{proposition}
\begin{proof}
The computations are similar to the ones carried out during the proof of Proposition \ref{ProComVminus}.
\end{proof}
Finally, we obtain some identities relating the radial part of the Vlasov operator $\pmb{R}=\T_g-r^{-2}\T_Q$, the radial scaling vector field $\pmb{L}$ and $V_{\pm}$.

\begin{proposition}\label{ProvarphiV}
The following relations hold,
\begin{align*}
 \varphi_-V_-+\varphi_+ V_+ &=  2r^2 \pmb{R} +54M^2p_t\partial_t , \\
 \varphi_-V_--\varphi_+ V_+ &=  \frac{54M^2p_t}{r^{\frac{1}{2}}|r+6M|^{\frac{1}{2}}} \pmb{L} .
\end{align*}
\end{proposition}
\begin{proof}
We have
\begin{align*}
 \varphi_+ V_++\varphi_-V_- &= \frac{|r+6M|^{\frac{1}{2}}}{r^{\frac{1}{2}}\Omega}(r-3M)(\varphi_+-\varphi_-) \partial_t+\frac{r}{\Omega} (\varphi_++\varphi_-)\partial_{r^*} \\
 & \quad -\bigg( \frac{r-3M}{r\Omega}p_{r^*}(\varphi_++\varphi_-)+\frac{r^{\frac{1}{2}}}{|r+6M|^{\frac{1}{2}} \Omega}p_t (\varphi_+-\varphi_-) \bigg) \partial_{p_{r^*}},
 \end{align*}
and
$$ \varphi_++\varphi_-= 2\frac{r p_{r^*}}{\Omega}, \quad \qquad \varphi_+-\varphi_-=-2\frac{|r+6M|^{\frac{1}{2}}}{r^{\frac{1}{2}}\Omega}(r-3M)p_t.$$
Consequently, we have
\begin{align*}
 \varphi_+ V_++\varphi_-V_- &= -2\frac{r+6M}{r\Omega^2}(r-3M)^2p_t \partial_t+2\frac{r^2}{\Omega^2} p_{r^*}\partial_{r^*}+2\frac{r-3M}{ \Omega^2}(|p_t|^2-|p_{r^*}|^2) \partial_{p_{r^*}}.
\end{align*}
Similarly, one has
$$ \varphi_+ V_+-\varphi_-V_-= 2\frac{r^{\frac{1}{2}}|r+6M|^{\frac{1}{2}}}{\Omega^2}(r-3M)p_{r^*} \partial_t -2\frac{r^{\frac{1}{2}}|r+6M|^{\frac{1}{2}}}{\Omega^2}(r-3M) p_t\partial_{r^*} -2\frac{27M^2}{r^{\frac{1}{2}}|r+6M|^{\frac{1}{2}}}p_{r^*}p_t \partial_{p_{r^*}} .$$
Finally, recall the null-shell relation $\frac{|p_t|^2-|p_{r^*}|^2}{\Omega^2}=\frac{|\slashed{p}|^2}{r^2} $ as well as the definitions \eqref{eq:defR} and \eqref{eq:defL} of $\pmb{R}$ and $\pmb{L}$.
\end{proof}

\subsection{The improved commutation vector field}

The commutation formula for $V_+$ is not completely satisfying because of the error term proportional to $\varphi_- \partial_{p_{r^*}}$. While one can expect the energy flux of $\varphi_- \partial_{p_{r^*}} f$ to be bounded in $L^1(\pi^{-1}(\Sigma_\tau))$, we will avoid controlling a non-degenerate bulk norm for this quantity in view of the discussion of Section \ref{Subsubsecstart} (recall in particular \eqref{eq:heuri}). For this purpose, we modify $V_+$ with a vector field proportional to $\partial_{p_{r^*}}$. The correction will compensate for the bad error term.

\subsubsection{Some properties of the vector field $\varphi_- \partial_{p_{r^*}}$}

It turns out that one obtains a simpler commutation formula by rescaling $\varphi_- \partial_{p_{r^*}}$ by $r\Omega^{-1}$. For this reason, we recall $\pmb{\varphi}_-\coloneqq\Omega^{-1} \varphi_-$.

\begin{proposition}\label{Propartialprstar}
There holds
\begin{align*}
 [ \T_g, r\pmb{\varphi}_- \partial_{p_{r^*}}]&=\frac{|\pmb{\varphi}_-|^2}{ p_t} \partial_t-\frac{\pmb{\varphi}_-}{\Omega} V_+ .
 \end{align*}
\end{proposition}
\begin{remark}\label{Rkregu}
Contrary to $r\pmb{\varphi}_- p_t \partial_{p_{r^*}}$, the vector field $r\pmb{\varphi}_- \partial_{p_{r^*}}$ is not regular up to the future event horizon.
\end{remark}

\begin{proof}
By Lemma \ref{LemComdrdpr} and Proposition \ref{ProdefVminus}, giving the expression of $V_+$ in local coordinates,
$$ [ \T_g, p_t\partial_{p_{r^*}}]=\frac{\varphi_-}{r\Omega} \partial_t-\frac{p_t}{r\Omega} V_+-\frac{p_t}{r\Omega}\bigg( \frac{r-3M}{r\Omega}p_{r^*}+\frac{r^{\frac{1}{2}}}{|r+6M|^{\frac{1}{2}} \Omega}p_t \bigg) \partial_{p_{r^*}}.$$
We then obtain from Lemma \ref{Lemphiminus} that
$$ [ \T_g,\varphi_- \partial_{p_{r^*}}]=\frac{|\varphi_-|^2}{r\Omega p_t} \partial_t-\frac{\varphi_-}{r\Omega} V_+-\frac{(r-3M)p_{r^*}}{r^2\Omega^2}\varphi_- \partial_{p_{r^*}}.$$
Finally, we get the result from $\pmb{\varphi}_-=\Omega^{-1}\varphi_-$ and
$$\partial_{r}(r\Omega^{-1})=\Omega^{-1}-\frac{M}{r}\Omega^{-3}=\frac{r-3M}{r^2 \Omega^2}r\Omega^{-1}.$$ 
\end{proof}

We note that one cannot control a non-degenerate bulk integral of $r\pmb{\varphi}_- p_N p_t\partial_{p_{r^*}}f$ near $\H$ by simply using the previous commutation formula and an energy estimate. Indeed, in view of Remark \ref{Rkregu}, one cannot use the redshift weight in order to generate a good error term because of regularity issues. To circumvent this difficulty, we will use the next estimate.

\begin{lemma}\label{LemBoundpartialpr}
Let $f$ be a solution the massless Vlasov equation. For $r < 2.7M$, there holds
$$ \bigg| \frac{2|p_u|}{\Omega^2} p_t \partial_{p_{r^*}} f-\frac{rp_{r^*}}{(r-3M)} \Omega^{-1} V_+ f \bigg| \lesssim   \big| p_N \partial_t f \big|+|\slashed{p}|\big|  \T_{|\slashed{p}|} f \big|+ \big(|p_t|^2+ |\slashed{p}|^2\big)  \big|  \partial_{p_{r^*}} f \big|.$$
\end{lemma}
\begin{proof}
As $p_t=2p_v-p_{r^*}$, in view of the null-shell relation $4r^2|p_u||p_v|=\Omega^2 |\slashed{p}|^2$, it suffices to estimate $\frac{2|p_u|}{\Omega^2} p_{r^*} \partial_{p_{r^*}} f$. We write
\begin{align*}
\frac{2p_u}{\Omega^2}\frac{r-3M}{r} p_{r^*} \partial_{p_{r^*}} &= -\frac{r-3M}{r\Omega^2}|p_{r^*}|^2 \partial_{p_{r^*}}-\frac{r^{\frac{1}{2}}p_{r^*}p_t}{|r+6M|^{\frac{1}{2}} \Omega^2} \partial_{p_{r^*}} + \frac{r^{\frac{1}{2}} p_{r^*}p_t}{|r+6M|^{\frac{1}{2}} \Omega^2}  \partial_{p_{r^*}}+\frac{r-3M}{r\Omega^2}p_{r^*}p_t \partial_{p_{r^*}}.
\end{align*}
We remark now that 
\begin{align*}
 \frac{r^{\frac{1}{2}} p_{r^*}p_t}{|r+6M|^{\frac{1}{2}} \Omega^2}  \partial_{p_{r^*}}+\frac{r-3M}{r\Omega^2}p_{r^*}p_t \partial_{p_{r^*}}=\frac{27M^2}{|r+6M|^{\frac{1}{2}} (r^{\frac{3}{2}}-|r+6M|^{\frac{1}{2}} (r-3M))}p_tp_{r^*} \partial_{p_{r^*}},
\end{align*}
and
\begin{align*}
-\frac{r-3M}{r\Omega^2}|p_{r^*}|^2 \partial_{p_{r^*}}-\frac{r^{\frac{1}{2}}p_{r^*}p_t}{|r+6M|^{\frac{1}{2}} \Omega^2} \partial_{p_{r^*}} = \frac{p_{r^*}}{\Omega} V_+-p_{r^*}\frac{|r+6M|^{\frac{1}{2}}}{r^{\frac{1}{2}}\Omega^2}(r-3M)  \partial_t-p_{r^*}\frac{r}{\Omega^2} \partial_{r^*}.
\end{align*}
Next, we use
\begin{align*}
\frac{r}{\Omega^2} p_{r^*}\partial_{r^*}= r \big( \T_g-r^{-2} \T_Q \big)+\frac{r}{\Omega^2}p_t \partial_t-\frac{r-3M}{r^3}|\slashed{p}|^2 \partial_{p_{r^*}},
\end{align*}
and we observe that
\begin{align*}
\frac{|r+6M|^{\frac{1}{2}}}{r^{\frac{1}{2}}\Omega^2}(r-3M) p_{r^*} \partial_t+\frac{r}{\Omega^2}p_t \partial_t =\frac{2rp_u}{\Omega^2} \partial_t+\frac{27M^2r}{r^{\frac{1}{2}}(r^{\frac{3}{2}}-|r+6M|^{\frac{1}{2}}(r-3M))} p_{r^*} \partial_t.
\end{align*}
The last five equalities imply the result since $\T_g (f)=0$ and $\T_Q=|\slashed{p}|\T_{|\slashed{p}|}$.
\end{proof}

\subsubsection{The modified vector field $V_{\, +}^{\mathrm{mod}}$}\label{subsub_modified} 

Let us now define the modification of the vector field $V_+$ that we will use to perform integrated energy estimates without relative degeneration.

\begin{definition}\label{DefVminusmod}
Let $V_{\, +}^{\mathrm{mod}} \in \Gamma(T \mathcal{P})$ be the vector field 
$$ V_{\, +}^{\mathrm{mod}} \coloneqq V_++\Omega \Phi r \pmb{\varphi}_- \partial_{p_{r^*}},$$
where $\Phi$ is uniquely determined by
$$ \T_g (\Omega \Phi)+ \frac{|p_t|}{r^{\frac{1}{2}}|r+6M|^{\frac{1}{2}}\Omega^2} (\Omega \Phi)   =  \frac{r+3M}{r^{\frac{5}{2}}|r+6M|^{\frac{3}{2}}}|p_t|\Omega, \qquad \quad (\Omega\Phi)\vert_{t^*=0}=0.$$
\end{definition}

\begin{proposition}\label{ProComVminusmod}
There holds
\begin{align*}
\big[\T_g , V_{\, +}^{\mathrm{mod}}\big] &=  - \frac{|p_t|}{r^{\frac{1}{2}}|r+6M|^{\frac{1}{2}}\Omega^2} V_{\, +}^{\mathrm{mod}}+ \frac{\Omega\Phi |\pmb{\varphi}_-|^2}{ p_t} \partial_t-  \Phi \pmb{\varphi}_- V_+  + 2\Omega \T_g.
\end{align*}
\end{proposition}
\begin{proof}
We write first
$$ \big[\T_g , V_{\, +}^{\mathrm{mod}}\big]  = [\T_g , V_+] +\Omega \Phi [\T_g, r \pmb{\varphi}_- \partial_{p_{r^*}} ]+ \T_g(\Omega\Phi) r \pmb{\varphi}_- \partial_{p_{r^*}}.$$
Next, recall the commutation formulae in Propositions \ref{ProComVminus} and \ref{Propartialprstar}. Finally, we use the definition of $\Phi$. 
\end{proof}

In order to use the vector field $V_{\, +}^{\mathrm{mod}}$, we have to estimate $\Phi$ as well as the derivative $\pmb{\varphi}_- \partial_{p_{r^*}} \Phi$. 

\begin{proposition}\label{ProboundPsi}
There holds
$$ \sup_{  \pi^{-1} (\{t^* \geq 0 \})} \,\frac{|p_N|+|r \pmb{\varphi}_-|}{|p_t|} \big| r\Phi \big|+\frac{|r\pmb{\varphi}_-|^{\frac{3}{4}}}{|p_t|^{\frac{3}{4}}} \big| r\Phi \big|+\big|\pmb{\varphi}_- \partial_{p_{r^*}} \Phi \big| < +\infty.$$
\end{proposition}
The proof of Proposition \ref{ProboundPsi} is performed in Appendix \ref{Append} (see Proposition \ref{ProboundPsibis} and Remark \ref{RkboundPsi}).

\subsection{The rescaled trapping vector fields}

In order to exploit the redshift effect and obtain a stronger control of massless Vlasov fields near $\H$, we will work with $\Omega^{-1}V_+$ which does not vanish there. This vector field is smooth up to the future event horizon.

\begin{lemma}\label{LemComVminusregularised}
Let $\pmb{V}_{\! +}\coloneqq\Omega^{-1}V_+$ and recall the definition of $\pmb{a}$ from Lemma \ref{Lemphiminus}. There holds
\begin{align*}
\big[\T_g,\pmb{V}_{\! +}\big] &= - \pmb{a}(r,p_{r^*},p_t) \pmb{V}_{\! +}- \frac{r+3M}{r^{\frac{3}{2}}|r+6M|^{\frac{3}{2}}}\pmb{\varphi}_- |p_t|\partial_{p_{r^*}}  + 2 \T_g.
\end{align*}
\end{lemma}
\begin{proof}
We have
$$ \big[\T_g,\pmb{V}_{\! +}\big]= \Omega^{-1}[\T_g,V_+]+p_{r^*} \partial_r \big( \Omega^{-1} \big) V_+.$$
To get the result, it remains to use the commutation formula of Proposition \ref{ProComVminus} and to perform the same computations as in the proof of Lemma \ref{Lemphiminus}, which lead to
\begin{equation}\label{eq:idbarbolda}
\pmb{a}(r,p_{r^*},p_t)= \frac{|p_t|}{r^{\frac{1}{2}}|r+6M|^{\frac{1}{2}}\Omega^2} -p_{r^*} \partial_r \big( \Omega^{-1} \big).
\end{equation}
\end{proof}

Similarly, we can work with a regularised version of $V_{\, +}^{\mathrm{mod}}$, allowing for a better control of $f$ near $\mathcal{H}^+$.
\begin{corollary}\label{corVminusmod}
Let $\pmb{V}_+^{  \mathrm{mod}}\coloneqq \Omega^{-1} V_{\, +}^{\mathrm{mod}}$, and let $f$ be a solution to the massless Vlasov equation. Then,
\begin{align*}
 \T_g \big( \big| \pmb{V}_{ +}^{ \mathrm{mod}} f \big| \big)  + \pmb{a}(r,p_{r^*},p_t) \big| \pmb{V}_+^{  \mathrm{mod}} f \big| \lesssim \frac{|\pmb{\varphi}_-|}{r^2 } \big|  \partial_t f \big|+\frac{|\pmb{\varphi}_-|^{\frac{1}{4}}|p_t|^{\frac{3}{4}}}{r^{\frac{7}{4}} } \big| \pmb{V}_{\! +} f|  .
\end{align*}
\end{corollary}
\begin{proof}
We use the commutation formula in Proposition \ref{ProComVminusmod} to get
\begin{align*}
 \T_g \big( \big| V_{\, +}^{\mathrm{mod}} f \big| \big)  + \frac{|p_t|}{r^{\frac{1}{2}}|r+6M|^{\frac{1}{2}}\Omega^2} \big| V_{\, +}^{\mathrm{mod}} f \big| \lesssim \Omega\frac{ |r\pmb{\varphi}_-||r\Phi|}{|p_t|} \cdot \frac{|\pmb{\varphi}_-|}{r^2} \big| \partial_t f \big|+  \frac{|r\pmb{\varphi}_-|^{\frac{3}{4}}|r\Phi|}{|p_t|^{\frac{3}{4}}}\cdot\frac{|\pmb{\varphi}_-|^{\frac{1}{4}}|p_t|^{\frac{3}{4}}}{r^{\frac{7}{4}} } \big| V_+ f \big|.
\end{align*}
We conclude the proof by using \eqref{eq:idbarbolda} and Proposition \ref{ProboundPsi}.
\end{proof}

It turns out that in order to control $|p_t|^{\frac{3}{4}}|r^{-1} \pmb{\varphi}_-|^{\frac{1}{4}} | \pmb{V}_{\! +} f|$, we will need to study a well-chosen derivative of $f$ collinear to $\partial_{p_{r^*}}$. 

\begin{definition}\label{DefDpr}
Let $\mathrm{D}_{p_{r^*}} \! \in \Gamma (T \mathcal{P})$ be the vector field given by
$$
\mathrm{D}_{p_{r^*}} \coloneqq \Omega^{\frac{1}{2}}|p_t|^{\frac{3}{4}} |r^{-1} \pmb{\varphi}_-|^{\frac{5}{4}} \partial_{p_{r^*}}.
$$
\end{definition}

\begin{remark}
We have $\Omega^{\frac{1}{2}}|p_t|^{\frac{3}{4}} |r^{-1} \pmb{\varphi}_-|^{\frac{5}{4}} \lesssim |p_t||p_N|$, so that $\mathrm{D}_{p_{r^*}}$ is regular up to $\H$.
\end{remark}

We conclude this section with a consequence of Lemma \ref{LemBoundpartialpr}.

\begin{corollary}\label{CorBoundpartialpr}
Let $f$ be a solution the massless Vlasov equation. For $r < 2.7M$, we have
$$  \bigg| \frac{|p_u|}{\Omega^2} \mathrm{D}_{p_{r^*}} f \bigg| \lesssim |p_t|^{\frac{3}{4}}|\Omega^2 \pmb{\varphi}_-|^{\frac{1}{4}} \big| \pmb{\varphi}_- \pmb{V}_{\! +} f \big| + \big| p_N \pmb{\varphi}_-\partial_t f \big|+\big| p_N \pmb{\varphi}_- \T_{|\slashed{p}|} f \big|+ \bigg(|p_t|+ \frac{|\slashed{p}|^2}{|p_t|}\bigg)  \big|  \mathrm{D}_{p_{r^*}} f \big|.$$
\end{corollary}
\begin{proof}
We use the estimate of Lemma \ref{LemBoundpartialpr}, multiplied by $\Omega^{\frac{1}{2}}|p_t|^{-\frac{1}{4}} |r^{-1} \pmb{\varphi}_-|^{\frac{5}{4}}$. We conclude the proof by using $\Omega^{\frac{1}{2}} |r^{-1} \pmb{\varphi}_-|^{\frac{1}{4}} \leq 3|p_t|^{\frac{1}{4}}$.
\end{proof}

\section{Integrated energy decay estimates without relative degeneration}\label{SecILEDwrd}

We introduce a well-chosen energy flux $\mathcal{E}[f]$ to show decay for massless Vlasov fields on Schwarzschild. Using this energy flux, we prove an integrated local energy decay estimates without relative degeneration.

\subsection{Energy norms and statement of the main results} 

Motivated by the discussion of Section \ref{Subsecideas} (recall in particular Steps $3$ and $4$), we introduce the following first order energy fluxes.

\begin{definition}
 We set, for any distribution function $g : \mathcal{P} \to \R$, the first order energy flux
\begin{align*}
 \mathcal{F}[g]& \coloneqq \mathbb{E} \big[  \pmb{\varphi}_- \partial_t g \big] +\mathbb{E} \big[  \pmb{\varphi}_- \T_{|\slashed{p}|} g \big] + \mathbb{E} \big[ p_N \, \pmb{V}_+^{  \mathrm{mod}} g \big]+\mathbb{E} \big[ |p_N|^{\frac{3}{4}} |r^{-1} \pmb{\varphi}_-|^{\frac{1}{4}}  \pmb{V}_{\! +} g \big]  +\mathbb{E} \big[\mathrm{D}_{p_{r^*}} g \big] .
\end{align*}
Let further
$$\mathcal{E}[g]\coloneqq\mathbb{E}[p_Ng]+\mathcal{F}[g] .$$
\end{definition}

The main result proved in this section can be stated as follows.

\begin{proposition}\label{ProILEDwrd2}
Let $f $ be a solution to the massless Vlasov equation such that $\mathcal{E}[f](0) <+ \infty$. There exists $C>0$, depending only on $M$, such that
$$ \sup_{\tau \geq 0} \,  \mathcal{E}[f](\tau ) +\int_{\tau=0}^{+\infty} \mathcal{E}\big[r^{-1}\log^{-2}(2+r)f \big](\tau ) \dr \tau \leq C \mathcal{E}[f](0 ).$$
\end{proposition}

For the proof of this proposition we will first show a similar result for the weaker norm $\mathcal{F}[f]$.

\begin{proposition}\label{ProILEDwrd}
Let $f $ be a solution to the massless Vlasov equation such that $\mathcal{F}[f](0) <+ \infty$. There exists $C>0$, depending only on $M$, such that
$$ \sup_{\tau \geq 0} \, \mathcal{F}[f](\tau ) +\int_{\tau=0}^{+\infty} \mathcal{F}\big[r^{-1}\log^{-2}(2+r)f \big](\tau ) \dr \tau \leq C \mathcal{F}[f](0).$$
\end{proposition}

We fix, for the remainder of this section \ref{SecILEDwrd}, a solution $f$ to $\T_g(f)=0$ verifying $\mathcal{F}[f](0)<+\infty$. Since $\T_g(\partial_t f)=\T_g(\T_{|\slashed{p}|}f)=0$, we get from Proposition \ref{ProILED000}, $|p_{n_{\Sigma_\tau}}| \lesssim |p_N|$ and \eqref{eq:coarea} that, for all $\tau \geq 0$,
\begin{equation}\label{kevatalenn:0} 
\mathbb{E}\big[ \pmb{\varphi}_- \partial_t f \big](\tau )+\mathbb{E}\big[ \pmb{\varphi}_- \T_{|\slashed{p}|} f \big](\tau ) +  \int_{\tau=0}^{+\infty} \! \mathbb{E}\big[ r^{-1} \pmb{\varphi}_- \partial_t f \big](\tau )+\mathbb{E}\big[r^{-1} \pmb{\varphi}_- \T_{|\slashed{p}|} f \big](\tau ) \dr \tau \lesssim \mathcal{F}[f](0 ). \color{white} \square \quad \color{black}
\end{equation}
Consequently, we only need to control the last three fluxes constituting $\mathcal{F}[f]$. We will proceed as follows.
\begin{enumerate}[label = (\alph*)]
\item In Section \ref{Subsec72}, we control the derivatives of $f$ in a favourable region that can be treated independently of the rest of the null-shell. This is explained by the fact that no future-directed null geodesic can enter this domain.
\item In Section \ref{Subsec73}, we will control the two degenerate derivatives $ |p_N|^{\frac{3}{4}} |r^{-1} \pmb{\varphi}_-|^{\frac{1}{4}}  \pmb{V}_{\! +} f$ and $\mathrm{D}_{p_{r^*}} f$.
\item In Section \ref{Subsec74}, we conclude the proof of Proposition \ref{ProILEDwrd} by estimating $p_N\pmb{V}_{\, +}^{\mathrm{mod}}f$.
\item In Section \ref{Subsec75}, we control $f$ using the norm $\mathcal{E}[f]$. This will show Proposition \ref{ProILEDwrd2}.
\end{enumerate}

\subsection{The far-away incoming particles case}\label{Subsec72}

We begin controlling the Vlasov field $f$ in the region where $r \gg 2M$ and $p_{r^*} <0$. In this region, motivated by the Minkowskian case, we will use the vector field
$$ S \coloneqq \frac{r}{\Omega^2}\partial_t+\frac{r}{\Omega^2}\partial_{r^*}-p_{r^*} \partial_{p_{r^*}}.$$
We will prove the next estimate.

\begin{proposition}\label{Proestiincoming}
Let $R \geq 832M$. There exists $C >0$ such that for all $\tau \geq 0$, we have
\begin{align*}
& \mathbb{E} \Big[ p_t  |Sf| \, \mathds{1}_{r \geq R} \, \mathds{1}_{p_{r^*} \leq 0} \Big](\tau)+\frac{1}{4}\mathbb{E} \Big[ p_t  \big|p_t \partial_{p_{r^*}} f \big| \, \mathds{1}_{r \geq R} \, \mathds{1}_{p_{r^*} \leq 0} \Big](\tau) \\
& + \int_{\pi^{-1}(\mathcal{R}) \cap \{ r\geq R, \, p_{r^*} \leq 0 \} } \bigg( \frac{|p_{r^*}|^2}{(r+2)\log^2(2+r)}+\frac{|\slashed{p}|^2}{4r^3} \bigg) \Big(\big|S f \big| + \frac{1}{4} \big|p_t \partial_{p_{r^*}} f \big| \Big) \dr \mu_{\mathcal{P}} \\
& \qquad \leq 4\mathbb{E} \Big[ p_t  |Sf| \, \mathds{1}_{r \geq R} \, \mathds{1}_{p_{r^*} \leq 0}  \Big](0)+\mathbb{E} \Big[ p_t  \big| p_t \partial_{p_{r^*}} f \big| \, \mathds{1}_{r \geq R} \, \mathds{1}_{p_{r^*} \leq 0} \Big](0) + C\mathbb{E} \big[  \pmb{\varphi}_- \partial_t f \big](0)+  C\mathbb{E} \big[ \pmb{\varphi}_- \T_{|\slashed{p}|} f \big](0).
\end{align*}
\end{proposition}
\begin{remark}\label{RkProestiincoming}
This energy estimate allows to control $|p_N|^{\frac{3}{4}} |r^{-1} \pmb{\varphi}_-|^{\frac{1}{4}}  \pmb{V}_{\! +} f$ and $D_{p_{r^*}} f$ in $\{ r \geq R, \, p_{r^*} \leq 0 \}$. As $R \geq 4M$ and $p_{r^*} \leq 0$, we have $|\pmb{\varphi}_-| \sim r|p_t|$ on this region, so 
$$\big| p_t S f \big|+ |p_t|\big| p_t \partial_{p_{r^*}} f \big| \lesssim  \big| \pmb{ \varphi}_- \partial_t f \big|+  |p_N|^{\frac{3}{4}} |r^{-1} \pmb{\varphi}_-|^{\frac{1}{4}} \big| \pmb{V}_{\! +} f \big|+\big| \mathrm{D}_{p_{r^*}} f \big|$$
and
$$      |p_N|^{\frac{3}{4}} |r^{-1} \pmb{\varphi}_-|^{\frac{1}{4}} \big| \pmb{V}_{\! +} f \big|+ \big|   p_N \pmb{V}_{\, +}^{ \mathrm{mod}} f \big| \lesssim  \big| \pmb{\varphi}_- \partial_t f \big|+ \big| p_t S f \big|+  |p_t|\big| p_t \partial_{p_{r^*}} f \big|, \qquad \qquad \big| \mathrm{D}_{p_{r^*}} f \big| \lesssim |p_t|\big| p_t \partial_{p_{r^*}} f \big|.$$
\end{remark}

To prove Proposition \ref{Proestiincoming}, we will apply the energy estimate in Proposition \ref{Proenergy} to $p_t Sf$ and $|p_t|^2 \partial_{p_{r^*}}f$, both multiplied by suitable cutoff functions. We begin computing two commutators.

\begin{lemma}\label{LemComdrdprbis}
We have
\begin{align*}
\big[\T_g,S\big] &= \bigg(\frac{4(r-3M)p_v^2}{r \Omega^4p_t}-\frac{2p_{r^*}p_v}{ \Omega^2p_t} \bigg)\partial_t     - \frac{3M}{r^4} |\slashed{p}|^2 \partial_{p_{r^*}}+2\T_g , \\
\big[ \T_g, p_t \partial_{p_{r^*}} \big] & =  -\frac{|\slashed{p}|^2\Omega^2}{r^3 p_t} S-\frac{M|\slashed{p}|^2}{r^4 p_t}p_{r^*} \partial_{p_{r^*}}+\frac{|\slashed{p}|^2}{r^2 p_t} \partial_t+\frac{p_{r^*}}{r^2 p_t} \T_Q -\frac{p_{r^*}}{ p_t} \T_g.
\end{align*}
\end{lemma}
\begin{proof}
We recall from Lemma \ref{LemComdrdpr} that
\begin{align*}
[\T_g,\partial_{r^*}] &= \frac{4(r-3M)p_v^2}{r^2 \Omega^2p_t} \partial_t-2\frac{r-3M}{r^2 \Omega^2}p_{r^*}\big(\partial_t+ \partial_{r^*}\big)+ \Omega^2 \frac{r-6M}{r^5} |\slashed{p}|^2 \partial_{p_{r^*}}+2\frac{\Omega^2}{r}\T_g  , \\[2pt]
\big[ \T_g, \partial_{p_{r^*}} \big] & = \frac{2p_v}{ \Omega^2p_t} \partial_t- \frac{1}{\Omega^2}\big( \partial_t+ \partial_{r^*} \big).
\end{align*}
As $ \T_g (r \Omega^{-2})=p_{r^*}\frac{r-4M}{r\Omega^4}$, we have
\begin{align*}
\big[\T_g,S\big] &=p_{r^*}\frac{r-4M}{r\Omega^4}\big(\partial_t+\partial_{r^*} \big) +\frac{r}{\Omega^2} [\T_g,\partial_{r^*}]-\frac{r-3M}{r^4}|\slashed{p}|^2 \partial_{p_{r^*}}-p_{r^*} \big[ \T_g, \partial_{p_{r^*}} \big] ,
\end{align*}
from where we obtain the first relation. For the second one, we use the null-shell relation \eqref{eq:defConsQua} to show 
\begin{align*}
\big[ \T_g, \partial_{p_{r^*}} \big] & =  \frac{p_{r^*}}{\Omega^2 p_t} \partial_t-\frac{|p_{r^*}|^2}{\Omega^2|p_t|^2} \partial_{r^*}-\frac{|\slashed{p}|^2}{r^2 |p_t|^2} \partial_{r^*} \\
& = \frac{p_{r^*}}{\Omega^2 p_t} \partial_t-\frac{|\slashed{p}|^2}{r^2 |p_t|^2} \partial_{r^*}-\frac{p_{r^*}}{ |p_t|^2}\bigg( \T_g+\frac{p_t}{\Omega^2} \partial_t-\frac{1}{r^2} \T_Q-\frac{(r-3M)|\slashed{p}|^2}{r^4} \partial_{p_{r^*}} \bigg).
\end{align*}
It remains to use $\T_g(p_t)=0$ and to rewrite $\partial_{r^*}$ in terms of $S$. 
\end{proof}

To generate good error terms and absorb the problematic ones, we will multiply the derivatives $Sf$ and $p_t \partial_{p_{r^*}}f$ by the weight $\omega(x,p)$ defined by
\begin{equation}\label{eq:defomegafaraway}
 \omega(x,p)\coloneqq 1-\frac{p_{r^*}}{|p_t|}\big(3- \log^{-1}(2+r) \big).
 \end{equation}
We note that $1 \leq \omega(x,p) \leq 4$ on the domain of incoming particles $\{ p_{r^*} \leq 0 \}$. The error terms arising from the cutoff function will have a good sign as well.

\begin{lemma}\label{Lemtruc}
For all $r \geq 6M$, we have
$$  |p_t| \T_g \big( \omega )\leq -\frac{|p_{r^*}|^2}{(2+r) \log^2(2+r)}-\frac{|\slashed{p}|^2}{2r^3} .$$
Let $ \chi \in C^\infty ( \R )$ be an increasing function such that $\chi (s)=0$ for $s \leq 0$, and $\chi(s) =1$ for $s \geq 1$. Then,
$$ \T_g \big[ \chi (r-R)\chi (-p_{r^*} ) \big] \leq 0.$$
\end{lemma}

\begin{proof}
As $\T_g(p_t)=0$, the first inequality follows from $3- \log^{-1}(2+r) \geq 1$, $r-3M \geq \frac{1}{2}r$, and
$$ \T_g \big[ -p_{r^*} \big(3- \log^{-1}(2+r) \big) \big]= -\frac{|p_{r^*}|^2}{(2+r) \log^2(2+r)}-\frac{r-3M}{r^4}|\slashed{p}|^2 \big(3- \log^{-1}(2+r)\big) .$$
For the second one, we have
$$ \T_g \big[ \chi (r-R)\chi (-p_{r^*} ) \big]=p_{r^*} \chi'(r-R)\chi (-p_{r^*})-\frac{r-3M}{r^4} |\slashed{p}|^2 \chi'(-p_{r^*})\chi (r-R). $$ Finally, it remains to use $ \chi' \geq 0$ as well as $r-3M \geq 0$ and $p_{r^*} \leq 0$, when the RHS does not vanish.
\end{proof}

We are now able to show Proposition \ref{Proestiincoming}. 

\begin{proof}[Proof of Proposition \ref{Proestiincoming}]
By applying the energy estimate in Proposition \ref{Proenergy} to the functions
$$\chi \big[ \epsilon^{-1}(r-R) \big] \chi \big(- \epsilon^{-1} p_{r^*} \big) \omega |Sf| , \qquad \qquad \chi \big[ \epsilon^{-1}(r-R) \big] \chi \big(- \epsilon^{-1} p_{r^*} \big)\omega|p_t \partial_{p_{r^*}} f|   ,$$
for $\epsilon >0$, we get from the dominated convergence theorem and Lemma \ref{Lemtruc} that
$$
\mathbb{E} \Big[ \omega p_t |Sf| \, \mathds{1}_{r \geq R} \, \mathds{1}_{p_{r^*} \leq 0} \Big](\tau)  \leq \mathbb{E} \Big[ \omega p_t  |Sf| \, \mathds{1}_{r \geq R} \, \mathds{1}_{p_{r^*} \leq 0} \Big](0)+\int_{\mathcal{R}_0^\tau \cap \{r \geq R, \; p_{r^*} \leq 0 \}}  \T_g(\omega) |p_t Sf| +\omega|p_t| \T_g\big( |Sf| \big) \dr \mu_{\mathcal{P}}
$$
and
\begin{align*}
\mathbb{E} \Big[ \omega p_t  |p_t \partial_{p_{r^*}} f| \, \mathds{1}_{r \geq R} \, \mathds{1}_{p_{r^*} \leq 0} \Big](\tau) & \leq \mathbb{E} \Big[ \omega p_t  |p_t \partial_{p_{r^*}} f|\, \mathds{1}_{r \geq R} \, \mathds{1}_{p_{r^*} \leq 0} \Big](0)\\
&\quad+\int_{\mathcal{R}_0^\tau \cap  \{r \geq R, \; p_{r^*} \leq 0 \}} \T_g(\omega) \big|p_t^2 \partial_{p_{r^*}}f \big| +\omega|p_t| \T_g\big( |p_t \partial_{p_{r^*}}f| \big) \dr \mu_{\mathcal{P}}.
\end{align*}
Applying Lemma \ref{LemComdrdprbis}, we have for $r \geq R$ that
\begin{align*}
|p_t| \T_g \big( |S f | \big) & \leq \frac{6p_tp_v}{ \Omega^4}|\partial_t f|   +\frac{3M}{R} \frac{ |\slashed{p}|^2}{r^3} |p_t \partial_{p_{r^*}} f | \leq 7 |p_t||p_v||\partial_t f|   +\frac{3M}{R} \frac{ |\slashed{p}|^2}{r^3} |p_t \partial_{p_{r^*}} f | , 
\end{align*}
and, as $r^{-1} |\T_Q f| = r^{-1} |\slashed{p}| |\T_{|\slashed{p}|} f| \leq 2 |p_t||\T_{|\slashed{p}|} f|$ for $r \geq R$,
\begin{align*}
 |p_t|\T \big(|p_t \partial_{p_{r^*}} f | \big) & \leq \frac{|\slashed{p}|^2}{r^3 } |Sf|+\frac{M}{R}\frac{|\slashed{p}|^2}{r^3}\big|p_{t} \partial_{p_{r^*}}f |+\frac{|\slashed{p}|^2}{r^2 } | \partial_t f|+2\frac{|p_{t}|^2}{r } \big|\T_{|\slashed{p}|} f \big|.
\end{align*}
By the null-shell relation, we have $\frac{|\slashed{p}|^2}{r^2} \leq 6 |p_t||p_v|$. Hence, as $\omega \leq 4$, we get for $r \geq R$,
\begin{align*}
 & \omega|p_t| \T_g\big( |Sf| \big)+\frac{\omega}{16} |p_t| \T_g\big( |p_t \partial_{p_{r^*}}f| \big)  \leq    \frac{13M}{R} \frac{ |\slashed{p}|^2}{r^3} |p_t \partial_{p_{r^*}} f |+\frac{|\slashed{p}|^2}{4r^3 } |Sf|+30 |p_t||p_v||\partial_t f| +\frac{|p_{t}|^2}{2r } \big|\T_{|\slashed{p}|} f \big|.
 \end{align*}
If $R\geq 2M$ is such that $\frac{13M}{R} \leq \frac{1}{64}$, we have from the upper bound for $\T_g(\omega)$ in Lemma \ref{Lemtruc} that
 \begin{align} \nonumber
& \T_g(\omega) |p_t Sf| +\omega|p_t| \T_g\big( |Sf| \big)+ \frac{1}{16}\T_g(\omega) |p_t| \big|p_t \partial_{p_{r^*}}f \big| +\frac{\omega}{16}|p_t| \T_g\big( |p_t \partial_{p_{r^*}}f| \big) \\ \label{eq:forRk641}
&\qquad\quad   \leq -\bigg( \frac{|p_{r^*}|^2}{(r+2)\log^2(2+r)}+\frac{|\slashed{p}|^2}{4r^3} \bigg) \Big[\big|S f \big| +  \frac{1}{16}\big|p_t \partial_{p_{r^*}} f \big| \Big]+30 |p_t||p_v||\partial_t f| +\frac{|p_{t}|^2}{2r } \big|\T_{|\slashed{p}|} f \big|.
 \end{align}
Proposition \ref{Proestiincoming} then follows from the integrated energy estimates satisfied by $\partial_t f$ and $\T_{|\slashed{p}|} f$. We recall \eqref{kevatalenn:0} and the estimate $|\pmb{\varphi}_-| \gtrsim r |p_t|$ in the region being considered.
\end{proof}

\begin{remark}
We note that the region of outgoing particles near $\mathcal{H}^+$, corresponding to $\{ r \leq R_2, \; p_{r^*} >0 \}$ for some $R_2 \in (2M,2.5M)$, can also be treated independently from the rest of the null-shell. This will however not be required for our purposes.
\end{remark}

\subsection{Boundedness and ILED without relative degeneracy for degenerate derivatives}\label{Subsec73}

Let us prove the following energy estimates.

\begin{proposition}\label{ProestivarphiminusVminus}
For all $\tau \geq 0$, we have
\begin{align*}
& \mathbb{E} \Big[   |p_N|^{\frac{3}{4}} |r^{-1} \pmb{\varphi}_-|^{\frac{1}{4}}  \pmb{V}_{\! +}  f \Big](\tau)+ \int_{\pi^{-1}(\mathcal{R} )\setminus \{ p_{r^*} \leq 0, \, r \geq R \} } \frac{|p_N|}{r} \Big| |p_N|^{\frac{3}{4}} |r^{-1} \pmb{\varphi}_-|^{\frac{1}{4}}   \pmb{V}_{\! +} f \Big| \dr \mu_{\mathcal{P}} \lesssim \mathcal{F}[f](0),
\end{align*}
and
\begin{align*}
& \mathbb{E} \Big[   \mathrm{D}_{p_{r^*}} f  \Big](\tau)+ \int_{\pi^{-1}(\mathcal{R})\setminus \{ p_{r^*} \leq 0, \, r \geq R \} } \frac{|p_t|}{r\Omega^2}   \big|\mathrm{D}_{p_{r^*}}  f  \big|  \dr \mu_{\mathcal{P}}\lesssim \mathcal{F}[f](0).
\end{align*}
\end{proposition}
\begin{remark}
Note that $|p_t|\Omega^{-2} \gtrsim |p_N|$.
\end{remark} 

In order to lighten the presentation and until the end of this section \ref{Subsec73}, we introduce the notation
$$ \mathrm{D}_{\pmb{V}_{\!+}} :=  |p_N|^{\frac{3}{4}} |r^{-1} \pmb{\varphi}_-|^{\frac{1}{4}}  \pmb{V}_{\! +}  .$$
For the purpose of controlling $\mathrm{D}_{\pmb{V}_{\!+}} f$ and $\mathrm{D}_{p_{r^*}} f$ in $L^1(\pi^{-1}(\Sigma_\tau))$, we begin computing the following identities. 

\begin{lemma}\label{Propartialprstarbisbis}
Let $f$ be a solution to the massless Vlasov equation. Then, there hold
\begin{align*}
\T_g \big(r^{\frac{1}{4}}\mathrm{D}_{\pmb{V}_{\!+}} f \big) & =\underbrace{-\bigg( \frac{5}{4}\pmb{a}(r,p_{r^*},p_t)-\frac{3\T_g(|p_N|)}{4|p_N|} \bigg) r^{\frac{1}{4}}\mathrm{D}_{\pmb{V}_{\!+}}f}_{\text{error term with a good sign}} - \underbrace{\frac{(r+3M)|p_N|^{\frac{3}{4}}|p_t|^{\frac{1}{4}}}{r^{\frac{5}{2}}|r+6M|^{\frac{3}{2}} \Omega^{\frac{1}{2}}} r^{\frac{9}{4}} \mathrm{D}_{p_{r^*}}f}_{\text{bad error term}} , \\[6pt]
 \T_g \big(r^{\frac{9}{4}}\mathrm{D}_{p_{r^*}} f \big)& = \underbrace{-\frac{1}{4}\bigg(\frac{(r^2+2Mr+3M^2)|p_t|}{r|r+6M|^{\frac{1}{2}}(r^{\frac{3}{2}}+M|r+6M|^{\frac{1}{2}}) }+\frac{2M|p_v|}{r^2 \Omega^2} \bigg) r^{\frac{9}{4}}\mathrm{D}_{p_{r^*}} f}_{\text{error term with a good sign}} -\underbrace{\Omega^{\frac{1}{2}}\pmb{\varphi}_-\frac{|p_t|^{\frac{3}{4}}}{|p_N|^{\frac{3}{4}}} r^{\frac{1}{4}}\mathrm{D}_{\pmb{V}_{\!+}} f}_{\text{bad error term}} \\
 & \quad +\frac{ |\Omega^2\pmb{\varphi}_-|^{\frac{1}{4}}}{ |p_t|^{\frac{1}{4}}} |\pmb{\varphi}_-|^2 \partial_tf .
 \end{align*}
\end{lemma} 
 \begin{proof}
Recall from Definition \ref{DefDpr} that $r^{\frac{9}{4}} \mathrm{D}_{p_{r^*}} = |p_t|^{\frac{3}{4}} |\Omega^2 \pmb{\varphi}_-|^{\frac{1}{4}}  r \pmb{\varphi}_- \partial_{p_{r^*}}$. For the first relation, we write 
 \begin{align*}
\T_g \big(|p_N|^{\frac{3}{4}}|\pmb{\varphi}_-|^{\frac{1}{4}}\pmb{V}_{\! +} f \big) & =\T_g \big( |p_N|^{\frac{3}{4}} \big)|\pmb{\varphi}_-|^{\frac{1}{4}}\pmb{V}_{\! +} f+\T_g \big( |\pmb{\varphi}_-|^{\frac{1}{4}} \big) |p_N|^{\frac{3}{4}} \pmb{V}_{\! +} f + |p_N|^{\frac{3}{4}} |\pmb{\varphi}_-|^{\frac{1}{4}} \big[\T_g,\pmb{V}_{\! +} \big](f)
\end{align*}
and we apply Lemmata \ref{Lemphiminus} and \ref{LemComVminusregularised}. For the second one, we have
$$ \T_g \big(r^{\frac{9}{4}} \mathrm{D}_{p_{r^*}} f \big) = \T_g \big(|\Omega^2 \pmb{\varphi}_-|^{\frac{1}{4}}  \big) |p_t|^{\frac{3}{4}}  r \pmb{\varphi}_- \partial_{p_{r^*}}f + |p_t|^{\frac{3}{4}} |\Omega^2 \pmb{\varphi}_-|^{\frac{1}{4}} \big[\T_g , r \pmb{\varphi}_- \partial_{p_{r^*}} \big](f) .$$
We deal with the last term on the RHS by using the commutation formula in Proposition \ref{Propartialprstar}. For the first term, we use Lemma \ref{Lemphiminus} and $|p_u|-p_{r^*}=|p_v|$, which provide
\begin{equation}\label{eq:phitilde}
\hspace{-4mm} \T_g(\Omega^2 \pmb{\varphi}_-)= \Omega^2\T_g(\pmb{\varphi}_-)+\frac{2Mp_{r^*}}{r^2 \Omega^2}\Omega^2 \pmb{\varphi}_-= -\frac{(r^2+2Mr+3M^2)|p_t|}{r|r+6M|^{\frac{1}{2}}(r^{\frac{3}{2}}+M|r+6M|^{\frac{1}{2}}) } \Omega^2 \pmb{\varphi}_--\frac{2M|p_v|}{r^2 \Omega^2} \Omega^2 \pmb{\varphi}_- .
\end{equation}
\end{proof} 

The next remarks will be crucial for us.
\begin{itemize}
\item The vector fields $r^{\frac{1}{4}}\mathrm{D}_{\pmb{V}_{\!+}}$ and $r^{\frac{9}{4}} \mathrm{D}_{p_{r^*}}$ carry too strong $r$-weights. For this reason, we rescale them by using the well-chosen weight functions $\overline{\omega}_{1/4}$ and $\overline{\omega}_{9/4}$ which are introduced in Definition \ref{Defweight}.
\item The system of commuted equations is not triangular. When performing the energy estimates, we will have to check that the second terms in the relations in Lemma \ref{Propartialprstarbisbis} can be absorbed by the error terms with a good sign.
\item The good error term in the second equation does not allow to control the bulk integral of $p_u\Omega^{-2} \mathrm{D}_{p_{r^*}}f$ near $\H$. We will circumvent this issue by applying Corollary \ref{CorBoundpartialpr}, which will not provide new problematic error terms.
\item We will crucially use that we have already controlled the derivatives $\mathrm{D}_{\pmb{V}_{\!+}} f$ and $ \mathrm{D}_{p_{r^*}}f$ in the region $\{ r \geq R, \; p_{r^*} <0 \}$.
\end{itemize}

Then, in order to deal with the error terms, we will proceed as follows:
\begin{enumerate}[label = (\alph*)]
\item The bad error term in the first equation can be bounded by a (large) multiple of the good one in the second equation and $|p_u\Omega^{-2} \mathrm{D}_{p_{r^*}}f|$.
\item  Then, we need to bound $r^{-2}\Omega^{\frac{1}{2}}\pmb{\varphi}_-|p_t|^{3/4}|p_N|^{-3/4} \mathrm{D}_{\pmb{V}_{\!+}} f$ by a small multiple of the good error term in the first equation. The strategy is the following.
\begin{itemize}
\item We will generate additional good error terms degenerating at the photon sphere, the future event horizon and spatial infinity. The idea will be to use the weight function $\z^A$, for $A>0$ large enough, where $\z$ is introduced in Definition \ref{Defz} and studied in Lemma \ref{LemPropz}.
\item As $\frac{|p_t|}{|p_N|}$ degenerates near $\H\!$, we will prove $\pmb{a}(r,p_{r^*},p_t)+A|p_t| \gg \Omega^{\frac{1}{2}}|\pmb{\varphi}_-|\frac{|p_t|^{3/4}}{|p_N|^{3/4}}$ for $r \sim 2M$.
\item Exploiting that $\pmb{\varphi}_-$ degenerates at trapping, we will also be able to absorb the error term near $r \sim 3M$ by making use of $\z^A$.
\item Finally, since $\pmb{\varphi}_- \sim rp_v$ is well behaved for $r \gg 1$ and $p_{r^*} \geq 0$, we will also be able to deal with a neighborhood of spatial infinity $\{ t <+\infty, \, r=+\infty \}$.
\end{itemize}
\end{enumerate}

We start by defining the weights that allow us to compensate for the $r^{\frac{1}{4}}$-growth of $|\pmb{\varphi}_-|^{\frac{1}{4}}$ and the $r^{\frac{9}{4}}$-growth of $r|\pmb{\varphi}_-|^{\frac{5}{4}}$.

\begin{definition}\label{Defweight}
Let $\chi \in C^\infty (\R)$ be a decreasing function such that $\chi (s) =1$ for $s \leq 0$, and $\chi (s) =0$ for $s \geq 1$. For $a \in \R_+$, we set
$$  \overline{\omega}_a(r,p) \coloneqq \frac{1}{M^a}\chi (r-R)+\frac{1}{r^a}\big[ 1-\chi(r-R) \big] \chi \Big( 2\frac{p_{r^*}}{p_t} \Big). $$
\end{definition}
\begin{remark}
By definition, we have
\begin{align*}
 \forall \, 2M \leq r \leq R, \qquad\qquad\qquad\qquad \quad \overline{\omega}_a &(r,p) =M^{-a},\\
 \forall \, r \geq R+1, \qquad\qquad\qquad \frac{1}{r^a} \mathds{1}_{p_{r^*} \geq 0} \lesssim &\,\,\overline{\omega}_a(r,p) \lesssim \frac{1}{r^a} \mathds{1}_{p_{r^*} \geq -\frac{|p_t|}{2}}.
 \end{align*}
\end{remark} 

\begin{remark}
The cutoff function $\chi$ will allow, when $p_{r^*}\leq 0$ and $r \geq R+1$, to use the property that 
\begin{equation}\label{eq:cutoffp}
|p_{r^*}| \leq \frac{1}{2}|p_t| \Rightarrow |p_t|^2 \leq \Omega^2\frac{4|\slashed{p}|^2}{3r^2},
 \end{equation}
which directly follows from \eqref{eq:defConsQua}. This property turns out to be important to deal with terms with borderline spatial decay. We recall from the degenerate ILED of Proposition \ref{ProremainderILED} that the $p_{r^*}$ component requires slightly more spatial decay to be integrable than the spherical ones.
\end{remark}
The error terms arising from $\T_g(\overline{\omega}_a)$ will be handled as follows.
\begin{lemma}\label{Lemerrorterm1}
There exists $\overline{C}>0$, depending only on $R$, such that
\begin{align*}
\T_g \big( \overline{\omega}_{\frac{1}{4}} \big) \big| r^{\frac{1}{4}} \mathrm{D}_{\pmb{V}_{\! +}} f \big|  &  \leq   \overline{C}\bigg( \frac{|p_{r^*}|^2}{r^2 }+ \frac{|\slashed{p}|^2}{r^3 } \bigg) \Big[\big|Sf \big|+ \big|p_t \partial_{p_{r^*}} f\big| \Big] \,  \mathds{1}_{p_{r^*} \leq 0} \, \mathds{1}_{r \geq R}+\overline{C}\frac{|p_t|}{r}\big| \pmb{\varphi}_- \partial_t f \big| , \\
 \T_g \big( \overline{\omega}_{\frac{9}{4}} \big) \big|r^{\frac{9}{4}} \mathrm{D}_{p_{r^*}} f \big|  &  \leq   \overline{C}\bigg( \frac{|p_{r^*}|^2}{r^2 }+ \frac{|\slashed{p}|^2}{r^3 } \bigg)  \big| p_t \partial_{p_{r^*}} f \big| \,  \mathds{1}_{p_{r^*} \leq 0} \, \mathds{1}_{r \geq R} .
 \end{align*}
\end{lemma}
\begin{proof}
Let $a \in  \R_+$. We start by observing, as $\T_g (p_t)=0$, that
\begin{align}
\nonumber \T_g(\overline{\omega}_a)&  =  p_{r^*} \chi' (r-R) \Big[\frac{1}{M^a}-\frac{1}{r^a}\chi \Big( 2\frac{p_{r^*}}{p_t} \Big) \Big]-  \frac{ap_{r^*}}{r^{a+1}}\big[ 1-\chi(r-R) \big] \chi \Big( 2\frac{p_{r^*}}{p_t} \Big) \nonumber \\
 & \quad + \frac{2(r-3M)|\slashed{p}|^2}{r^{a+4}p_t}\big[ 1-\chi(r-R) \big] \chi' \Big( 2\frac{p_{r^*}}{p_t} \Big). \nonumber
 \end{align}
By support considerations and since $\chi' \leq 0$, we have for $r \geq 2M$ that
$$ \T_g(\overline{\omega}_a) \lesssim  |p_{r^*}| \,\mathds{1}_{p_{r^*} \leq 0} \,  \mathds{1}_{R \leq r \leq R+1}+  \bigg( \frac{|p_{r^*}|}{r^{a+1}} + \frac{|\slashed{p}|^2}{r^{a+3}|p_t|} \bigg) \mathds{1}_{-\frac{|p_t|}{2} \leq p_{r^*} \leq 0} \, \mathds{1}_{r \geq R} .$$
We now use \eqref{eq:cutoffp} to get
\begin{equation}\label{equation:unan}
 \T_g(\overline{\omega}_a) \lesssim  \frac{|p_t|}{r^{a+2}} \,\mathds{1}_{p_{r^*} \leq 0} \,  \mathds{1}_{R \leq r \leq R+1}+   \frac{|\slashed{p}|^2}{|p_t|r^{a+3}} \mathds{1}_{-\frac{|p_t|}{2} \leq p_{r^*} \leq 0} \, \mathds{1}_{r \geq R} .
 \end{equation}
It remains to use Remark \ref{RkProestiincoming}.
\end{proof}

Next, we introduce and study a $W^{1,1}_{\mathrm{loc}}$-weight $\zeta$ that can be used to derive an ILED that degenerates at $r=2M$, at $r=3M$, and when $r \to +\infty$.
\begin{definition}\label{Defz}
Let $\z \colon \mathcal{P}\to \R$ be the weight function
$$ \z(x,p) \coloneqq 2 -M^{\frac{1}{4}} \mathrm{sgn}(p_{r^*}) \mathrm{sgn}(r-3M) \Big|\frac{p_{r^*}}{p_t} \Big|^{\frac{1}{4}}\Big| \, \frac{1}{r} -\frac{1}{3M} \Big|^{\frac{1}{4}}   .   
$$
\end{definition}
\begin{lemma}\label{LemPropz}
The weight function $\z$ satisfies $1 \leq |\z| \leq 3$. Moreover, we have
$$ \T_g (\z) \lesssim - \frac{|p_{r^*}|}{r^2}-\bigg| 1-\frac{3M}{r} \bigg| \, \frac{|\slashed{p}|^2}{r^3|p_t|}.$$
\end{lemma}
\begin{remark}
Compared to $\T_g(\xi)$, estimated in Proposition \ref{Promultiplierm}, the degeneracy at the photon sphere is linear in $|p_{r^*}|+|r-3M|$ and not quadratic. This will allow for simplifications in the upcoming analysis.
\end{remark}
\begin{proof}
The first statement ensues from $r \geq 2M$ and $|p_{r^*}| \leq |p_t|$. For the second one, as $\T_g(p_t)=0$, we have
\begin{align*}
M^{-\frac{1}{4}}|p_t|^{\frac{1}{4}} \T_g (\z) &=- \frac{  |p_{r^*}|^{\frac{5}{4}}}{4r^2}\Big| \, \frac{1}{r} -\frac{1}{3M} \Big|^{-\frac{3}{4}}-\frac{\mathrm{sgn}(r-3M)}{4|p_{r^*}|^{\frac{3}{4}}}\Big| \, \frac{1}{r} -\frac{1}{3M} \Big|^{\frac{1}{4}} \cdot \frac{(r-3M)|\slashed{p}|^2}{r^4} \\
& \lesssim -\frac{|p_{r^*}|^{\frac{5}{4}}}{r^{\frac{5}{4}} |r-3M|^{\frac{3}{4}}}-|r-3M|^{\frac{5}{4}}\frac{|\slashed{p}|^2}{r^{\frac{17}{4}}|p_{r^*}|^{\frac{3}{4}}} .
 \end{align*}
 Note that, as $|p_{r^*}| \leq |p_t|$, it implies the stated estimate for the region $ \{r \leq 2.5M \} \cup \{ r \geq 4M \}$. Next, we note that if $2.5 \leq r \leq 4M$, then $|p_t| \sim |p_{r^*}|+|\slashed{p}|$ and $1 \lesssim |r-3M|^{-1}$. Consequently, we have on this region
 $$ |p_{r^*}| \leq \frac{4}{5} \cdot \frac{|p_{r^*}|^{\frac{5}{4}}}{|r-3M|^{\frac{1}{4}}|p_t|^{\frac{1}{4}}}+\frac{1}{5}\cdot |r-3M||p_t| \lesssim \frac{|p_{r^*}|^{\frac{5}{4}}}{ |r-3M|^{\frac{3}{4}}|p_t|^{\frac{1}{4}}}  + |r-3M|\frac{|\slashed{p}|^2}{|p_t|}$$
 as well as
 $$ |r-3M|\frac{|\slashed{p}|^2}{|p_t|} \leq  \frac{4}{5} \cdot \frac{|r-3M|^{\frac{5}{4}} |\slashed{p}|^{\frac{5}{2}}}{|p_{r^*}|^{\frac{3}{4}}|p_t|^{\frac{1}{2}}}+\frac{1}{5}\cdot \frac{|p_{r^*}|^3}{|p_t|^2} \lesssim |r-3M|^{\frac{5}{4}}\frac{|\slashed{p}|^2}{|p_{r^*}|^{\frac{3}{4}}|p_t|^{\frac{1}{4}}}+ \frac{|p_{r^*}|^{\frac{5}{4}}}{ |r-3M|^{\frac{3}{4}}|p_t|^{\frac{1}{4}}}  .$$
These last two inequalities allow to conclude the proof.
\end{proof}

We now start the proof of Proposition \ref{ProestivarphiminusVminus}. Let us consider a constant $A>0$ that will be chosen sufficiently large. We apply the energy estimate of Proposition \ref{Proenergy} in order to have
$$ \mathbb{E} \Big[ \overline{\omega}_{\frac{1}{4}} \z^A  r^{\frac{1}{4}}\mathrm{D}_{\pmb{V}_{\!+}} f \Big](\tau) \leq \mathbb{E} \Big[ \overline{\omega}_{\frac{1}{4}}   \z^A  r^{\frac{1}{4}}\mathrm{D}_{\pmb{V}_{\!+}} f \Big](0)+\int_{\pi^{-1} ( \mathcal{R}_0^\tau)}   \T_g \Big(  \overline{\omega}_{\frac{1}{4}}   \z^A \big| r^{\frac{1}{4}}\mathrm{D}_{\pmb{V}_{\!+}} f \big| \Big) \dr \mu_{\mathcal{P}} .$$
Similarly, we have
\begin{align*}
& \mathbb{E} \Big[\overline{\omega}_{\frac{9}{4}}  \z^A    r^{\frac{9}{4}}\mathrm{D}_{p_{r^*}} f \Big](\tau) \leq \mathbb{E} \Big[ \overline{\omega}_{\frac{9}{4}}  \z^A   r^{\frac{9}{4}}\mathrm{D}_{p_{r^*}} f \Big](0)+\int_{\pi^{-1} ( \mathcal{R}_0^\tau)}   \T_g \Big(  \overline{\omega}_{\frac{9}{4}}  \z^A \big| r^{\frac{9}{4}}\mathrm{D}_{p_{r^*}} f \big| \Big) \dr \mu_{\mathcal{P}} .
\end{align*}
We now make two observations:
\begin{itemize}
\item In order to prove Proposition \ref{ProestivarphiminusVminus}, it is enough to control $\mathrm{D}_{\pmb{V}_{\!+}} f$ and $\mathrm{D}_{p_{r^*}}f$ in the set $\{ r \leq R\} \cup \{ r \geq R, \; p_{r^*} \geq 0 \}$, in view of Proposition \ref{Proestiincoming} and Remark \ref{RkProestiincoming}.
\item On the region $\{ r \leq R\} \cup \{ r \geq R, \; p_{r^*} \geq 0 \}$, we have $\overline{\omega}_{1/4} \sim r^{-1/4}$, $\overline{\omega}_{9/4} \sim r^{-9/4}$ and $\z \sim 1$.
\end{itemize}
Thus, in view of the ILED \eqref{kevatalenn:0} satisfied by $\partial_t f$ and $\T_{|\slashed{p}|} f$ and Proposition \ref{Proestiincoming}, a sufficient condition for Proposition \ref{ProestivarphiminusVminus} to hold is given in the following lemma.

\begin{lemma}\label{LemforProestivarphiminusVminus}
For $A>0$ large enough, there exists $c>0$ such that
\begin{align*}
 & \frac{1}{8}\T_g \Big( \overline{\omega}_{\frac{1}{4}} \z^A \big| r^{\frac{1}{4}}\mathrm{D}_{\pmb{V}_{\!+}} f \big| \Big)  +\T_g \Big( \overline{\omega}_{\frac{9}{4}} \z^A \big| r^{\frac{9}{4}} \mathrm{D }_{p_{r^*}} f \big| \Big)  +c\bigg( \frac{|p_N|}{r } \big|\mathrm{D}_{\pmb{V}_{\!+}} f \big|+\frac{|p_t|}{r \Omega^2} \big| \mathrm{D }_{p_{r^*}} f \big|  \bigg) \mathds{1}_{\{p_{r^*} <0, \, r \geq R\}^\mathsf{c}} \\
  & \qquad \qquad \qquad \qquad \qquad   \lesssim  \frac{|p_N|}{r} \big| \pmb{\varphi}_- \partial_t f \big| +\frac{|p_N|}{r} \big| \pmb{\varphi}_- \T_{|\slashed{p}|} f \big|  + \bigg( \frac{|p_{r^*}|^2}{r^2 }+ \frac{|\slashed{p}|^2}{r^3 } \bigg)\Big[\big|Sf\big|+  \big|p_t \partial_{p_{r^*}} f\big| \Big]\, \mathds{1}_{\{p_{r^*} <0, \, r \geq R\}} .
\end{align*}
\end{lemma}
\begin{proof}
Recall first $|p_t| \leq |p_N|$ and Proposition \ref{Prorsweight}, which in particular implies
$$ \T_g(|p_N|) \lesssim  |p_t|^2 \, \mathds{1}_{2.5M \leq r \leq 2.7M} \lesssim  \Big| 1-\frac{3M}{r} \Big| |p_t|^2 \mathds{1}_{r \leq R} .$$
Next, we have from Lemma \ref{LemPropz} that $\z^A \leq 3^A$ and
$$  \T_g (\z^A) = A \T_g(\z) \z^{A-1}  \leq A \T_g(\z) \z^{A}  \lesssim  -A\bigg( \frac{|p_{r^*}|}{r^2}+\bigg|1-\frac{3M}{r}\bigg| \frac{|\slashed{p}|^2}{r^3|p_t|} \bigg) \z^A .$$
We further stress that $|p_t| \leq |p_{r^*}|+|\slashed{p}|r^{-1}$. Thus, according to Lemmata \ref{Propartialprstarbisbis} and \ref{Lemerrorterm1}, there exists $b>0$ depending only on $M$, such that if $A$ is large enough,
\begin{align*}
\T_g \Big( \overline{\omega}_{\frac{1}{4}} \z^A \big| r^{\frac{1}{4}}\mathrm{D}_{\pmb{V}_{\!+}} f \big| \Big) & \leq - \mathfrak{F}(r,p) \overline{\omega}_{\frac{1}{4}} \z^A \big| r^{\frac{1}{4}}\mathrm{D}_{\pmb{V}_{\!+}} f \big| + \frac{(r+3M)|p_N|^{\frac{3}{4}} |p_t|^{\frac{1}{4}} }{r^{\frac{5}{2}}|r+6M|^{\frac{3}{2}}\Omega^{\frac{1}{2}}} \overline{\omega}_{\frac{1}{4}} \z^A \big|r^{\frac{9}{4}} \mathrm{D}_{p_{r^*}}f \big| \\
& \quad + 3^A\overline{C}\bigg( \frac{|p_{r^*}|^2}{r^2 }+ \frac{|\slashed{p}|^2}{r^3 } \bigg) \Big(\big|Sf \big|+ \big|p_t \partial_{p_{r^*}} f\big| \Big) \,  \mathds{1}_{p_{r^*} \leq 0} \, \mathds{1}_{r \geq R}+3^A\overline{C}\frac{|p_t|}{r}\big| \pmb{\varphi}_- \partial_t f \big|,
 \end{align*}
 where the function $\mathfrak{F}(r,p)$ is defined as
 \begin{align}
 \nonumber \mathfrak{F}(r,p) &\coloneqq  \frac{5(r^2+2Mr+3M^2)|p_t|}{4r|r+6M|^{\frac{1}{2}}(r^{\frac{3}{2}}+M|r+6M|^{\frac{1}{2}}) } +\frac{5M|p_u|}{2r^2\Omega^2}+ bA\bigg( \frac{|p_{r^*}|}{r^2}+\bigg|1-\frac{3M}{r}\bigg| \frac{|\slashed{p}|^2}{r^3|p_t|} \bigg). \label{eq:DefCalF}
 \end{align}
Similarly, there exists $\widetilde{C}>0$ such that
\begin{align*}
 \T_g \Big(   \overline{\omega}_{\frac{9}{4}} \z^A \big|r^{\frac{9}{4}} \mathrm{D }_{p_{r^*}} f \big| \Big)& \leq - \mathfrak{G}_0(r,p)   \overline{\omega}_{\frac{9}{4}} \z^A \big|r^{\frac{9}{4}}  \mathrm{D }_{p_{r^*}} f \big| +  \frac{\Omega^{\frac{1}{2}}\pmb{\varphi}_-  |p_t|^{\frac{3}{4}}}{|p_N|^{\frac{3}{4}}} \overline{\omega}_{\frac{9}{4}} \z^A  \big|  r^{\frac{1}{4}} \mathrm{D}_{\pmb{V}_{\! +}} f \big|\\
 & \quad +3^A\overline{C}\bigg( \frac{|p_{r^*}|^2}{r^2 }+ \frac{|\slashed{p}|^2}{r^3 } \bigg)  |p_t \partial_{p_{r^*}} f| \,  \mathds{1}_{p_{r^*} \leq 0} \, \mathds{1}_{r \geq R} +3^A \widetilde{C}\frac{|p_N|}{ r} \Big| \pmb{\varphi}_- \partial_t f \Big|,
 \end{align*}
 with
  \begin{align*}
 \mathfrak{G}_0(r,p) & \coloneqq \frac{(r^2+2Mr+3M^2)|p_t|}{4r|r+6M|^{\frac{1}{2}}(r^{\frac{3}{2}}+M|r+6M|^{\frac{1}{2}}) }+\frac{M|p_v|}{2r^2 \Omega^2}+ bA\bigg( \frac{|p_{r^*}|}{r^2}+\bigg|1-\frac{3M}{r}\bigg| \frac{|\slashed{p}|^2}{r^3|p_t|} \bigg) \\
 & \, \geq \frac{|p_t|}{4r^{\frac{1}{2}}|r+6M|^{\frac{1}{2}} }+\frac{M|p_v|}{2r^2 \Omega^2} + bA\bigg( \frac{|p_{r^*}|}{r^2}+\bigg|1-\frac{3M}{r}\bigg| \frac{|\slashed{p}|^2}{r^3|p_t|} \bigg),
 \end{align*}
and where we controlled the term proportional to $\partial_t f$ by using 
$$|\pmb{\varphi}_-| \lesssim r|p_N|,  \qquad |\Omega^2 \pmb{\varphi}_-| \lesssim r |p_t|, \qquad \overline{\omega}_{\frac{9}{4}} \lesssim r^{-\frac{9}{4}}, \qquad \z^A \leq 3^A. $$
Note that $\mathfrak{G}_0(r,p)$ does not control $|p_u| \Omega^{-2}$ near $\H$. To deal with this issue, we apply Corollary \ref{CorBoundpartialpr}. It implies that there exists $C_0 >0$, depending only on $M$, such that, if $A$ is large enough,
\begin{align*}
& \T_g \Big(   \overline{\omega}_{\frac{9}{4}} \z^A \big| r^{\frac{9}{4}} \mathrm{D }_{p_{r^*}} f \big| \Big) \leq - \mathfrak{G}(r,p)   \overline{\omega}_{\frac{9}{4}} \z^A \big| r^{\frac{9}{4}} \mathrm{D }_{p_{r^*}} f \big| +\big(1+ C_0 \mathds{1}_{r \leq 2.7M} \big) \frac{\Omega^{\frac{1}{2}}\pmb{\varphi}_-  |p_t|^{\frac{3}{4}}}{|p_N|^{\frac{3}{4}}}  \overline{\omega}_{\frac{9}{4}} \z^A  \big|  r^{\frac{1}{4}} \mathrm{D}_{\pmb{V}_{\! +}} f \big| \\  &\qquad + \! 3^A\overline{C}\bigg( \frac{|p_{r^*}\!|^2}{r^2 }+ \frac{|\slashed{p}|^2}{r^3 } \bigg)  |p_t \partial_{p_{r^*}} f| \,  \mathds{1}_{p_{r^*} \leq 0} \, \mathds{1}_{r \geq R} + 3^A \widetilde{C}\frac{|p_N|}{ r} \Big| \pmb{\varphi}_- \partial_t f \Big|+3^AC_0 \mathds{1}_{r \leq 2.7M} | \pmb{\varphi}_- p_N| \big( |\partial_t f | + |\T_{|\slashed{p}|} f | \big), 
 \end{align*}
  where
  \begin{equation}\label{eq:DefCalG}
 \mathfrak{G}(r,p) \coloneqq \frac{|p_t|}{4r^{\frac{1}{2}}|r+6M|^{\frac{1}{2}} }+\frac{M|p_v|}{2r^2 \Omega^2}+ \frac{M|p_u|}{2r^2 \Omega^2} \mathds{1}_{r \leq 2.7M}. 
 \end{equation}
We recall Remark \ref{RkProestiincoming} and the properties of $\overline{\omega}_a$, in particular \eqref{eq:cutoffp}. Then, we have that in the domain $\{ r \geq R, \; p_{r^*} <0 \}$,
\begin{align*}
\T_g \Big( \overline{\omega}_{\frac{1}{4}} \z^A\big|r^{\frac{1}{4}} \mathrm{D}_{\pmb{V}_{\! +}} f \big| \Big) +  \T_g \Big(   \overline{\omega}_{\frac{9}{4}} \z^A \big| r^{\frac{9}{4}} \mathrm{D }_{p_{r^*}} f \big| \Big) \lesssim \bigg(\frac{|p_{r^*}|^2}{r^2}+\frac{|\slashed{p}|^2}{r^3 } \bigg) \big(  |p_t \partial_{p_{r^*}} f| + |Sf| \big)+\frac{|p_N|}{ r} \big| \pmb{\varphi}_- \partial_tf \big|.
\end{align*}
Finally, the result is a consequence of the following technical Lemma \ref{LemtechniCompu}.
\end{proof}

\begin{lemma}\label{LemtechniCompu}
We have, on $\{ r \leq R\} \cup \{ r \geq R, \, p_{r^*} \geq 0 \}$, that
\begin{equation*}
  \mathfrak{G}(r,p)- \frac{1}{8} \cdot \frac{(r+3M)|p_N|^{\frac{3}{4}} |p_t|^{\frac{1}{4}}\overline{\omega}_{1/4} }{r^{\frac{5}{2}}|r+6M|^{\frac{3}{2}}\Omega^{\frac{1}{2}} \overline{\omega}_{9/4}} \gtrsim  \frac{|p_N|}{r}.
 \end{equation*}
 and, if $A$ is large enough,
\begin{equation*}
 \frac{1}{8}\mathfrak{F}(r,p)-  \frac{ \Omega^{\frac{1}{2}}|\pmb{\varphi}_-||p_t|^{\frac{3}{4}}\overline{\omega}_{9/4}}{ |p_N|^{\frac{3}{4}}\overline{\omega}_{1/4}}\big(1+ C_0 \mathds{1}_{r \leq 2.7M} \big) \gtrsim \frac{|p_N|}{r}
 \end{equation*}
\end{lemma}
\begin{proof}
We will prove several estimates which, together with $\mathfrak{F}(r,p), \, \mathfrak{G} (r,p) \gtrsim |p_N| r^{-1}$, will imply the result. We focus first on the domain $\{ r \geq R+1, \, p_{r^*} \geq 0 \}$, with $R \geq 832M$. For this, we will use the next properties.
 \begin{itemize}
 \item For $r \geq R+1$, we have $p_N=p_t$, $\overline{\omega}_{1/4}(r,p)=r^{-1/4}$ and $\overline{\omega}_{9/4}(r,p)=r^{-9/4}$.
 \item If  $p_{r^*} \geq 0$, then $2|p_u|=|p_t|+p_{r^*}\geq |p_t|$.
 \item For $ r \geq 3M$, we have by Remark \ref{RkvarphiHplus} that
 \begin{equation*}
 \pmb{\varphi}_-(x,p)= \frac{r }{\Omega^2}(p_t+ p_{r^*}) -\frac{27M^2r^{\frac{1}{2}}}{|r+6M|^{\frac{1}{2}}(r-3M)+r^{\frac{3}{2}}} p_t.
 \end{equation*}
 \end{itemize}
 Consequently, in view of the null-shell relation $4r^2p_up_v=\Omega^2 |\slashed{p}|^2$, one obtains that for $r \geq R+1$ and $p_{r^* } \geq 0$, 
  $$ \frac{ \Omega^{\frac{1}{2}}|\pmb{\varphi}_-||p_t|^{\frac{3}{4}}\overline{\omega}_{9/4}}{ |p_N|^{\frac{3}{4}}\overline{\omega}_{1/4}} =  \frac{ \Omega^{\frac{1}{2}}|\pmb{\varphi}_-|}{r^2}\leq \frac{2|p_v|}{r\Omega^2}+ \frac{27M^2|p_t|}{r^3} \leq \frac{4|p_v||p_u|}{r\Omega^2|p_t|}+ \frac{27M^2|p_t|}{r^3} \leq 2\frac{|\slashed{p}|^2}{r^3|p_t|}+\frac{27M^2|p_{r^*}|}{r^3} .$$
  Hence, if $A$ is large enough, we have
  $$ \frac{ \Omega^{\frac{1}{2}}|\pmb{\varphi}_-||p_t|^{\frac{3}{4}}\overline{\omega}_{9/4}}{ |p_N|^{\frac{3}{4}}\overline{\omega}_{1/4}}  \leq \frac{bA}{16} \bigg( \frac{|p_{r^*}|}{r^2}+\bigg|1-\frac{3M}{r}\bigg| \frac{|\slashed{p}|^2}{r^3 |p_t|} \bigg) \leq \frac{1}{16} \mathfrak{F}(r,p)  .$$
  For the other inequality, note that
  $$  \frac{(r+3M)|p_N|^{\frac{3}{4}} |p_t|^{\frac{1}{4}}\overline{\omega}_{1/4} }{r^{\frac{5}{2}}|r+6M|^{\frac{3}{2}}\Omega^{\frac{1}{2}} \overline{\omega}_{9/4}} =  \frac{(r+3M) |p_t| }{r^{\frac{1}{2}}|r+6M|^{\frac{3}{2}}\Omega^{\frac{1}{2}} } \leq \frac{|p_t|}{r^{\frac{1}{2}}|r+6M|^{\frac{1}{2}} \Omega^{\frac{1}{2}}(R)} \leq 5 \mathfrak{G}(r,p) .$$

Assume now that $2.7M \leq r \leq R+1$, so that $p_N=p_t$ and $\overline{\omega}_{1/4} /\overline{\omega}_{9/4} \leq 2M^2$. Hence,
$$  \frac{(r+3M)|p_N|^{\frac{3}{4}} |p_t|^{\frac{1}{4}}\overline{\omega}_{1/4} }{r^{\frac{5}{2}}|r+6M|^{\frac{3}{2}}\Omega^{\frac{1}{2}} \overline{\omega}_{9/4}} \leq  \frac{2M^2 |p_t| }{r^{\frac{5}{2}}|r+6M|^{\frac{1}{2}}\Omega^{\frac{1}{2}} } \leq \frac{|p_t|}{r^{\frac{1}{2}}|r+6M|^{\frac{1}{2}}} \leq 4 \mathfrak{G}(r,p) .$$
Since $\Omega^2|\pmb{\varphi}_-| \leq r|p_{r^*}|+2|r-3M||p_t|$, we have for a certain constant $C>0$ depending only on $M$,
$$\frac{ \Omega^{\frac{1}{2}}|\pmb{\varphi}_-||p_t|^{\frac{3}{4}}\overline{\omega}_{9/4}}{ |p_N|^{\frac{3}{4}}\overline{\omega}_{1/4}} \leq C(|p_{r^*}|+|r-3M||p_t|) \leq \frac{bA}{16} \bigg( \frac{|p_{r^*}|}{r^2}+\bigg|1-\frac{3M}{r}\bigg| \frac{|\slashed{p}|^2}{r^3 |p_t|} \bigg) \leq \frac{1}{16} \mathfrak{F}(r,p)  ,$$
provided that $A$ is large enough.

We finally consider the case $r \leq 2.7M$, so that $\overline{\omega}_{a}=M^{-a}$. As $|p_N| \leq 16|p_t|\Omega^{-2}$, we have
$$  \frac{(r+3M)|p_N|^{\frac{3}{4}} |p_t|^{\frac{1}{4}}\overline{\omega}_{1/4} }{r^{\frac{5}{2}}|r+6M|^{\frac{3}{2}}\Omega^{\frac{1}{2}} \overline{\omega}_{9/4}} \leq  \frac{8 M^2|p_t|}{r^{\frac{5}{2}}|r+6M|^{\frac{1}{2}}\Omega^2} \leq \frac{2M|p_t|}{r^2\Omega^2} = 2M \frac{|p_v|+|p_u|}{r^2\Omega^2} \leq 4 \mathfrak{G}(r,p)  .$$
We conclude the proof by using $|\pmb{\varphi}_-| \lesssim |p_N|$ and $\Omega^2|\pmb{\varphi}_-|\lesssim |p_t|$, which implies
$$ \frac{ \Omega^{\frac{1}{2}}|\pmb{\varphi}_-||p_t|^{\frac{3}{4}}\overline{\omega}_{9/4}}{ |p_N|^{\frac{3}{4}}\overline{\omega}_{1/4}} \leq C |p_t| \leq \frac{bA}{16} \bigg( \frac{|p_{r^*}|}{r^2}+\bigg|1-\frac{3M}{r}\bigg| \frac{|\slashed{p}|^2}{r^3 |p_t|} \bigg) \leq \frac{1}{16} \mathfrak{F}(r,p)  ,$$
for $A$ chosen large enough.  
 \end{proof}

\subsection{Boundedness and ILED without relative degeneracy for $p_N \pmb{V}_{\, +}^{\mathrm{mod}}$}\label{Subsec74}

We are now able to control the derivative $\pmb{V}_{\, +}^{ \mathrm{mod}} f$.

\begin{proposition}\label{ProestiVminusmod}
For all $\tau \geq 0$, we have
\begin{align*}
& \mathbb{E} \Big[ p_N  \pmb{V}_{\, +}^{ \mathrm{mod}} f \Big](\tau)+ \int_{\pi^{-1}(\mathcal{R} )  } \frac{|p_N|}{r} \big| p_N \pmb{V}_{\, +}^{ \mathrm{mod}} f \big|  \dr \mu_{\mathcal{P}}  \lesssim \mathcal{F}[f](0).
\end{align*}
\end{proposition}
\begin{proof}
 Recall the weight function $\xi(x,p)$ introduced in Section \ref{Subsecredshift}, which verifies $\T_g(|\xi|) \leq 0$ and $|\xi| \sim |p_N|$. By Corollary \ref{corVminusmod}, we have
\begin{align}\label{eq:12345}
  \T_g \big( \big| \xi \pmb{V}_{\, +}^{ \mathrm{mod}} f \big| \big)  + \pmb{a}(r,p_{r^*},p_t) \big| \xi \pmb{V}_{\, +}^{  \mathrm{mod}} f \big| \lesssim \frac{|p_N|}{r } \big| \pmb{\varphi}_-  \partial_t f \big|+\frac{|p_N|}{r^{\frac{3}{2}}} |p_t|^{\frac{3}{4}}| r^{-1} \pmb{\varphi}_- |^{\frac{1}{4}} \big|  \pmb{V}_{\! +} f| .
\end{align}
Then, we apply the energy estimate of Proposition \ref{Proenergy} to $| \xi \pmb{V}_{\, +}^{ \mathrm{mod}} f |$ and we control the integral over $\pi^{-1}(\mathcal{R})$ of the two error terms in the RHS of \eqref{eq:12345} as follows.
\begin{itemize}
\item We use \eqref{kevatalenn:0} in order to deal with the first one. 
\item For the second one, we apply Proposition \ref{ProestivarphiminusVminus} for the region $\{ r \leq R \} \cup \{ p_{r^*} \geq 0, \, r \geq R \}$. For the domain $\{p_{r^*} \leq 0, \, r \geq R \}$, we use Remark \ref{RkProestiincoming}, Proposition \ref{Proestiincoming} and \eqref{kevatalenn:0}.
\end{itemize}
\end{proof}

We then obtain Proposition \ref{ProILEDwrd} from \eqref{kevatalenn:0}, Proposition \ref{Proestiincoming}, Remark \ref{RkProestiincoming} as well as Propositions \ref{ProestivarphiminusVminus} and \ref{ProestiVminusmod}. Note that we also used $|p_{n_{\Sigma_\tau}}| \lesssim |p_N|$ and \eqref{eq:coarea}.

 In the perspective of applying the $r^p$-weighted energy method through Proposition \ref{ProILEDrpbis}, we conclude this subsection with the following result. Let us define
\begin{align*}
&g_1 \coloneqq   \pmb{\varphi}_- \partial_t f , \qquad\quad
g_2 \coloneqq  \pmb{\varphi}_- \T_{|\slashed{p}|} f  , \qquad \quad
g_3 \coloneqq  p_t  Sf \, \mathds{1}_{r \geq R} \, \mathds{1}_{p_{r^*} \leq 0}  , \qquad \quad
g_4  \coloneqq  |p_t|^2 \partial_{p_{r^*}} f  \, \mathds{1}_{r \geq R} \, \mathds{1}_{p_{r^*} \leq 0} , \\
& \qquad \qquad g_5  \coloneqq    \overline{\omega}_{\frac{1}{4}} \z^A |p_N|^{\frac{3}{4}}  |\pmb{\varphi}_-|^{\frac{1}{4}} \pmb{V}_{\! +} f   , \qquad \quad
g_6  \coloneqq     \overline{\omega}_{\frac{9}{4}} \z^A  r^{\frac{9}{4}} \mathrm{D}_{p_{r^*}} f  , \qquad \quad
g_7  \coloneqq \xi  \pmb{V}_{\, +}^{ \mathrm{mod}} f , 
\end{align*}
with $A>0$ chosen large enough in accordance with Lemmata \ref{LemforProestivarphiminusVminus}--\ref{LemtechniCompu}.

\begin{proposition}\label{Rksystemderiv}
Let $k\in \llbracket 1, 7 \rrbracket$. There exist constants $c_k>0$ such that
\begin{equation}
 \sum_{1 \leq k \leq 7} c_k \T_g (|g_k|) \lesssim  -\sum_{1 \leq k \leq 7} \frac{|p_N|}{r\log^2(2+r)} |g_k| .
\end{equation}
\end{proposition}

\begin{proof}
It follows from Lemma \ref{Lemphiminus}, the identities $\T_g(\partial_t f )=\T_g(\T_{|\slashed{p}|}f)=0$, the estimate \eqref{eq:forRk641}, Lemma \ref{LemforProestivarphiminusVminus} and \eqref{eq:12345}.
\end{proof}

\subsection{Proof of Proposition \ref{ProILEDwrd2}}\label{Subsec75}

Our argument will rely on a Sobolev inequality involving the vector field $G$ defined in \eqref{defG}. We introduce first a coordinate system on Schwarzschild spacetime that is related to the integral curves of $r^{-1} \Omega^{-1}G$. These regular hyperboloidal coordinates have been previously considered in \cite{Mavr23}. 

\begin{definition}\label{Defhypcoor}
The hyperboloidal coordinate system $(\overline{t},r,\theta , \phi) \in \R \times (2M,+\infty)\times (0, \pi) \times (0,2\pi)$ is defined by
$$ \overline{t} \coloneqq t-H(r), \qquad  H(r)\coloneqq \int_{3M}^{r} \frac{\xi(s)}{\Omega^2(s)} \dr s, \qquad \xi(r) \coloneqq \Big( 1-\frac{3M}{r} \Big)\Big( 1+\frac{6M}{r}\Big)^{\frac{1}{2}}.$$
To avoid any confusion, we denote by $\overline{\partial}_{x^\mu}$ the derivatives with respect to the variables $x^\mu$ in the hyperboloidal coordinate system.
\end{definition}
We recall $t^* = t+2M\log(r-2M)$ and we collect several useful properties.
\begin{proposition}\label{Prohyperboloidalcoord}
The following properties hold:
\begin{enumerate}[label = (\alph*)]
\item The Jacobian determinant of the map $(t,r,\theta , \phi) \mapsto (\overline{t},r,\theta , \phi)$ is constant equal to one.
\item We have the relations
$$ \overline{\partial}_r= \partial_r+\frac{1}{\Omega^2}\Big( 1-\frac{3M}{r} \Big)\Big( 1+\frac{6M}{r}\Big)^{\frac{1}{2}}\partial_t, \qquad  \overline{\partial}_t = \partial_t, \qquad  \overline{\partial}_\theta = \partial_\theta, \qquad \overline{\partial}_\phi = \partial_\phi.$$
\item The vector fields $G$ and $\pmb{V}_{\, +}^{ \mathrm{mod}}$ are related to the vector field $\overline{\partial}_r$ through the relations
$$ \Omega^{-1}G=\dr \pi \big( \pmb{V}_{\, +}^{ \mathrm{mod}} \big)= r \overline{\partial}_r.$$
\item If $\epsilon >0$ is small enough, then we have the following inclusions
$$ \Big\{ t^* \geq 2, \; |r-3M| \leq \epsilon  \Big\} \subset \Big\{ \overline{t}+2M\log(M) \geq 1, \; |r-3M| \leq \epsilon  \Big\} \subset \Big\{ t^* \geq 0, \; |r-3M| \leq \epsilon  \Big\} .$$ 
\end{enumerate}
\end{proposition}

Moreover, we set 
$$\overline{g}\big(\overline{t},r,\theta,\phi,p\big)\coloneqq g\big(\overline{t}+H(r),r^*(r),\theta,\phi,p\big).$$ 
 
We are now able perform the proof of Proposition \ref{ProILEDwrd2}. We then assume that $\mathcal{E}[f](0)<+\infty$.

\begin{proof}[Proof of Proposition \ref{ProILEDwrd2}]
We first recall the volume form $\dr \mu_{\mathcal{P}} =h(r) \dr \tau \dr \mu_{\Sigma_\tau} \dr \mu_{\mathcal{P}_x}$, where $h(r) \sim 1$. For convenience, we introduce for a distribution function $g$ the quantities
\begin{align*}
\mathcal{J}[g]&\coloneqq \int_{\pi^{-1}(\mathcal{R})\,\cap \,\{ \overline{t}+2M\log(M) \geq 1, \; |r-3M| \leq \epsilon  \} } |\slashed{p}||p_t||g| \dr \mu_{\mathcal{P}}, \\
I[g] (r)& \coloneqq \int_{\overline{t}=1+2M\log(M)}^{+ \infty} \int_{\theta=0}^\pi \int_{\phi =0}^{2\pi} \int_{p_{r^*} \in \R} \int_{p_\theta \in \R} \int_{p_\phi \in \R} |\slashed{p}||\overline{g}| \dr p_{\phi} \dr p_{\theta} \dr p_{r^*} \dr \phi \dr \theta  \dr \overline{t},
\end{align*}
where $\epsilon <M$ is a constant that will be fixed small enough. In view of the expression of the volume form $\dr \mu_{\mathcal{P}}$, we have
\begin{equation}\label{eq:equivalence}
 \mathcal{J}[g] \sim \int_{|r-3M|\leq \epsilon} I[g](r) \dr r.
 \end{equation}
We also note that
 $$ \int_{\pi^{-1}(\mathcal{R})\,\cap\, \{ 0 \leq t^* \leq 2, \; |r-3M| \leq \epsilon  \} } |\slashed{p}||p_t||g| \dr \mu_{\mathcal{P}}  \lesssim \sup_{0 \leq \tau \leq 2} \mathbb{E}\big[p_N g \big](\tau) \leq \mathcal{E}[g](0)     .$$

Let us come back to the study of the solution $f$ of the massless Vlasov equation. In view of $|\slashed{p}|^2 \lesssim |\slashed{p}||p_t|$ for $r \sim 3M$, the ILED in Proposition \ref{ProremainderILED} and Proposition \ref{ProILEDwrd}, it suffices to show 
\begin{equation}\label{eq:toprove}
\mathcal{J}[f]  \lesssim \mathcal{E}[f](0)
 \end{equation} 
 in order to prove Proposition \ref{ProILEDwrd2}. According to \eqref{eq:equivalence}, we have
\begin{equation}\label{equa:auxi}
 \mathcal{J}[f] \lesssim  \epsilon  \sup_{|r-3M| \leq \epsilon } I[f]\big(r\big).
 \end{equation}
The factor $\epsilon>0$ will allow us to absorb in the LHS terms proportional to $\mathcal{J}[f]$ that we cannot control a priori by $\mathcal{E}[f](0)$. We now apply a local Sobolev inequality to have
$$ \sup_{|r-3M| \leq \epsilon } I[f]\big(r\big) \lesssim \int_{|r-3M| \leq \epsilon }  I[f]  + \big| \overline{\partial}_r I[f] \big| \dr r. $$
To deal with the radial derivatives, we stress that $|\slashed{p}|$ is independent of $(t,r,p_{r^*})$ and we write
$$ r  \overline{\partial}_r\big|\overline{f}\big|= \pmb{V}_{\, +}^{ \mathrm{mod}}(|f|)+\bigg( \frac{r-3M}{r\Omega^2}p_{r^*}+\frac{r^{\frac{1}{2}}}{|r+6M|^{\frac{1}{2}} \Omega^2}p_t \bigg) \partial_{p_{r^*}}|f|-\Phi r \pmb{\varphi}_- \partial_{p_{r^*}}|f| .    $$
Consequently, as $\Omega^{-1}$ and $r$ are bounded on $\{|r-3M| \leq \epsilon  \}$, we obtain by integration by parts in $p_{r^*}$ that
$$ \sup_{|r-3M| \leq \epsilon } I[f](r) \lesssim \int_{|r-3M| \leq \epsilon} I\big[ \, \pmb{V}_{\, +}^{ \mathrm{mod}}f \big]+ I \big[(1+|\Phi|+|\pmb{\varphi}_-\partial_{p_{r^*}} \Phi|)f \big]  \dr r. $$
We recall now that $\Phi$ and $\pmb{\varphi}_-\partial_{p_{r^*}} \Phi$ are bounded in $L_{x,p}^\infty(\pi^{-1}(\mathcal{R}))$ from Proposition \ref{ProboundPsi}. Thus, by \eqref{eq:equivalence} and the ILED of Proposition \ref{ProILEDwrd}, we have
$$ \sup_{|r-3M| \leq \epsilon } I[f](r) \lesssim \int_{|r-3M| \leq \epsilon}  I \big[ \, \pmb{V}_{\, +}^{ \mathrm{mod}}f \big] + I \big[f \big]  \dr r \lesssim \mathcal{J} \big[ \, \pmb{V}_{\, +}^{ \mathrm{mod}}f \big]+\mathcal{J}[f] \lesssim \mathcal{F}[f](0)+ \mathcal{J}[f]. $$
We then deduce \eqref{eq:toprove} from this last estimate and \eqref{equa:auxi}, by choosing $\epsilon$ small enough.
\end{proof}

\section{Decay of the energy flux}\label{Sec7}

Let us now show quantitative decay estimates for the first order energy flux $\mathcal{E}[f]$ studied in the previous section. For this, we use the $r^p$-weighted energy method suitably adapted to the framework in place.

\subsection{The $r^p$-weighted energy method}

Let us begin by proving the following elementary result.

\begin{lemma}\label{Lemrpmethod}
Let $p \in \mathbb{N}^*.$ For $q \in \llbracket 0, p \rrbracket$, let $\mathcal{F}_q \colon \R_+ \to \R_+$ be functions satisfying the following properties:
\begin{enumerate}[label = (\alph*)]
\item \label{assum1unifbound} Uniform boundedness. There exist constants $\overline{C}_q >0$ such that
$$ \forall \, \tau_2 \geq \tau_1 \geq 0, \qquad \qquad \mathcal{F}_q (\tau_2) \leq \overline{C}_q \mathcal{F}_q (\tau_1).$$
\item \label{assum2hierenergest} Hierarchy of integrated energy decay estimates. For any $q \in \llbracket 1, p \rrbracket$, we have
$$ \forall \, \tau_1 \geq \tau_2 \geq 0, \qquad \qquad \int_{\tau=\tau_1}^{\tau_2} \mathcal{F}_{q-1}(\tau) \dr \tau \leq \overline{C}_q \mathcal{F}_q (\tau_1).$$
Then, there exists a constant $C_p>0$ such that
$$ \forall \, \tau \geq 0, \quad\qquad \mathcal{F}_0(\tau) \leq \frac{C_p}{\langle \tau\rangle^p}  \mathcal{F}_p(0).$$ 
\end{enumerate}
\end{lemma}

\begin{proof}
We note that the result holds for $0 \leq \tau \leq 1$ in view of the assumption \ref{assum1unifbound}. We deal with the case $\tau \geq 1$ by performing an induction. Let us prove that there exists a constant $C>0$ such that
\begin{equation}\label{eq:IH}
\forall \, \tau \geq 1, \quad\qquad  \mathcal{F}_q(\tau) \leq \frac{C}{\tau^{p-q}}  \mathcal{F}_p(0) \tag{IH-q}
\end{equation}
holds for any $q \in \llbracket 0,p \rrbracket$. By \ref{assum1unifbound}, the induction hypothesis holds for $q=p$. Let $q \in \llbracket 1,p \rrbracket$ such that \eqref{eq:IH} holds. Fix $i \in \mathbb{N}$ and use the first assumption \ref{assum1unifbound} for $(\tau_1,\tau_2)=(\tau,2^{i+1})$ to obtain
\begin{align*}
\int_{ \tau= 2^i}^{2^{i+1}} \mathcal{F}_{q-1}(\tau) \dr \tau \geq  \int_{\tau = 2^i}^{2^{i+1}} \frac{1}{\overline{C}_{q-1}} \mathcal{F}_{q-1} (2^{i+1})\dr \tau =\frac{2^i}{\overline{C}_{q-1}} \mathcal{F}_{q-1} (2^{i+1}).
\end{align*} 
Combining this estimate with the assumption \ref{assum2hierenergest} applied on the interval $(\tau_1,\tau_2)=(2^i,2^{i+1})$, we obtain from the induction hypothesis that
\begin{align*}
 \mathcal{F}_{q-1} (2^{i+1}) \leq \frac{\overline{C}_{q-1}}{2^i} \int_{ \tau=2^i}^{2^{i+1}} \mathcal{F}_{q-1}(\tau) \dr \tau \leq \frac{\overline{C}_{q-1}}{2^i} \cdot \overline{C}_q  \mathcal{F}_q(2^i) \leq  \frac{\overline{C}_{q-1}\overline{C}_q}{2^i} \cdot \frac{C}{2^{(p-q)i}}  \mathcal{F}_p(0).
\end{align*} 
Let us consider $\tau \geq 2$ since the case $\tau \leq 2$ is straightforward by \ref{assum1unifbound}. Then, there exists $i \in \mathbb{N}$ such that $2^{i+1} \leq \tau < 2^{i+2}$, so $2^i \sim \tau$. Using again the assumption \ref{assum1unifbound}, we have
$$  \mathcal{F}_{q-1} (\tau) \leq \overline{C}_{q-1}\mathcal{F}_{q-1} (2^{i+1}) \leq |\overline{C}_{q-1}|^2\overline{C}_q \cdot \frac{C}{2^{(p-q+1)i}}  \mathcal{F}_p(0) \lesssim \frac{1}{\tau^{p-q+1}}  \mathcal{F}_p(0) .$$
By considering a larger constant $C$, the induction hypothesis holds at step $q-1$. From here the proof follows.
\end{proof}

\subsection{Decay of the energy flux $\mathcal{E}[f]$}\label{subsect_decay_fluxes}

For $p \geq 0$ and $g \colon \mathcal{P} \to \R$, we define the norm
$$\mathbb{E}^p[g]  \coloneqq \mathbb{E} \Big[ \Big\langle r^{p} \Big| \frac{p_v}{p_t} \Big|^{\lceil \frac{p}{2} \rceil} \Big\rangle \, g \Big].$$

\begin{remark}
As a direct application of Lemma \ref{Lemrpmethod}, we get from Proposition \ref{ProILEDrp} applied to $\partial_t f$ that
$$ \forall \, \tau \geq 0, \qquad\quad \mathbb{E} \big[ \pmb{\varphi}_- \partial_t f \big] (\tau) \lesssim_p \langle \tau \rangle^{-p} \, \mathbb{E}^p \big[ \pmb{\varphi}_- \partial_t f \big](0) .$$
\end{remark}

 We set further the following energy fluxes, 
\begin{align*}
 \mathcal{F}^p[g] & \coloneqq \mathbb{E}^p \big[  \pmb{\varphi}_- \partial_t g \big] +\mathbb{E}^p \big[  \pmb{\varphi}_- \T_{|\slashed{p}|} g \big] + \mathbb{E}^p \big[ p_N \, \pmb{V}_{\,+}^{\mathrm{mod}} g \big] +\mathbb{E}^p \big[ |p_N|^{\frac{3}{4}} | r^{-1} \pmb{\varphi}_-|^{\frac{1}{4}} \pmb{V}_{\! +} g \big]  +\mathbb{E}^p \big[  \mathrm{D}_{p_{r^*}} g \big] , \\
 \mathcal{E}^p [g] & \coloneqq \mathcal{F}^p[g]+\mathbb{E}^p[p_Ng].
 \end{align*}

We now prove the main result of this section.

\begin{proposition}\label{Propdecayrp}
Let $p_0 \in \mathbb{N}^*$ and $f$ be a solution to $\T_g(f)=0$. For any $p \in \llbracket 1,p_0 \rrbracket$, we have
\begin{equation}\label{equation:daou}
\sup_{\tau \geq 0} \, \mathcal{E}^p[f](\tau) + \int_{\tau=0}^{+\infty}  \mathcal{E}^{p-1}[f](\tau) \dr \tau \lesssim  \mathcal{E}^p[f](0).
\end{equation}
Moreover, there holds
\begin{equation}\label{prop_decay_flux}
\forall \, \tau \geq 0, \qquad \mathcal{E}[f](\tau) \lesssim \frac{\mathcal{E}^{p_0}[f](0)}{\langle \tau \rangle^{p_0}} .
\end{equation} 
\end{proposition}

\begin{proof}
Let us observe first that \eqref{prop_decay_flux} is a consequence of \eqref{equation:daou} and Lemma \ref{Lemrpmethod}. Since the exterior of Schwarzschild black hole is static or, alternatively, by applying \eqref{equation:daou} for the foliation $(\Sigma'_\tau)_{\tau \geq 0}\coloneqq(\Sigma_{\tau+\tau_1})_{\tau \geq 0}$, we get
$$ \forall \, \tau_1 \geq 0, \qquad  \sup_{\tau \geq \tau_1} \, \mathcal{E}^p[f](\tau) + \int_{\tau=\tau_1}^{+\infty}  \mathcal{E}^{p-1}[f](\tau) \dr \tau \lesssim  \mathcal{E}^p[f](\tau_1).$$
The stated decay rate of the energy flux then follows from $\mathcal{E}[f]=\mathcal{E}^0[f]$ and Lemma \ref{Lemrpmethod}.
According to Proposition \ref{Rksystemderiv} and Proposition \ref{ProILEDrpbis}, for any $p \in \llbracket 1,p_0 \rrbracket$, we have
\begin{equation}\label{eq:Frp}
\sup_{\tau \geq 0}\,  \mathcal{F}^p[f](\tau) + \int_{\tau =0}^{+\infty}  \mathcal{F}^{p-1}[f](\tau) \dr \tau \lesssim  \mathcal{F}^p[f](0).
\end{equation}
Next, by Corollary \ref{Corrp}, where the cutoff function $\overline{\chi}$ is introduced, and $\T(p_t f)=0$, we have
$$ \T_g \Big(  \overline{\chi}(r) \frac{r^{p}}{\Omega^{2p}} \Big|\frac{p_v}{p_t}\Big|^{\lceil \frac{p}{2} \rceil} | p_t f | \Big) \lesssim |p_t|^2| f| \mathds{1}_{4M \leq r \leq 7M}- r^{p-1} \Big|\frac{p_v}{p_t}\Big|^{\lceil \frac{p-1}{2} \rceil}|p_v||p_t f| \mathds{1}_{r \geq 7M}.$$
Hence, by the energy estimate of Proposition \ref{Proenergy} and since $2\Omega^2 \geq 1$ for $ r \geq 4M$, we have, for all $\tau \geq 0$,
\begin{align}
 \mathbb{E} \bigg[ r^{p} \Big|\frac{p_v}{p_t}\Big|^{\lceil \frac{p}{2} \rceil} p_t f \, \mathds{1}_{r \geq 7M} \bigg] (\tau)+\int_{\pi^{-1}(\mathcal{R})}  \!  r^{p-1} & \Big|\frac{p_v}{p_t}\Big|^{\lceil \frac{p-1}{2} \rceil}|p_v||p_t f|  \mathds{1}_{r \geq 7M} \dr \mu_{\mathcal{P}} \label{eq:pesti} \tag{Eq-p}\\
 & \qquad \quad \lesssim \mathbb{E} \bigg[ r^{p} \Big|\frac{p_v}{p_t}\Big|^{\lceil \frac{p}{2} \rceil} p_t f \, \mathds{1}_{r \geq 4M} \bigg] (0)+\int_{\pi^{-1}(\mathcal{R})} \!  \frac{ |p_t|^2}{r^2} |f| \dr \mu_{\mathcal{P}}   . \nonumber
 \end{align}
Recall now that $p_{n_{\Sigma_\tau}} = p_v$ for $r >R_0$. By adding the estimate of Proposition \ref{ProILEDwrd2} multiplied by a sufficiently large constant to \eqref{eq:pesti} and (Eq-1), we have
$$ \sup_{\tau \geq 0}\,\mathbb{E}^p [p_N f](\tau) +\int_{\tau =0}^{+\infty} \mathbb{E}^{p-1} [p_N f](\tau) \dr \tau  \lesssim \mathbb{E}^p [p_N f](0)+\mathcal{E}[f](0).$$
The result follows by \eqref{eq:Frp}.
\end{proof}

\begin{remark}
The estimate \eqref{prop_decay_flux} can be extended to the case $p_0\in \R_+$ instead of $p_0\in \mathbb{N}^*$ (see Remark \ref{Rkrpgen}).
\end{remark}

Finally, we prove that if the initial data is exponentially decaying, then one can prove an integrated energy decay estimate that leads to exponential decay for the energy flux. For $b>0$, we introduce
$$ \mathbb{E}_{\mathrm{exp}}^b[g]  \coloneqq \mathbb{E} \Big[ e^{br \frac{p_v}{p_t}} \, g \Big]$$
and we define accordingly $\mathcal{F}_{\mathrm{exp}}^b [g]$ as well as $\mathcal{E}_{\mathrm{exp}}^b[g]$.

\begin{proposition}\label{Proexpodec}
There exists a constant $0< b_0 \leq b$, depending on $M$ and $b$, such that
$$ \forall \, \tau \geq 0, \qquad \mathcal{E}[f](\tau) \lesssim e^{-b_0 \tau} \mathcal{E}_{\mathrm{exp}}^b[f](0).$$
\end{proposition}

\begin{proof}
Note first, in view of the support of the cutoff function $\overline{\chi}$, 
$$e^{b \, \overline{\chi}(r) \frac{rp_v}{\Omega^2p_t}} \sim e^{b  \frac{rp_v}{p_t}} \quad \text{for $r \leq 2M$,}  \qquad \qquad \qquad  e^{b  \frac{rp_v}{p_t}} \sim 1 \quad \text{for $r \leq \max (7M,R_0)$}.$$ 
Hence, according to Corollary \ref{Corrp}, applied for $p=1$, we have
$$ \T_g \Big( e^{b \, \overline{\chi}(r) \frac{rp_v}{\Omega^2p_t}} \Big)  = b\T_g \Big(\overline{\chi}(r) \frac{rp_v}{\Omega^2p_t} \Big) e^{b \, \overline{\chi}(r) \frac{rp_v}{\Omega^2p_t}}  \lesssim |p_N| \mathds{1}_{r \leq 7M} -|p_v| e^{b  \frac{rp_v}{p_t}} \mathds{1}_{r \geq 7M} .$$
We then get from Proposition \ref{Rksystemderiv} that there exists a sufficiently large constant $A>0$ such that
\begin{equation}
 \sum_{1 \leq k \leq 7} c_k \T_g \Big[\Big( A+e^{b \, \overline{\chi}(r) \frac{rp_v}{\Omega^2p_t}} \Big)|g_k| \Big] \lesssim  -\Big(|p_N|  \mathds{1}_{r \leq R_0} + |p_v|e^{b  \frac{rp_v}{p_t}}  \mathds{1}_{r \geq R_0} \Big) \sum_{1 \leq k \leq 7} |g_k| .
\end{equation} 
Recall that $p_{n_{\Sigma_\tau}} \sim p_N$ for $r \leq R_0$ and $p_{n_{\Sigma_\tau}} = p_v$ for $r >R_0$. By the energy estimate of Proposition \ref{Proenergy}, we then obtain
$$ \sup_{\tau \geq 0}\, \mathcal{F}_{\mathrm{exp}}^b [f](\tau)+\int_{\pi^{-1}(\mathcal{R})} |p_{n_{\Sigma_\tau}}|e^{b \frac{rp_v}{p_t}}   \sum_{1 \leq k \leq 7} |g_k|   \dr \mu_{\mathcal{P}} \lesssim \mathcal{F}_{\mathrm{exp}}^b [f](0).$$
Similarly, one has
$$ \sup_{\tau \geq 0} \,\mathbb{E}_{\mathrm{exp}}^b [p_t f](\tau)+\int_{\pi^{-1}(\mathcal{R})} |p_{n_{\Sigma_\tau}}|e^{b  \frac{rp_v}{p_t}} |p_t| |f|  \dr \mu_{\mathcal{P}} \lesssim \mathbb{E}_{\mathrm{exp}}^b [p_tf](0)+\int_{\pi^{-1}(\mathcal{R})\cap \{ r \leq R_0 \}} |p_N| |p_t| |f|  \dr \mu_{\mathcal{P}}.$$
The last term on the RHS is bounded by $\mathcal{E}[f](0) \lesssim \mathcal{E}_{\mathrm{exp}}^b[f](0)$ according to Proposition \ref{ProILEDwrd2}. This yields,
$$ \sup_{\tau \geq 0} \mathcal{E}_{\mathrm{exp}}^b [f](\tau)+\int_{\tau =0}^{+\infty} \mathcal{E}_{\mathrm{exp}}^b [f](\tau) \dr \tau  \lesssim \mathcal{E}_{\mathrm{exp}}^b [f](0).$$
Using again the time-invariance of the exterior of Schwarzschild black hole, we get that the same estimate holds but with initial time $\tau_1$ instead of $0$, for all $\tau_1 \geq 0$. This directly implies the exponential decay.
\end{proof}
 
\section{Pointwise decay estimates}\label{Sec8}

We start by collecting some relations which will be useful later.

\begin{lemma}\label{identityVplusp}
There holds
\begin{alignat*}{2}
\bigg|\pmb{V}_{\, +}^{ \mathrm{mod}} (p_t)+\frac{r-3M}{r^2}p_t\partial_{p_{r^*}} \pmb{\varphi}_-\bigg| & \lesssim |p_t| \qquad &&\text{on $\pi^{-1}(\mathcal{R})$,} \\
\big| \pmb{V}_{\, +}^{ \mathrm{mod}} (p_N) \big| & \lesssim |p_N| \qquad &&\text{on $\pi^{-1}(\mathcal{R})$,} \\
\big| \pmb{V}_{\, +}^{ \mathrm{mod}} (p_v) \big| & \lesssim |p_v| \qquad &&\text{on $\{ r \geq R_0 \}$}.
\end{alignat*}
Moreover, for any complete lift $ \widehat{K} \in \{\partial_t, \widehat{\mathbf{\Omega}}_i,\widehat{\mathbf{\Omega}}_2,\widehat{\mathbf{\Omega}}_3 \}$, we have
$$ \widehat{K}(r)= \widehat{K}(p_{r^*}) = \widehat{K}(|\slashed{p}|) = \widehat{K}(p_t) = \widehat{K}(p_N)=0 .$$
\end{lemma}
\begin{remark}\label{Rkpnpt}
Note that the first relation implies $|\pmb{V}_{\, +}^{ \mathrm{mod}} (p_t)| \lesssim |p_N|$.
\end{remark}
\begin{proof}
We start with the first inequality. Since
\begin{equation*}
  r\partial_r (p_t) = \frac{r}{2}\partial_r \bigg( \frac{\Omega^2}{r^2} \bigg) \frac{|\slashed{p}|^2}{p_t} = - \frac{(r-3M)|\slashed{p}|^2}{r^3 p_t}  ,
\end{equation*}
 we have
\begin{equation}\label{eq:Vpluspt}
\pmb{V}_{\, +}^{ \mathrm{mod}} (p_t)= -\frac{(r-3M)|\slashed{p}|^2}{r^3 p_t} -\bigg(\frac{r-3M}{r^2} \pmb{\varphi}_- +\frac{27M^2 }{r^{\frac{3}{2}}|r+6M|^{\frac{1}{2}}}p_t \bigg) \frac{p_{r^*}}{p_t}+\Phi  r \pmb{\varphi}_- \frac{p_{r^*}}{p_t} .   
 \end{equation}
 Since $r |\pmb{\varphi}_- \Phi| \lesssim |p_t|$ by Proposition \ref{ProboundPsi}, the last term is bounded above by $|p_t|$. The same holds true for the third one.  Finally, we compute
 $$ p_t \partial_{p_{r^*}} \pmb{\varphi}_- - \frac{p_{r^*}}{p_t} \pmb{\varphi}_- =r\frac{|p_t|^2-|p_{r^*}|^2}{\Omega^2 p_t}=\frac{|\slashed{p}|^2}{rp_t}   .$$
 Note now that
 \begin{equation*}
\pmb{V}_{\, +}^{ \mathrm{mod}} \big( p_t \pm p_{r^*} \big)= -\frac{4(r-3M)p_up_v}{r \Omega^2 p_t} -\bigg(\frac{r-3M}{r^2} \pmb{\varphi}_- +\frac{27M^2 }{|r+6M|^{\frac{1}{2}}r^{\frac{3}{2}}}p_t-\Phi  r \pmb{\varphi}_- \bigg) \frac{p_{r^*}\pm p_t}{p_t} ,   
 \end{equation*}
 from which we directly get the third inequality since $2p_v=p_t+p_{r^*}$. For the second one, we assume that $r \leq 2.7M$ since $p_N=p_t$ otherwise. From the last equality, one gets
 \begin{align}
\nonumber \pmb{V}_{\, +}^{ \mathrm{mod}} \Big( \, \frac{2r^2 }{\Omega^2}p_u \Big) &= \frac{4r(r-3M)p_u}{\Omega^4} -\frac{4r(r-3M)p_up_v}{ \Omega^4 p_t} -\bigg((r-3M) \pmb{\varphi}_- +\frac{27M^2r^{\frac{1}{2}} }{|r+6M|^{\frac{1}{2}}}p_t-\Phi  r^3 \pmb{\varphi}_- \bigg) \frac{2|p_u|}{\Omega^2p_t} \\
 & = \frac{2(r-3M)|p_u|}{\Omega^2p_t}\bigg(  \frac{2r|p_u|}{\Omega^2}- \pmb{\varphi}_- \bigg)-\frac{54M^2r^{\frac{1}{2}}|p_u| }{|r+6M|^{\frac{1}{2}} \Omega^2 }+  r^3 \pmb{\varphi}_-\frac{2|p_u|}{\Omega^2p_t} \Phi. \label{eq:VpluspN}
 \end{align}
 By Remark \ref{RkvarphiHplus}, the second factor of the first term in the RHS is bounded by $|p_t|$. According to Proposition \ref{ProboundPsi}, the last term is bounded by $|p_N|$. We then deduce that
 $$ \Big| \pmb{V}_{\, +}^{ \mathrm{mod}} \Big( \, \frac{2r^2 }{\Omega^2}p_u \Big) \Big| \lesssim |p_N| \qquad \text{on $\{ r \leq 2.7M \}$}.$$
 In view of the definition \eqref{eq:defpN} of $p_N$ and Remark \ref{Rkpnpt}, it implies the second inequality. Next, for any $1 \leq i \leq 3$, there exists a spherical coordinate system $(\theta',\phi')$ in which $\widehat{\mathbf{\Omega}}_i=\partial_{\phi'}$, from which we directly deduce $\widehat{\mathbf{\Omega}}_i(|\slashed{p}|)=0$. The identities for the complete lifts of the Killing vector fields follow from direct computations.
\end{proof}

 We recall a standard result, proved for instance in \cite[Lemma $5.2$]{L23}.
 
\begin{lemma}\label{Sobsphere}
Let $h : \mathcal{P} \to \R$ be a sufficiently regular function. Then,
$$ \forall \, (t,r^*) \in \R^2, \qquad \bigg\| \int_{\mathcal{P}_x} |h| \, \dr \mu_{\mathcal{P}_x} \bigg\|_{L^\infty(\mathbb{S}^2)}(t,r^*) \lesssim \sum_{|I| \leq 2} \int_{\mathbb{S}^2} \int_{\mathcal{P}_x} \big| \widehat{\mathbf{\Omega}}^I h \big| \dr \mu_{\mathcal{P}_x} \dr \mu_{\mathbb{S}^2} .$$
\end{lemma}

The next step consists in proving the next result. 

\begin{proposition}\label{ProdecaypointBoundedregion}
Let $h : \mathcal{P} \to \R$ be a sufficiently regular function and 
$$ h^*(t^*,r,\theta,\phi,p)=h(t^*-2M\log(r-2M),r^*(r),\theta,\phi,p).$$ 
Let $x \in \mathcal{R} $ and $\tau \geq 0$ such that $x \in \Sigma_\tau$. We have
$$ r(x)^2 \int_{\mathcal{P}_x} |h|  \dr \mu_{\mathcal{P}_x}  \lesssim \sum_{n+q \leq 1} \sum_{|I| \leq 2} \int_{r= r(x)}^{ r(x)+M }  \int_{\mathbb{S}^2_\omega} \int_{\mathcal{P}_x} \Big| \Big[ \pmb{V}_{\, +}^{ \mathrm{mod},n} \partial_t^q \widehat{\mathbf{\Omega}}^I h \Big]^* \Big| \big( t^*(x),r,\omega,p \big) \dr \mu_{\mathcal{P}_x} r^2 \dr \mu_{\mathbb{S}^2_\omega} \dr r .$$
\end{proposition}
\begin{remark}\label{Rkdomain}
The radial domain of the integral could be chosen to be $[r(x)-M,r(x)]$ if $r(x) \geq 3M$.
\end{remark}
\begin{proof}
It will be convenient to work with $g\coloneqq h/p_t$. We have
\begin{align} \nonumber
&\partial_r\big|g^*\big|(t^*,r,\theta,\phi,p)  = \bigg[ \, -\frac{2M}{r \Omega^2} \partial_t | g|+\frac{1}{\Omega^2} \partial_{r^*} |g | \bigg] \big( t^*-2M\log(r-2M),r^*(r),\theta,\phi,p \big) \\
  & \quad \;  = \frac{1}{r} \pmb{V}_{\, +}^{ \mathrm{mod}}\big(|g| \big)-\frac{|r+6M|^{\frac{1}{2}}(r-3M)+2Mr^{\frac{1}{2}}}{r^{\frac{3}{2}}\Omega^2} \partial_t |g|  +\bigg(\frac{r-3M}{r^3} \pmb{\varphi}_- +\frac{27M^2 }{|r+6M|^{\frac{1}{2}}r^{\frac{5}{2}}}p_t - \Phi   \pmb{\varphi}_-\bigg) \partial_{p_{r^*}}|g| ,  \nonumber
 \end{align}
 where the functions in the RHS are evaluated at $( t^*-2M\log(r-2M),r^*(r),\theta,\phi,p)$.
 
As $|p_t|r^2 \dr \mu_{\mathcal{P}_x}= \sin^{-1} (\theta)\dr p_{r^*} \dr p_\theta \dr p_\phi$, we have by a one dimensional local Sobolev inequality,
\begin{align*}
r(x)^2 \! \int_{\mathcal{P}_x} |p_t g| \, \dr \mu_{\mathcal{P}_x} & \lesssim \int_{r=r(x)}^{r(x)+M} \int_{\mathcal{P}_x} \big|p_t g^*\big| \, \dr \mu_{\mathcal{P}_x} r^2 \dr r+\int_{r=r(x)}^{r(x)+M}  \bigg|  \int_{\mathcal{P}_x}|p_t|\partial_r\big(\big|  g^*\big|\big)  \dr \mu_{\mathcal{P}_x} \bigg| r^2 \dr r .
\end{align*}
Performing integration by parts in $p_{r^*}$ and using that $|r+6M|^{\frac{1}{2}}(r-3M)+2Mr^{\frac{1}{2}}$ vanishes for $r=2M$, 
\begin{align*}
\bigg|  \int_{\mathcal{P}_x}|p_t|\partial_r\big(\big| g^* \big|\big)  \dr \mu_{\mathcal{P}_x} \bigg| & \lesssim \frac{1}{r} \int_{\mathcal{P}_x}|p_t| \bigg| \Big[ \pmb{V}_{\, +}^{ \mathrm{mod}} g \Big]^* -\frac{r-3M}{r^2} \partial_{p_{r^*}}(\pmb{\varphi}_-) g^*  \bigg| \dr \mu_{\mathcal{P}_x} + \int_{\mathcal{P}_x}|p_t| \Big| \big[ \partial_t g \big]^* \Big| \dr \mu_{\mathcal{P}_x} \\
& \quad  +  \int_{\mathcal{P}_x} |p_t|\big(1+ |\Phi \partial_{p_{r^*}} \pmb{\varphi}_-|+|\pmb{\varphi}_- \partial_{p_{r^*}} \Phi |\big) \big| g^* \big| \dr \mu_{\mathcal{P}_x}.
\end{align*}
Recall now from Proposition \ref{ProboundPsi} that $\pmb{\varphi}_- \partial_{p_{r^*}} \Phi \in L^\infty_{x,p} (\pi^{-1}(\mathcal{R}))$. We further have $|p_t\partial_{p_{r^*}} \pmb{\varphi}_-| \lesssim r|p_N|$ since $p_t \partial_{p_{r^*}}$ and $\pmb{\varphi}_-$ are both smooth up to $\H$. This can also be obtained by using Remark \ref{RkvarphiHplus}. Consequently,
$$  \big| \Phi \partial_{p_{r^*}} \pmb{\varphi}_- \big| \lesssim \frac{1}{|p_t|}|\Phi| \big| p_t \partial_{p_{r^*}} \pmb{\varphi}_- \big|  \lesssim \frac{|p_N|}{|p_t|}|r\Phi|  <+\infty   .$$
Finally, as $h=p_t g$, we get from Lemma \ref{identityVplusp}
$$ |p_t| \bigg|  \Big[ \pmb{V}_{\, +}^{ \mathrm{mod}} g \Big]^* -\frac{r-3M}{r^2} \partial_{p_{r^*}}(\pmb{\varphi}_-) g^*  \bigg| \lesssim \bigg|  \Big[ \pmb{V}_{\, +}^{ \mathrm{mod}} h \Big]^* \bigg|+\big|h^* \big| $$
and it remains to apply the Sobolev inequality of Lemma \ref{Sobsphere}.
\end{proof}

It allows us to deduce the next result. For convenience, with $n_{\{t^*= \textrm{cst}\}}$ the normal of a hypersurface of constant $t^*$, we define
\begin{equation}\label{eq:defnormal}
p_{t^*} \coloneqq p \big( n_{\{t^*= \textrm{cst} \}} \big) = \Big(1+\frac{2M}{r} \Big)^{-\frac{1}{2}} p_v+\Big(1+\frac{2M}{r}\Big)^{\frac{1}{2}}\frac{p_u}{\Omega^2} .
\end{equation}
Note in particular that $p_{t^*} \sim p_N$. We will also denote by $\dr \mu_{\{ t^*=\textrm{cst}\} }$ the induced volume form on the hypersurface $\{ t^*=\textrm{cst}\}$.

\begin{corollary}\label{Cordecaypoint}
Let $ g :\mathcal{P} \to \R$ be a sufficiently regular function, $\tau \in \R_+$ and $x \in \Sigma_\tau$. Then,
\begin{itemize}
\item if $r(x) \leq R_0$, we have
$$ r(x)^2 \int_{\mathcal{P}_x} |p_N|^2 |g| \dr \mu_{\mathcal{P}_x} \lesssim   \sum_{n+q \leq 1} \sum_{|I| \leq 2} \mathbb{E} \big[p_N \pmb{V}_{\, +}^{ \mathrm{mod},n} \partial_t^q \widehat{\mathbf{\Omega}}^I g \big] (\tau)  .$$
\item Otherwise $r(x) \geq R_0$ and we have, for $p \in \mathbb{N}$,
\begin{align*}
\color{white} \square \qquad \color{black}	 & r(x)^{2+p} \! \int_{\mathcal{P}_x} \Big|\frac{p_v}{p_t}\Big|^{\frac{p}{2}} |p_N|^2 |g| \dr \mu_{\mathcal{P}_x} \\
	& \qquad \qquad \lesssim   \sum_{n+q \leq 1} \sum_{|I| \leq 2} \int_{\{ t^*=t^*(x), \,  0\leq r(x)-r \leq M \}} \int_{\mathcal{P}_x} r^p\Big|\frac{p_v}{p_t}\Big|^{\frac{p}{2}}  |p_{t^*}| \Big| p_N \pmb{V}_{\, +}^{ \mathrm{mod},n} \partial_t^q \widehat{\mathbf{\Omega}}^I g \Big| \dr \mu_{\mathcal{P}_x} \dr \mu_{\{ t^*=t^*(x)\} }   .
	\end{align*}
\end{itemize}
\end{corollary}
\begin{proof}
 We apply the previous Proposition \ref{ProdecaypointBoundedregion} to $h=|p_N|^2 g$ or $h=r^p|p_v/p_t|^{p/2}|p_N|^2 g$ and we use $\partial_t(p_N)=0$, $\widehat{\mathbf{\Omega}}_i(p_N)=0$, $ \big|\pmb{V}_{\, +}^{ \mathrm{mod}} (p_N) \big| \lesssim |p_N|$ as well as $ \big|\pmb{V}_{\, +}^{ \mathrm{mod}} \big(r^p|p_v/p_t|^{p/2}\big) \big| \lesssim r^p |p_v/p_t|^{p/2}$, which follow from Lemma \ref{identityVplusp}. Note that if $r(x) \geq 3M$, we also use Remark \ref{Rkdomain}. Finally, we exploit $p_N \sim p_{t^*}$ and that $\Sigma_\tau \cap \{ r \leq R_0 \}$ is equal to $\{ t^* = \tau, \; r \leq R_0 \}$.
 \end{proof}

In order to be able to exploit this Sobolev inequality for $r(x) \geq R_0$, we will make use of the next result.

\begin{lemma}\label{LemSobVla}
Let $p \in \mathbb{N}$ and $f$ be a solution to $\T_g(f)=0$. For all $\tau \in \R_+$ and $x \in \Sigma_\tau \cap \{ r \geq R_0 \}$, we have
$$ \sum_{n \leq 1} \int_{\{ t^*=t^*(x), \,  0\leq r(x)-r \leq M \}} \int_{\mathcal{P}_x} r^p\Big|\frac{p_v}{p_t}\Big|^{\frac{p}{2}}  |p_{t^*}| \Big| p_N \pmb{V}_{\, +}^{ \mathrm{mod},n} f \Big| \dr \mu_{\mathcal{P}_x} \dr \mu_{\{ t^*=t^*(x)\} } \lesssim \mathcal{E}^p[f](\tau).$$
\end{lemma}
\begin{proof}
Let $y \in \mathcal{R}$ such that $(t^*(y),\theta (y), \phi (y) )= (t^* (x), \theta (x) , \phi (x) )$ and $r(y)=r(x)-M$. Then, there exists $\tau' \geq \tau$ such that $y \in \Sigma_{\tau'}$. We then consider the domain $\mathcal{D}$ with boundary
$$ \{ t^*=t^*(x), \,  0\leq r(x)-r \leq M \}, \qquad  \Sigma_{\tau'} \cap \{ r \leq r(x)-M \}, \qquad  \mathcal{H}^+ \cap \{ \tau \leq t^* \leq \tau' \}, \qquad  \Sigma_\tau \cap \{ r \leq r(x) \}.$$
Recall now the functions $g_k$, $1 \leq k \leq 7$ and Proposition \ref{Rksystemderiv}. In particular, $g_7 \sim p_N \pmb{V}_{\, +}^{ \mathrm{mod}} f$. Let further $g_0 \coloneqq \xi f$ and $c_0=1$. Then,
\begin{itemize}
\item if $p=0$, we use Proposition \ref{Rksystemderiv} and the divergence theorem, applied in $\mathcal{D}$ to $\mathrm{N}[c_k g_k]$.
\item If $ p \geq 1$, we apply the divergence theorem to 
$$\mathrm{N}\bigg[c_k \overline{\chi} (r) r^p \Big|\frac{p_v}{\Omega^2 p_t} \Big|^{\frac{p}{2}} g_k\bigg]$$
 and we use Corollary \ref{Corrp}. The error terms arising from the cutoff function $\overline{\chi}$ are handled through the ILED of Proposition \ref{ProILEDwrd2} since $\pi^{-1}(\mathcal{D}) \subset \pi^{-1}(\mathcal{P})$.
\end{itemize}
\end{proof}

We are now able to get the main result of this section, from which one can obtain decay estimates for the null components of the energy-momentum tensor $\mathrm{T}[f]$. In particular, by using Proposition \ref{Propdecayrp} (respectively Proposition \ref{Proexpodec}), one obtains Corollary \ref{Corintro} (respectively Corollary \ref{Corintro2}).
 
 \begin{corollary}\label{Cordecay}
Let $f : \mathcal{P} \to \R$ be a solution to the massless Vlasov equation. Then, for all $x \in \mathcal{R}$,
$$  \mathrm{T}_{NN}\big[|f|\big] (x) \lesssim \frac{1}{r(x)^2} \sum_{q+|I| \leq 3} \mathcal{E}\big[ \partial_t^q \widehat{\mathbf{\Omega}}^I f \big] \big(\tau (x) \big) .$$
Consider futher $p \in \mathbb{N}$. We have
 $$ \int_{\mathcal{P}_x} \frac{|p_v|^p}{|p_t|^p} |f| |p_N|^2  \dr \mu_{\mathcal{P}_x} \lesssim \frac{1}{r(x)^2 \langle \tau (x)+r(x) \rangle^p} \sum_{q+|I| \leq 3} \mathcal{E}^p\big[ \partial_t^q \widehat{\mathbf{\Omega}}^I f \big] (0).$$
 \end{corollary}
 \begin{proof}
For the first estimate, we apply the previous Corollary \ref{Cordecaypoint} and Lemma \ref{LemSobVla} to $ \partial_t^q \widehat{\mathbf{\Omega}}^I f$, for any $q+|I| \leq 3$. For the second one we further use Proposition \ref{Propdecayrp} providing the decay of the energy flux.
 \end{proof}

\addtocontents{toc}{\protect\setcounter{tocdepth}{2}}

\appendix

\section{   Uniform bounds for the correction term}\label{Append}

In this appendix, we prove the next result which implies Proposition \ref{ProboundPsi}. 

\begin{proposition}\label{ProboundPsibis}
Let $\Psi \coloneqq (r+6M)\Phi$ and $0 < \delta \leq 1$. Then, we have
$$ \sup_{  \pi^{-1} (\{t^* \geq 0 \})} \,\frac{|p_N|+r^2|p_v|}{|p_t|}|\Psi|+\frac{1}{r}\big|\pmb{\varphi}_- \partial_{p_{r^*}} \Psi \big| < +\infty.$$
\end{proposition}
\begin{remark}\label{RkboundPsi}
In view of Remark \ref{RkvarphiHplus}, we have $|r \pmb{\varphi}_-| \lesssim r^2|p_v|+|p_t|$ for $r \geq 3M$. As $|r \pmb{\varphi}_-| \lesssim |p_N|$ for $r \leq 3M$, it implies $|r \pmb{\varphi}_-|^\delta |p_t|^{1-\delta} \lesssim |p_N|+r^2|p_v|$ for any $0 \leq \delta \leq 1$.
\end{remark}

\subsection{Preliminaries} 
 
The study of $\Psi$ and its derivatives will be based on $L^\infty_{x,p}$ estimates relying on Duhamel formula (see Lemma \ref{LemDuhamel} below). For this purpose, we need to introduce the flow map defined by $\T_g$.
\begin{definition}
For $\tau^*_0\in \R$ and $y \in \pi^{-1}( \{t^*=\tau_0^* \})$, we define 
$$\tau^* \mapsto \Phi_{\tau^*} (\tau_0^*,y)$$
 to be the flow of $\T_g$ parametrised by $t^*$, so that $\Phi_{\tau_0^*}(\tau_0^*,y)=y$. 
\end{definition}
We also recall from \eqref{eq:defnormal} the notation $p_{t^*}$ for the contraction of $p$ with the normal of a hypersurface of constant $t^*$. It verifies $p_{t^*} \sim p_N$.

We now prove a general estimate for a solution to a damped massless Vlasov equation.

\begin{lemma}\label{LemDuhamel}
Let $\tau_1^* \in \R$, and let $\Xi \colon \mathcal{P} \to \R$ be a function such that
$$ \T_g (\Xi)+ \pmb{d}(x,p) \Xi=\pmb{s}(x,p), \qquad \qquad  \sup_{ \pi^{-1}( \{ t^* = \tau_1^* \}) }\big| \Xi \big| <+\infty  ,$$
where the damping term $\pmb{d} \colon \mathcal{P} \to \R_+$ and the source term $\pmb{s} \colon \mathcal{P} \to \R$ verify
$$ \forall \, (x,p) \in \mathcal{P}, \qquad \big|\pmb{s} (x,p) \big| \lesssim \pmb{d}(x,p) .$$
Then, there exists a constant $C>0$ depending only on $\pmb{d}$ and $\pmb{s}$ such that
$$ \forall \, \tau_2^* \geq \tau_1^*,\qquad \qquad \sup_{\pi^{-1} ( \{\tau_1^* \leq t^* \leq \tau_2^* \})} \big| \Xi \big|  \leq \,C \,+\sup_{ \pi^{-1}( \{ t^* = \tau_1^* \}) }\big| \Xi \big|   . $$
\end{lemma}

\begin{proof}
We parametrise the flow of $\T_g$ by the variable $t^*$, so we introduce
$$ \underline{\pmb{d}} \coloneqq \frac{\pmb{d}}{|p_{t^*}|}, \qquad \qquad \underline{\pmb{s}} \coloneqq  \frac{\pmb{s}}{|p_{t^*}|}.$$ 
We further consider $\tau^* \geq \tau_1^*$ and $y \in \pi^{-1}( \{t^*=\tau^* \})$. According to Duhamel's formula, we have
$$ \Xi (y) = e^{-\int_{z=\tau_1^*}^{\tau^*} \underline{\pmb{d}}  \circ \Phi_z(\tau^*,y) \dr z} \, \Xi \circ \Phi_{\tau_1^*}(\tau^*,y)+ \int_{q=\tau_1^*}^{\tau^*} e^{-\int_{z=q}^{\tau^*} \underline{\pmb{d}}  \circ \Phi_z(\tau^*,y) \dr z} \,  \underline{\pmb{s}}  \circ \Phi_{q}(\tau^*,y) \dr q .$$
It remains to remark that, by assumption,
$$ \bigg|\int_{q=\tau_1^*}^{\tau^*} \! e^{-\int_{z=q}^{\tau^*} \underline{\pmb{d}}  \circ \Phi_z(\tau^*,y) \dr z} \,  \underline{\pmb{s}}  \circ \Phi_{q}(\tau^*,y) \dr q \bigg| \lesssim \int_{q=\tau_1^*}^{\tau^*} \! e^{-\int_{z=q}^{\tau^*} \underline{\pmb{d}}  \circ \Phi_z(\tau^*,y) \dr z} \,  \underline{\pmb{d}}  \circ \Phi_{q}(\tau^*,y) \dr q=1- e^{-\int_{z=\tau_1^*}^{\tau^*} \!  \underline{\pmb{d}} \circ \Phi_z(\tau^*,y) \dr z}  .$$
\end{proof}

\subsection{Pointwise estimate of $\Psi$}
 
We compute first the damped Vlasov equation verified by $\Psi =(r+6M)\Phi$.

\begin{lemma}\label{LembasicPsi}
The function $\Psi$ is uniquely determined as the solution to
\begin{equation*}
\T_g ( \Psi )+ \overline{\pmb{a}} (r,p_{r^*},p_t) \Psi=  \pmb{b}(r,p_t), \qquad \qquad \Psi\vert_{\pi^{-1}(\{t^*=0\})} =0,
\end{equation*}
where $\overline{\pmb{a}} \colon \mathcal{P} \to \R_+$ and $\pmb{b} \colon \mathcal{P} \to \R$ are given by
\begin{align*}
\overline{\pmb{a}}(r,p_{r^*},p_t) & \coloneqq  \frac{r^2+2Mr+3M^2}{r|r+6M|^{\frac{1}{2}}(r^{\frac{3}{2}}+M|r+6M|^{\frac{1}{2}}) }|p_t| +\frac{2M}{r^2}\frac{|p_u|}{\Omega^2}-\frac{p_{r^*}}{r+6M} , \\
\pmb{b}(r,p_t) & \coloneqq \frac{(r+3M)|p_t|}{r^{\frac{5}{2}}|r+6M|^{\frac{1}{2}} }.
\end{align*}
Moreover, they satisfy the bounds
$$ \overline{\pmb{a}}(r,p_{r^*},p_t) \geq  \frac{6M |p_t|}{r|r+6M|+r^{\frac{1}{2}} |r+6M|^{\frac{3}{2}}  } + \frac{2M}{r^2}\frac{|p_u|}{\Omega^2} , \qquad \quad \big|\pmb{b}(r,p_t) \big| \leq \frac{2}{r^2}|p_t| . $$
\end{lemma}
\begin{proof}
Recall from Definition \ref{DefVminusmod} the equation verified by $\Omega\Phi$. Then, using the computation performed in \eqref{eq:kevatalenn1}, we have $\T_g(\Phi )+\pmb{a} \Phi=(r+6M)^{-1} \pmb{b}$. The first part of the statement then follows from $\T_g(r+6M)=p_{r^*}$. Next, as $r^2+2Mr+3M^2 \geq r^{2}+Mr^{\frac{1}{2}}|r+6M|^{\frac{1}{2}}$ and $|p_{r^*}| \leq |p_t|$, we have
$$ \overline{\pmb{a}}(r,p_{r^*},p_t) \geq  \frac{|p_t|}{r^{\frac{1}{2}} |r+6M|^{\frac{1}{2}}  }-\frac{|p_t|}{r+6M} + \frac{2M}{r^2}\frac{|p_u|}{\Omega^2} ,$$
implying the lower bound for $\overline{\pmb{a}}$. The estimate for $\pmb{b}$ follows from $r+3M \leq 2r^{\frac{1}{2}}|r+6M|^{\frac{1}{2}}$.
\end{proof}
 
We are now able to control $\Psi$ pointwise. 
 
\begin{proposition}\label{ProboundPsiA}
There holds
$$ \sup_{  \pi^{-1} (\{t^* \geq 0 \})} \, \frac{|p_N|}{|p_t|}|\Psi|+\frac{r^2|p_v|}{|p_t|}|\Psi| < +\infty.$$
\end{proposition}
\begin{proof}
Let us recall from Proposition \ref{Promultiplierm} the properties of the weight $\xi$ capturing the redshift effect. In particular we will use that $\xi \sim p_N$ and $-\T_g(|\xi|) \geq 0$. Thus, we have
\begin{equation*}
\T_g \Big( \, \frac{|\xi|}{|p_t|} \Psi \Big)+ \widetilde{\pmb{a}} (r,p_{r^*},p_t)\frac{|\xi|}{|p_t|}  \Psi=  \frac{|\xi|}{|p_t|}\pmb{b}(r,p_t),
\end{equation*}
where, using $|p_N|\sim |p_t|+\Omega^{-2} |p_u|$ and Lemma \ref{LembasicPsi},
$$ \widetilde{\pmb{a}} (r,p_{r^*},p_t) \coloneqq \overline{\pmb{a}} (r,p_{r^*},p_t)-|\xi|^{-1}\T_g(|\xi|) \geq \overline{\pmb{a}} (r,p_{r^*},p_t) \gtrsim \frac{|p_N|}{r^2} .$$
Consequently, we have
$$   \frac{|\xi|}{|p_t|} \big| \pmb{b} (r,p_t) \big| \lesssim \frac{|p_N|}{r^2} \lesssim \widetilde{\pmb{a}} (r,p_{r^*},p_t) .$$
The first estimate ensues from Lemma \ref{LemDuhamel}, since $\Psi$ vanishes initially. Note that we only need to prove the second one for the region $\{ r \geq 7M \}$. Recall the cutoff function $\overline{\chi} \in C^\infty (\R)$ which satisfies $\overline{\chi} (s) =0$ for $s \leq 4M$ and $\overline{\chi}(s) =1$ for $s \geq 7M$. Then, by Lemma \ref{Lemrp}, we have
\begin{equation*}
\T_g \Big( \overline{\chi} (r) \frac{r^2|p_v|}{\Omega^2|p_t|} \Psi \Big)+ \widehat{\pmb{a}} (r,p_{r^*},p_t)\frac{|\xi|}{|p_t|}  \Psi= \widehat{\pmb{b}}(r,p_{r^*},p_t),
\end{equation*}
where 
$$
\widehat{\pmb{a}} (r,p_{r^*},p_t)\coloneqq \overline{\pmb{a}} (r,p_{r^*},p_t)+2\overline{\chi}(r)\frac{r-3M}{\Omega^4|p_t|}|p_v|^2, \qquad \quad \widehat{\pmb{b}}(r,p_{r^*},p_t)\coloneqq p_{r^*} \overline{\chi}'(r)\frac{r^2|p_v|}{\Omega^2|p_t|} \Psi+ \overline{\chi} (r) \frac{r^2|p_v|}{\Omega^2|p_t|}\pmb{b}(r,p_t).
$$
In particular, by support considerations and the boundednes of $\Psi$,
$$  \big| \widehat{\pmb{b}} \big|(r,p_{r^*},p_t) \lesssim |p_v| \mathds{1}_{r \geq 4M} .$$
As $\widehat{\pmb{a}} \geq 0$, the result would follow from Duhamel formula if we could prove that
$$\forall \, y \in \pi^{-1}(\{t^*=0\}), \qquad \quad \int_{s=0}^{+\infty}  \bigg[ \frac{p_v}{p_{t^*}} \mathds{1}_{r \geq 4M} \bigg] \circ \Phi_s (0,y) \dr s \leq C, $$
for a constant $C>0$ independent of $y$. For this, one can simply exploit that, according to Proposition \ref{Promultiplierm} and Corollary \ref{Corrp}, for $A>0$ large enough,
$$ \T_g \Big(A \xi + \overline{\chi} (r) \frac{r^2|p_v|}{\Omega^2|p_t|} \Psi \Big) \lesssim -A\frac{|r-3M|}{r^3}|p_t|+ |p_t|\mathds{1}_{r \leq 7M}-r|p_v| \mathds{1}_{r \geq 7M} \lesssim -|p_v| \mathds{1}_{r \geq 4M}.$$
\end{proof} 
 
\subsection{Boundedness for the derivatives of $\Psi$} 
 
The analysis performed to estimate derivatives of $\Psi$ shares some similarities with the one carried out in Section \ref{SecILEDwrd}. However, since we will control $L^\infty_{x,p}$ norms of the derivatives of $\Psi$ along null geodesics, we are able to separate the estimates in the different regions of the null-shell in a cleaner way. We recall that the final goal consists in estimating $r^{-1}\pmb{\varphi}_- \partial_{p_{r^*}} \Psi$.

\subsubsection{Preparatory results}
We will have to commute the damped massless Vlasov equation satisfied by $\Psi$, by the vector fields $\partial_t$, $\T_{|\slashed{p}|}$, $\pmb{V}_{\! +}$ and $ \pmb{ \varphi}_- \partial_{p_{r^*}}$. For this, we will use that
\begin{equation}\label{eq:ComZ}
 \T_g \big( Z \Psi \big)+ \overline{\pmb{a}}(r,p_{r^*},p_t) Z \Psi = [\T_g,Z] (\Psi)- Z \big[\overline{\pmb{a}}(r,p_{r^*},p_t) \big] \Psi+Z \big[\pmb{b} (r,p_t) \big]   ,
 \end{equation}
for any vector field $Z$. Then, we estimate $Z\overline{\pmb{a}}$ and $Z\pmb{b}$ for all the derivatives that we will consider. 

\begin{lemma}\label{Lemsourcetermesti}
There holds
$$ \partial_t \overline{\pmb{a}}=\T_{|\slashed{p}|} \overline{\pmb{a}}= \partial_t \pmb{b}=\T_{|\slashed{p}|} \pmb{b}=0 .$$
Moreover, the following estimates hold,
\begin{alignat*}{2}
\big| p_t \partial_{p_{r^*}} \big[  \overline{\pmb{a}} (r,p_{r^*},p_t) \big] \Psi \big| & \lesssim \frac{|p_t|}{r}, \qquad \qquad \qquad     \big| p_t \partial_{p_{r^*}} \big[  \pmb{b} (r,p_t) \big] \big| && \lesssim \frac{|p_t|}{r^2} , \\
\big| \pmb{V}_+ \big[  \overline{\pmb{a}} (r,p_{r^*},p_t) \big] \Psi \big| & \lesssim  \frac{|p_t|}{r}, \qquad \qquad \qquad  \quad \big|  \pmb{V}_+ \big[  \pmb{b} (r,p_t) \big] \big| &&\lesssim \frac{|p_N|}{r^2}     .
\end{alignat*}
\end{lemma}

\begin{proof}
The first identities follow from the relations $Z(p_{r^*})=Z(p_t)=Z(r)=0$ for $Z \in \{ \partial_t, \T_{|\slashed{p}|} \}$. Next, we recall from \eqref{eq:Vpluspt}--\eqref{eq:VpluspN}, where we formally set $\Phi=0$, that
\begin{align*}
\big|\pmb{V}_{\! +}(p_t) \big| & \lesssim |p_N| ,  \qquad \qquad \Big| \pmb{V}_{\! +} \Big( \, \frac{2r^2 }{\Omega^2}p_u \Big) \Big|  \lesssim r^2|p_N|, \qquad \qquad  \pmb{V}_{\! +}(r)=r.
\end{align*}
We further note that 
$$  p_t \partial_{p_{r^*}} p_t =p_{r^*}, \qquad \qquad  p_t \partial_{p_{r^*}} p_{r^*} =p_t   , \qquad \qquad    \Big| p_t \partial_{p_{r^*}} \frac{p_u}{\Omega^2} \Big|  \leq \frac{|p_u|}{\Omega^2} \lesssim |p_N|   .$$
It allows to obtain the upper bounds for the derivatives of $\pmb{b}$. Using in addition $|p_N \Psi| \lesssim |p_t|$, we get the estimates related to $\overline{\pmb{a}}$.
\end{proof}

For the derivatives $\pmb{V}_{\! +} \Psi$ and $r^{-1}\pmb{\varphi}_- \partial_{p_{r^*}} \Psi$, the analysis will be divided in three regions, $\{ p_{r^*} \leq 0, \; r \geq R \}$, $\{ r \leq R \}$ and $\{ p_{r^*} \geq 0, r \geq R\}$, where we recall that $R \geq 832M$. This separation is allowed by the properties of the null geodesic flow in Schwarzschild stated in the next lemma. See for instance \cite[Chapter 13]{O83} for more information.

\begin{lemma}\label{Lemgeodtime}
Let $y \in \{ t^*=0 \}$. There exists $0 \leq t^*_{1}(y) \leq  t_2^*(y) \leq + \infty$ such that
\begin{enumerate}[label = (\alph*)]
\item $t_1^*(y) <+\infty$ and
\begin{itemize}
\item for all $s^* \in (0,t^*_1(y))$, we have $\Phi_{s^*}(0,y) \in \{ p_{r^*} \leq 0, \; r \geq R \}$,
\item for all $s^* > t^*_1(y)$, there holds $\Phi_{s^*}(0,y) \notin \{ p_{r^*} \leq 0, \; r \geq R \}$.
\end{itemize}
\item $t_2^*(y)$ is infinite if and only if the orbit $s^* \mapsto \Phi_{s^*}(0,y)$ is future trapped and
\begin{itemize}
\item for all $s^* \in (t^*_1(y),t_2^*(y))$, we have $\Phi_{s^*}(0,y) \in \{ r \leq R \}$.
\item If $t^*_2(y)<+\infty$, either the orbit $s^* \mapsto \Phi_{s^*}(0,y)$ crosses the future event horizon $\mathcal{H}^+$ at $t^*_2(y)$ or $\Phi_{s^*}(0,y) \in \{ p_{r^*} \geq 0, \; r \geq R \}$ for all $s^* \geq t^*_2(y)$.
\end{itemize}
\end{enumerate}
\end{lemma}

For our purposes, we also need to control the derivatives of $\Psi$ on the hypersurface $\pi^{-1}(\{t^* =0 \})$.

\begin{lemma}\label{LemderivPsiini}
The following estimates hold,
$$ \frac{|\pmb{\varphi}_-|}{|p_t|}| \partial_t \Psi | \Big\vert_{\pi^{-1} (\{t^* = 0 \})} \lesssim \frac{1}{r}, \qquad \! \big| \pmb{V}_{\! +} \Psi \big| \Big\vert_{\pi^{-1} (\{t^* = 0 \})} \lesssim \frac{|p_t|}{r|p_N|}, \qquad \! \T_{|\slashed{p}|} \Psi \Big\vert_{\pi^{-1} (\{t^* = 0 \})}= p_t\partial_{p_{r^*}}\Psi \Big\vert_{\pi^{-1} (\{t^* = 0 \})}=0.$$
\end{lemma}
\begin{proof}
Since the (trivial) data for $\Psi$ are prescribed on $\pi^{-1}(\{t^*=0 \})$, it is convenient to work in the coordinate system $\mathscr{C} \coloneqq (t^*,r,\theta , \varphi,p_r,p_\theta,p_\varphi)$ of $\mathcal{P}$ induced by the coordinate system $(t^*,r,\theta,\varphi)$ on Schwarzschild. We note that
$$t^* = t+2M\log (r-2M), \qquad \qquad p_r=\frac{p_{r^*}}{\Omega^2}-\frac{2M p_t}{r \Omega^2}= \frac{p_{r^*}-p_t}{\Omega^2}  +p_t.$$
In order to avoid any confusion, we will denote by $\partial^*_a$ the derivative according to the variable $a$ of the coordinate system $\mathscr{C}$. For any $A \in \{ \theta , \varphi \}$, we have
$$ \partial^*_{t^*}  = \partial_{t}^{\mathstrut}, \qquad    \partial^*_{p_r} =\frac{\Omega^2 p_t}{p_t-\frac{2M}{r}p_{r^*}} \partial_{p_{r^*}}^{\mathstrut}, \qquad \partial^*_A=\partial_A^{\mathstrut}, \qquad \partial^*_{p_A}=\partial_{p_A}^{\mathstrut} . $$
For the radial derivative, the computations are slightly more complicated and we will prove
\begin{equation}\label{eq:partialrtstarcoor}
 \partial^*_{r}=-\frac{2M}{r\Omega^2} \partial_t+ \frac{1}{\Omega^2} \partial_{r^*} -\frac{2Mp_u}{r^2\Omega^2}\partial_{p_{r^*}}-\frac{6M p_v}{r^2}\partial_{p_{r^*}}+\frac{2M p_{r^*}(p_u+3 \Omega^2 p_v)}{r^2(p_t-p_{r^*}+\Omega^2 p_{r^*})} \partial_{p_{r^*}}.
\end{equation}
For this, note first that
$$ \partial_r = \frac{2M}{r \Omega^2} \partial^*_{t^*}+\partial^*_r-\frac{2M(p_{r^*}-p_t)}{r^2\Omega^4}\partial^*_{p_r}-\frac{2M}{r\Omega^2} \partial_r (p_t)\partial^*_{p_r} .$$
Next, we have using the null-shell relation \eqref{eq:defConsQua} that
$$ \partial_r (p_t) = \frac{1}{2}\partial_r \bigg( \frac{\Omega^2}{r^2} \bigg) \frac{|\slashed{p}|^2}{p_t} = - \frac{(r-3M)|\slashed{p}|^2}{r^4 p_t} = - \frac{(r-3M)(p_t-p_{r^*})(p_t+p_{r^*})}{r^2 \Omega^2 p_t} . $$ 
It implies, as $p_v=p_t+p_{r^*}$ and $p_u=p_t-p_{r^*}$,
 $$ \partial_r = \frac{2M}{r \Omega^2} \partial^*_{t^*}+\partial^*_r+\frac{4M|p_u|^2}{r^2\Omega^4p_t}\partial^*_{p_r}+\frac{12M p_up_v}{r^2\Omega^2p_t} \partial^*_{p_r} ,$$
from which we get \eqref{eq:partialrtstarcoor}.
 
By definition $\Psi \vert_{\pi^{-1}(\{t^*=0 \})}=0$, so we have $\partial^*_a\Psi \vert_{\pi^{-1}(\{t^*=0 \})}=0$ for any variable $a \neq t^*$ of the coordinate system $\mathscr{C}$. Using the damped massless Vlasov equation verified by $\Psi$, we then have
$$ -\frac{p_t}{\Omega^2} \partial_t \Psi+\frac{2Mp_{r^*}}{r\Omega^2} \partial_t \Psi = \pmb{b}(r,p_t) \qquad \text{on $\,\pi^{-1}(\{t^*=0 \})$}  .$$
Since
$$ -\frac{p_t}{\Omega^2}+\frac{2Mp_{r^*}}{r\Omega^2} = |p_t|+\frac{4M|p_u|}{r\Omega^2} \sim |p_N|,$$
and $|\pmb{b}(r,p_t)| \lesssim r^{-2} |p_t|$, we have
$$ | \partial_t \Psi | \Big\vert_{\pi^{-1} (\{t^* \geq 0 \})} \lesssim  \frac{|p_t|}{r^2|p_N|}.$$
We then get the result by using $|\pmb{\varphi}_-| \lesssim r |p_N|$ and that, on $\pi^{-1}(\{t^*=0 \})$,
$$ \pmb{V}_{\! +} \Psi = \frac{|r+6M|^{\frac{1}{2}}}{r^{\frac{1}{2}}\Omega^2}(r-3M) \partial_t \Psi +\frac{r}{\Omega^2} \cdot \frac{2M}{r} \partial_t \Psi =  \frac{27M^2r^{\frac{1}{2}} }{r^{\frac{3}{2}}-|r+6M|^{\frac{1}{2}}(r-3M)}\partial_t \Psi-r \partial_t \Psi.$$
\end{proof}

\subsubsection{The case of the Killing vector fields}

As a starting point, we treat the case of the derivatives associated to the symmetries of Schwarzschild.

\begin{proposition}\label{ProPsiKilling}
We have $\T_{|\slashed{p}|} ( \Psi )=0$ and
$$ \sup_{  \pi^{-1} (\{t^* \geq 0 \})} \, \frac{|\pmb{\varphi}_-|}{|p_t|} \big|\partial_t\Psi \big| < +\infty  .$$
\end{proposition}
\begin{proof}
For $Z \in \{\partial_t , \T_{|\slashed{p}|} \}$, we have $[\T_g,Z]=0$. Hence, we get from \eqref{eq:ComZ} and Lemma \ref{Lemsourcetermesti} that
$$\T_g (Z \Psi)+ \overline{\pmb{a}}(r,p_{r^*},p_t) Z \Psi =0.$$
As $\T_{|\slashed{p}|} (\Psi ) =0$ initially according to Lemma \ref{LemderivPsiini}, we get $\T_{|\slashed{p}|} (\Psi)=0$. For $\partial_t \Psi$, the Lemmata \ref{Lemphiminus} and \ref{LemderivPsiini} provide
$$  \T_g \Big( \, \frac{\pmb{\varphi}_-}{p_t} \partial_t \Psi \, \Big)+ \big[\overline{\pmb{a}}+\pmb{a} \big](r,p_{r^*},p_t) \frac{\pmb{\varphi}_-}{p_t} \partial_t \Psi =0, \qquad \qquad  \sup_{\pi^{-1} (\{t^* \geq 0 \}) } \frac{|\pmb{\varphi}_-|}{|p_t|}| \partial_t \Psi | < + \infty .$$     
We conclude the proof by using Duhamel's formula through Lemma \ref{LemDuhamel}.
\end{proof}

\subsubsection{Radial derivatives along trajectories of incoming far-away particles}

We now control the derivatives of $\Psi$ along the flow $s^* \mapsto \Phi_{s^*} (0,y)$ until time $t^*_1(y)$.
\begin{proposition}\label{ProfarawayincoPsi}
Let $y \in \{ t^*=0 \}$. There exists an absolute constant $B_1>0$ such that
$$ \sup_{0 < t^* < t^*_1(y)} \big| \pmb{V}_{\! +}\Psi \big|\circ \Phi_{t^*}(0,y)+\big|r^{-1} \pmb{\varphi}_- \partial_{p_{r^*}} \Psi \big|\circ \Phi_{t^*}(0,y) \leq B_1 .$$
\end{proposition}

\begin{proof}
We first recall that in $\{ p_{r^*} \leq 0, \; r \geq R \}$, we have $\Omega^{-2} \sim 1$ and $|\pmb{\varphi}_-| \sim r|p_t|$. Hence, in this region,
$$\big| \pmb{V}_{\! +} \Psi \big| +\big|r^{-1} \pmb{\varphi}_- \partial_{p_{r^*}} \Psi \big| \sim  \big| S \Psi \big| +\big|p_t \partial_{p_{r^*}} \Psi \big|+\big|\partial_t \Psi \big|, $$
so, in view of the previous Proposition \ref{ProPsiKilling}, it is enough to control $S \Psi$ and $p_t \partial_{p_{r^*}} \Psi$. By following the analysis of Section \ref{Subsec72}, which allowed us to obtain \eqref{eq:forRk641}, one can prove
\begin{align} \nonumber
& \T_g(\omega | S \Psi|)+ \frac{1}{16}\T_g \big(\omega \big|p_t \partial_{p_{r^*}} \Psi \big| \big)  \leq -\omega \overline{\pmb{a}}(r,p_{r^*},p_t) \Big[\big|S \Psi \big| +  \frac{1}{16}\big|p_t \partial_{p_{r^*}} \Psi \big| \Big]+C\frac{|p_N|}{r},
 \end{align}
 where the weight $\omega$ is defined in \eqref{eq:defomegafaraway}. The differences in the analysis are the following:
\begin{enumerate}[label = (\alph*)]
\item According to Proposition \ref{ProPsiKilling}, we have $\T_{|\slashed{p}|}(\Psi)=0$ and $|p_v||\partial_t \Psi|  \lesssim |p_t|r^{-1}$ in the region considered. For the second estimate, we also use $|p_t| \lesssim r^{-1}|\pmb{\varphi}_-|$, which holds in $\{r \geq 4M, \; p_{r^*} \leq 0 \}$.
\item Here we have $\T_g(\Psi) \neq 0$, whereas we had $\T_g(f) =0$ in Section \ref{Subsec72}. The extra (bad) error terms are handled by applying Lemma \ref{Lemsourcetermesti}.
\item We simply bound above by $0$ certain good error terms related to $\T_g(\omega)$.
\end{enumerate}
Note now that $\overline{\pmb{a}}(r,p_{r^*},p_t)\! \gtrsim r^{-1}|p_N|$ since $p_{r^*}\! \leq 0$ in this region. As $\omega \sim 1$, it remains to use Lemma \ref{LemDuhamel}.
\end{proof} 
 
\subsubsection{Radial derivatives along trajectories located in the bounded region}
We now deal with the bounded region $\{ r \leq R \}$, that is, we focus on the time interval $(t_1^*(y),t_2^*(y))$. On this domain, there holds $2M \leq r \leq R$, so we can work with multiples of $\pmb{V}_{\! +}$ and $ r^{-1}\pmb{\varphi}_- \partial_{p_{r^*}}$. We recall the function $\pmb{a}$ introduced in Lemma \ref{Lemphiminus}.
\begin{lemma}\label{LemComPsi00}
There exists an absolute constant $C_0 >0$ such that
\begin{align*}
 \T_g \big( \big| r \pmb{\varphi}_- \partial_{p_{r^*}} \Psi \big|  \big) +\overline{\pmb{a}} (r,p_{r^*},p_t)   \big| r \pmb{\varphi}_- \partial_{p_{r^*}} \Psi \big|  & \leq   \big|\pmb{\varphi}_- \pmb{V}_{\! +} \Psi \big|+C_0r|p_N|, \\
 \T_g \big( \big| \pmb{V}_{\! +} \Psi \big| \big) + \big[\overline{\pmb{a}}+\pmb{a} \big](r,p_{r^*},p_t) \big| \pmb{V}_{\! +} \Psi \big| & \leq \frac{|p_t|}{r^{\frac{5}{2}}|r+6M|^{\frac{1}{2}}}\big| r \pmb{\varphi}_- \partial_{p_{r^*}}\Psi \big|+C_0\frac{|p_N|}{r}.
 \end{align*}
\end{lemma}

\begin{proof}
From \eqref{eq:ComZ} applied to $Z=r \pmb{\varphi}_- \partial_{p_{r^*}}$ and Lemma \ref{Lemsourcetermesti}, there exists $C_1>0$ such that
\begin{equation*}\label{eq:Psi1}
 \T_g \big( | r \pmb{\varphi}_- \partial_{p_{r^*}} \Psi | \big) + \overline{\pmb{a}} (r,p_{r^*},p_t) | r  \pmb{\varphi}_- \partial_{p_{r^*}} \Psi | \leq   \big|[\T_g,  r \pmb{\varphi}_- \partial_{p_{r^*}}  ](\Psi) \big|+C_1|\pmb{\varphi}_-|.
 \end{equation*}
Next, we apply the commutation formula of Proposition \ref{Propartialprstar} together with Lemma \ref{ProPsiKilling} to get
 $$  \big| \big[\T_g,  r \pmb{\varphi}_- \partial_{p_{r^*}}   \big]\Psi \big| \leq \frac{|\pmb{\varphi}_-|^2}{  |p_t| }\big| \partial_t \Psi \big| + \big|\pmb{\varphi}_- \pmb{V}_{\! +} \Psi \big| \leq  C_1|\pmb{\varphi}_-|  +  \big|\pmb{\varphi}_- \pmb{V}_{\! +} \Psi \big| .$$
It remains to use $|\pmb{\varphi}_-| \lesssim r|p_N|$. We now prove the second estimate. Let us recall the commutator $[\T_g,\pmb{V}_{\! +}]$ from Lemma \ref{LemComVminusregularised}. Applying \eqref{eq:ComZ} to $Z=\pmb{V}_{\! +}$ and Lemma \ref{Lemsourcetermesti}, we then get
$$ \T_g \big( \big| \pmb{V}_{\! +} \Psi \big| \big) + \big[ \overline{\pmb{a}}+\pmb{a}\big](r,p_{r^*},p_t) \big| \pmb{V}_{\! +} \Psi \big| \leq  \frac{r+3M}{r^{\frac{3}{2}}|r+6M|^{\frac{3}{2}}}|p_t|\big| \pmb{\varphi}_- \partial_{p_{r^*}} \Psi \big| +2\big| \T_g (\Psi) \big|+C_1\frac{|p_N|}{r }.$$
Finally, we use $|\T_g(\Psi)|\lesssim r^{-1}|p_t| $, which follows from Lemma \ref{LembasicPsi} and $|\overline{\pmb{a}} \Psi| \lesssim r^{-1}|p_N \Psi| \lesssim r^{-1} |p_t|$.
 \end{proof}

As in Section \ref{Subsec73}, we use the weight $\zeta^A$ with $A$ large enough, in order to absorb the bad error terms away from the future event horizon $\H$ and the photon sphere $\{ r=3M \}$.

\begin{corollary}\label{CorComBounded}
There exists an absolute constant $c_0>0$ such that for all $A \geq 0$, we have on $\{ r \leq R \}$,
\begin{align*}
 \T_g \big( \zeta^A| r \pmb{\varphi}_- \partial_{p_{r^*}} \Psi| \big) +\Big( c_0A |p_{r^*}|+c_0A|r-3M||p_t|+\overline{\pmb{a}}(r,p_{r^*},p_t)  \Big) \zeta^A|r  \pmb{\varphi}_- \partial_{p_{r^*}} \Psi| \leq   \zeta^A \big| \pmb{\varphi}_- \pmb{V}_{\! +} \Psi \big|+C_0\zeta^A r|p_N|
 \end{align*}
 and
 \begin{align*}
 \T_g \big(\zeta^A \big| \pmb{V}_{\! +} \Psi \big| \big) + \Big( c_0A |p_{r^*}|+c_0A|r-3M||p_t|+ \big[\overline{\pmb{a}}+\pmb{a} \big](r,p_{r^*},p_t)\Big) \zeta^A \big| \pmb{V}_{\! +} \Psi \big| \qquad \qquad \qquad \qquad \qquad \qquad &\\
  \leq \frac{|p_t|}{r^{\frac{5}{2}}|r+6M|^{\frac{1}{2}}}\zeta^A\big| r \pmb{\varphi}_- \partial_{p_{r^*}}\Psi \big|+C_0\zeta^A\frac{|p_N|}{r}. &
 \end{align*}
\end{corollary} 

\begin{proof}
Let us recall from Lemma \ref{LemPropz} that $\zeta \sim 1$ and $\T_g(\zeta)\lesssim -|p_{r^*}|-|r-3M||p_t|$ on $\{ r \leq R \}$. The result then follows from Lemma \ref{LemComPsi00}.
 \end{proof}

We are now able to conclude the analysis in the bounded region.

\begin{proposition}\label{ProboundedPsi}
Let $y \in \{ t^*=0 \}$. There exists an absolute constant $B_2>0$ such that 
$$ \sup_{0 \leq  t^* < t^*_2(y)} \big| \pmb{V}_{\! +}\Psi \big|\circ \Phi_{t^*}(0,y)+\big|r \pmb{\varphi}_- \partial_{p_{r^*}} \Psi \big|\circ \Phi_{t^*}(0,y) \leq B_2.$$
\end{proposition} 

 \begin{proof}
We assume that $t_2^*(y) >t_1^*(y)$, otherwise there is nothing to prove. We claim that for $A >0$ large enough, there exists $c>0$ such that we have on $\{ r \leq R \}$,
 \begin{equation}\label{equa:toprove}
 \T_g \bigg(  \frac{\zeta^A}{8M^2}| r \pmb{\varphi}_- \partial_{p_{r^*}} \Psi| \bigg)+\T_g \big(\zeta^A \big| \pmb{V}_{\! +} \Psi \big| \big)+c|p_N| \bigg( \frac{\zeta^A}{8M^2}| r \pmb{\varphi}_- \partial_{p_{r^*}} \Psi|+\zeta^A \big| \pmb{V}_{\! +} \Psi \big| \bigg) \lesssim |p_N|.
 \end{equation}
Then, the result will be a direct consequence of Proposition \ref{ProfarawayincoPsi} (providing the estimates up to time $t_1^*(y)$), $\zeta \sim 1$, and Lemma \ref{LemDuhamel}.

 We now prove \eqref{equa:toprove}. For this, we remark that according to Corollary \ref{CorComBounded} and since $\overline{\pmb{a}}(r,p_{r^*},p_t) \gtrsim |p_N|r^{-2}$, it suffices to show that if $A>0$ is large enough, then
\begin{align}
\label{equa:toprove3}
 \frac{1}{8M^2}\Big( c_0A |p_{r^*}|+c_0A|r-3M||p_t|+\overline{\pmb{a}}(r,p_{r^*},p_t) \Big) & > \frac{5}{4} \cdot  \frac{|p_t|}{r^{\frac{5}{2}}|r+6M|^{\frac{1}{2}}}  , \\
 \label{equa:toprove2}
 c_0A |p_{r^*}|+c_0A|r-3M||p_t|+ \big[\overline{\pmb{a}}+\pmb{a} \big](r,p_{r^*},p_t) & >  \frac{3}{2} \cdot  \frac{|\pmb{\varphi}_-|}{8M^2} .
 \end{align}
We introduce $\epsilon_0>0$, which will be fixed small enough and we recall
\begin{align*}
\pmb{a}(r,p_{r^*},p_t) & = \frac{(r^2+2Mr+3M^2)|p_t|}{r|r+6M|^{\frac{1}{2}}(r^{\frac{3}{2}}+M|r+6M|^{\frac{1}{2}}) }  +\frac{2M|p_u|}{r^2\Omega^2}  , \qquad \qquad \overline{\pmb{a}}(r,p_{r^*},p_t)=\pmb{a}(r,p_{r^*},p_t)-\frac{p_{r^*}}{r+6M}.
 \end{align*}
We deal first with \eqref{equa:toprove3} and we consider two cases:
\begin{itemize}
\item Close to the photon sphere, where $|r-3M| < \epsilon_0$. We have
\begin{align*}
 \pmb{a}(3M,p_{r^*},p_t) & = \frac{2|p_t|}{3(\sqrt{3}+1)M}+\frac{2|p_u|}{3M }, \\
  \frac{4}{3} \cdot \frac{|p_t|}{|3M|^{\frac{5}{2}}|9M|^{\frac{1}{2}}} & = \frac{4}{3} \cdot \frac{|p_t|}{27 \sqrt{3} M^3} \leq \frac{1}{ 8M^2} \cdot \frac{2|p_t|}{3(\sqrt{3}+1)M}.
  \end{align*}
  Consequently, if $A$ large enough so that $|p_{r^*}|(r+6M)^{-1} \leq c_0A|p_{r^*}|$, we get by continuity that \eqref{equa:toprove3} holds on $\{|r-3M|<\epsilon_0\}$, provided that $\epsilon_0$ is small enough.
  \item The rest of the bounded region, $r \leq R$ and $|r-3M| \geq \epsilon_0$. In that case, since $|r-3M|$ is bounded below, the inequality is verified as soon as $A>0$ is large enough.
\end{itemize}
For the estimate \eqref{equa:toprove2}, we consider again two cases:
\begin{itemize}
\item Close to the future event horizon $\mathcal{H}^+$ where $2M \leq r < 2M+\epsilon_0$. By Remark \ref{RkvarphiHplus}, we have
\begin{align*}
 \color{white} \square \qquad \color{black} \frac{|\pmb{\varphi}_-|}{8M^2}  \leq \frac{2r|p_u|}{8M^2 \Omega^2}+ \frac{1}{8M^2} \cdot\frac{27M^2r^{\frac{1}{2}}  |p_t|}{r^{\frac{3}{2}}-|r+6M|^{\frac{1}{2}}(r-3M)} \leq \Big|1+\frac{\epsilon_0}{2M}\Big|^3 \cdot \frac{2M|p_u|}{r^2 \Omega^2}+\frac{c_0}{2}A|r-3M||p_t|,  
 \end{align*}
 provided that $A>0$ is large enough. We get \eqref{equa:toprove2} if $\epsilon_0$ is small enough so that $\frac{3}{2}\big|1+\frac{\epsilon_0}{2M}\big|^3\leq 2$.     
\item The rest of the bounded region, $2M+\epsilon_0 \leq r \leq R$. In this domain, we have 
$$  \frac{3}{2} \cdot  \frac{|\pmb{\varphi}_-|}{8M^2}  \lesssim |p_{r^*}|+|r-3M||p_t|.$$
We then get \eqref{equa:toprove2} if $A>0$ chosen large enough compared to $\epsilon_0^{-1}$.
\end{itemize}
 \end{proof}

 \subsubsection{Radial derivatives along trajectories of escaping far-away particles}

We finally treat the domain $\{ p_{r^*} \geq 0, \, r \geq R \}$ or, equivalently, the time interval $(t_2^*(y),+\infty)$. 

\begin{lemma}\label{LemComPsi000}
There exists an absolute constant $C_0 >0$ such that on $\{ p_{r^*} \geq 0, \, r \geq R \}$,
\begin{align*}
 \T_g \big( \big| r^{-1} \pmb{\varphi}_- \partial_{p_{r^*}} \Psi \big|  \big) +\pmb{a} (r,p_{r^*},p_t)   \big| r^{-1} \pmb{\varphi}_- \partial_{p_{r^*}} \Psi \big|  & \leq  \frac{|\pmb{\varphi}_-|}{r^2} \big| \pmb{V}_{\! +} \Psi \big|+C_0\frac{|p_N|}{r}, \\
 \T_g \big( \big| \pmb{V}_{\! +} \Psi \big| \big) + \big[\overline{\pmb{a}} + \pmb{a} \big](r,p_{r^*},p_t) \big| \pmb{V}_{\! +} \Psi \big| & \leq \frac{|p_t|}{r^{\frac{1}{2}}|r+6M|^{\frac{1}{2}}}\big| r^{-1} \pmb{\varphi}_- \partial_{p_{r^*}}\Psi \big|+C_0\frac{|p_N|}{r}.
 \end{align*}
\end{lemma}
\begin{proof}
As $\T_g(r^{-2})=-2p_{r^*}r^{-3} $, we get the result using Lemma \ref{LemComPsi00} and that, for $p_{r^*} \geq 0$,
$$\overline{\pmb{a}} (r,p_{r^*},p_t)+2p_{r^*}r^{-1}=\pmb{a}(r,p_{r^*},p_t)-p_{r^*}(r+6M)^{-1}+2p_{r^*}r^{-1} \geq \pmb{a}(r,p_{r^*},p_t).$$
\end{proof}
This lemma allows us to deduce the next result, which in particular implies Proposition \ref{ProboundPsibis}.

\begin{proposition}\label{ProfinalPsideriv}
There exists a constant $B_3>0$ such that for all $y \in \pi^{-1}(\{t^* =0 \})$,
$$ \sup_{t^* \geq 0} \, \big| \pmb{V}_{\! +}\Psi \big|\circ \Phi_{t^*}(0,y)+\big|r^{-1} \pmb{\varphi}_- \partial_{p_{r^*}} \Psi \big|\circ \Phi_{t^*}(0,y) \leq B_3.$$
\end{proposition}

\begin{proof}
Let $y \in \pi^{-1}(\{t^* =0 \})$. We remark that in view of Proposition \ref{ProboundedPsi}, we only need to treat the time interval $t^* \geq t_2^*(y)$. We then assume that the orbit $\tau \mapsto \Phi_\tau(0,y)$ escapes to future null infinity $\mathcal{I}^+$, since there is nothing to prove otherwise. The key step consists in proving that 
\begin{equation}\label{equation:toprove}
\T_g \bigg( \frac{3}{2} | r^{-1} \pmb{\varphi}_- \partial_{p_{r^*}} \Psi|+| \pmb{V}_+ \Psi| \bigg) +  \frac{|p_t| }{4(r+6M)}\bigg( \frac{3}{2} | r^{-1} \pmb{\varphi}_- \partial_{p_{r^*}} \Psi|+| \pmb{V}_{\! +} \Psi| \bigg) \lesssim  C_0\frac{|p_t|}{r }    
\end{equation}
on $\{p_{r^*} \geq 0 , \, r \geq R \}$. Indeed, since 
$$ \big| \pmb{V}_{\! +}\Psi \big|\circ \Phi_{t^*_2(y)}(0,y)+\big|r^{-1} \pmb{\varphi}_- \partial_{p_{r^*}} \Psi \big|\circ \Phi_{t^*_2(y)}(0,y) \leq B_2$$
 by Proposition \ref{ProboundedPsi}, the result will follow by Lemma \ref{LemDuhamel}. 

We note now that \eqref{equation:toprove} is implied by Lemma \ref{LemComPsi000}, $p_N=p_t$ for $r \geq R$ and
$$ \pmb{a}(r,p_{r^*},p_t) \geq \frac{|p_t|}{r^{\frac{1}{2}}|r+6M|^{\frac{1}{2}}}, \qquad \qquad \big[\overline{\pmb{a}}+\pmb{a} \big](r,p_{r^*},p_t) \geq \frac{3}{2}\frac{|\pmb{\varphi}_-|}{r^2}+\frac{|p_t|}{4(r+6M)} .$$
The first inequality follows from $r^2+2Mr+3M^2 \geq r^2+Mr^{\frac{1}{2}}|r+6M|^{\frac{1}{2}}$. By Remark \ref{RkvarphiHplus}, we have for $r \geq 3M$,
$$\frac{|\pmb{\varphi}_-|}{r^2} \leq  \frac{|p_t|-p_{r^*} }{r \Omega^2}  +\frac{27M^2 |p_t|}{r^3+r^{\frac{3}{2}}|r+6M|^{\frac{1}{2}}(r-3M)} \leq \frac{|p_t|-p_{r^*}}{r \Omega^2}  +\frac{27M^2 |p_t|}{r^3}.$$
Moreover, we have
\begin{align*}
\overline{\pmb{a}}(r,p_{r^*},p_t) & = \pmb{a}(r,p_{r^*},p_t)-\frac{p_{r^*}}{r+6M} \geq \frac{|p_t|}{r^{\frac{1}{2}}|r+6M|^{\frac{1}{2}}}-\frac{p_{r^*}}{r+6M} \geq \frac{|p_t|-p_{r^*}}{r+6M},
\end{align*}
so that, for $p_{r^*} \geq 0$,
\begin{align*}
\overline{\pmb{a}}(r,p_{r^*},p_t) & \geq \frac{|\pmb{\varphi}_-|}{r^2}-\frac{8M(|p_t|-p_{r^*})}{(r-2M)(r+6M)}-\frac{27M^2 |p_t|}{r^3} \geq \frac{|\pmb{\varphi}_-|}{r^2}-\frac{8M|p_t|}{(r-2M)(r+6M)}-\frac{27M^2 |p_t|}{r^3} .
\end{align*}
Now, still for $p_{r^*} \geq 0$ and $r \geq 3M$,
\begin{align*}
\frac{1}{2}\pmb{a}(r,p_{r^*},p_t) & \geq  \frac{|p_t|}{2(r+6M)} \geq \frac{|\pmb{\varphi}_-|}{2r^2}-\frac{8M|p_t|}{2(r-2M)(r+6M)} -\frac{27M^2 |p_t|}{2r^3} .
\end{align*}
We conclude the proof by noting that for $r \geq R \geq 832M$,
\begin{align*}
 \pmb{a}(r,p_{r^*},p_t)-\frac{24M|p_t|}{(r-2M)(r+6M)} -\frac{81M^2 |p_t|}{r^3} & \geq \frac{|p_t|}{r+6M}-\frac{24M|p_t|}{(r-2M)(r+6M)} -\frac{81M^2 |p_t|}{r^3} \geq \frac{|p_t|}{2(r+6M)}.
\end{align*}
\end{proof}

\section{ Commutator associated to a conserved quantity arising from trapping}\label{SecCotan}

In order to consider symplectic gradients, we cannot merely work on the null-shell $\mathcal{P}$ and we need to consider a larger subset of the cotangent bundle, that is the one of the causal geodesics. The goal of this section consists in justifying our choice to work $V_+$, which is a particular projection of $\T_{\varphi_-}$ on $T \mathcal{P}$. For this purpose, we recall the notations introduced in Section \ref{Subseccotan}.
 
We start by defining, for all $m \geq 0$, the mass-shell
$$\mathcal{P}_m= \Big\{ (x,p)\in T^*\mathcal{S} \; \big| \; g_x^{-1}(p,p)=-m^2, \; \text{ $p$ is future-directed} \Big\},$$
so that $\mathcal{P}_0=\mathcal{P}$. In particular, if $s \mapsto \gamma (s)$ is a future-directed geodesic verifying $g(\dot{\gamma},\dot{\gamma})=-m^2$, we have for all $s$ that
$$\big(\gamma(s),g_{\gamma(s)}(\dot{\gamma}(s),\cdot) \big) \in \mathcal{P}_m.$$
The relation $g_x^{-1}(p,p)=-m^2$, often called the mass-shell relation, is equivalent to
\begin{equation}\label{eq:massshell}
 |p_t|^2=|p_{r^*}|^2+\Omega^2\Big(m^2+\frac{|\slashed{p}|^2}{r^2}\Big).
 \end{equation}
The mass-shell $\mathcal{P}_m$ can be parametrised by the coordinates $(t,r^*,\theta,\phi, p_{r^*},p_\theta,p_\phi)$ using \eqref{eq:massshell}. The geodesic spray, which is tangent to the mass-shell, then reads on $\mathcal{P}_m$ as
\begin{equation}\label{eq:Vlamassm}
\mathbb{X}_{g}=-\frac{p_t}{\Omega^2} \partial_t + \frac{p_{r^*}}{\Omega^2} \partial_{r^*}+\frac{p_{\theta}}{r^2} \partial_{\theta} + \frac{p_{\phi}}{r^2 \sin^2  \theta } \partial_{\phi} -\frac{m^2Mr^2-(r-3M)|\slashed{p}|^2}{r^4}  \partial_{p_{r^*} }+\frac{\cot \theta }{r^2 \sin^2 \theta}  p_{\phi}^2 \partial_{p_{\theta}} .
 \end{equation} 
 
\subsection{Circular orbits for causal geodesics}

Let $m \geq 0$. Since the circular orbits for causal geodesics lie in $\{ p_{r^*} =0 \}$, the possible orbits associated to a given angular momentum $|\slashed{p}|$ are the roots of 
$$ m^2r^2-\frac{r-3M}{M}|\slashed{p}|^2= m^2 r^2-\frac{r}{M}|\slashed{p}|^2+3|\slashed{p}|^2 .$$ 
As a result, we find circular orbits if and only if $|\slashed{p}|\geq 2\sqrt{3}m M$. They are then located on the spheres $\{r=r_{\pm}^{m}(|\slashed{p}|)\}$, where 
\begin{equation}\label{eq:defrpm}
r_{\pm}^{m}(|\slashed{p}|)=\frac{|\slashed{p}|^2}{2Mm^2}\bigg(1\pm \sqrt{1-\frac{12M^2m^2}{|\slashed{p}|^2}} \, \bigg).
\end{equation}
We remark in particular that
\begin{equation}\label{eq:equirminus}
\hspace{-5mm} r_{-}^{m}(|\slashed{p}|) = 3M+\frac{9M^3m^2}{|\slashed{p}|^2} +O_{m \to 0} \big( m^4 \big), \qquad \qquad \qquad r_+^m (|\slashed{p}|) \isEquivTo{m \to 0} \frac{|\slashed{p}|^2}{Mm^2}.
\end{equation}

\subsection{The weight functions $\varphi_{\pm}^m$}

We start by extending the functions $\varphi_\pm$, which are defined on the null-shell (see Definition \ref{Defvarphimin} and \eqref{eq:defvarphiplus}). The choice of the extensions of $\varphi_\pm$ is justified in \cite[Subsection 5.1.4]{V24}.

\begin{definition}
Let $m > 0$. We define the following quantities on $\mathcal{P}_m$:
\begin{itemize}
\item If $|\slashed{p}| \geq 2\sqrt{3}mM$, we set
$$a^m(|\slashed{p}|) \coloneqq \frac{2m^2M^2|r^{m}_-(|\slashed{p}|)|^3}{|\slashed{p}|^2(4M-r^{m}_-(|\slashed{p}|))(r^{m}_-(|\slashed{p}|)-3M)}.$$
\item For $|\slashed{p}| >4mM$, let $\rho^m_{\pm}(|\slashed{p}|)$ be the roots of $m^2 r^2-\frac{|\slashed{p}|^2}{2M}r+|\slashed{p}|^2$. They are given by
$$ \rho_{\pm}^{m}(|\slashed{p}|)=\frac{|\slashed{p}|^2}{4Mm^2}\bigg(1\pm \sqrt{1-\frac{16M^2m^2}{|\slashed{p}|^2}} \, \bigg). $$
\item We consider, for $|\slashed{p}| >4mM$ and $ \rho_{-}^{m}(|\slashed{p}|) < r  <  \rho_{+}^{m}(|\slashed{p}|)$, the functions
$$ \varphi_{\pm}^m(x,p) \coloneqq \frac{ \sqrt{2M}|\slashed{p}|r^{\frac{3}{2}}}{(-m^2r^2+\frac{|\slashed{p}|^2}{2M}r-|\slashed{p}|^2)^{\frac{1}{2}}}\bigg(p_{r^*}\pm\Big(1+\frac{a^{m}(|\slashed{p}|)}{r}\Big)^{\frac{1}{2}}\Big(1-\frac{r^{m}_-(|\slashed{p}|)}{r}\Big)\big||p_t|^2-m^2\big|^{\frac{1}{2}} \bigg).$$
\end{itemize} 
\end{definition}

One can check, using \eqref{eq:equirminus}, that
\begin{equation}\label{eq:mto0}
a^m(|\slashed{p}|) \isEquivTo{m \to 0} 6M, \qquad \quad \rho^m_-(|\slashed{p}|) \isEquivTo{m \to 0} 2M, \qquad \quad \rho^m_+(|\slashed{p}|) \isEquivTo{m \to 0} \frac{r_+^m}{2} , \qquad \quad \varphi_\pm^m(x,p) \isEquivTo{m \to 0} \varphi_\pm.
\end{equation}
We can then set $\rho^0_+(|\slashed{p}|)=+\infty$ and extend by continuity $a^m(|\slashed{p}|)$, $\rho_-^m(|\slashed{p}|)$ and $\varphi_\pm^m$ at $m=0$.

\begin{proposition}\label{Provarphicotang}
Let $m > 0$ and $|\slashed{p}| > 4Mm$. We have on $\mathcal{P}_m$ and for $\rho_{-}^{m}(|\slashed{p}|) < r  <  \rho_{+}^{m}(|\slashed{p}|)$,
\begin{equation*}
\T_g \big(\varphi_\pm^m \big)=\pm \frac{m^2(|p_t|^2-m^2)^{\frac{1}{2}}r^{\frac{1}{2}}(r-r^{m}_+(|\slashed{p}|))}{2(r+a^{m}(|\slashed{p}|))^{\frac{1}{2}}\big(m^2 r^2-\frac{|\slashed{p}|^2}{2M}r+|\slashed{p}|^2\big)}\varphi_{\pm}^m.
\end{equation*}
\end{proposition}
\begin{remark}\label{Rklimitmto0}
In the limit $m \to 0$, we recover the relation satisfied by $\varphi_\pm = \varphi^0_\pm$ on the null-shell (see Lemmata \ref{Lemphiminus} and \ref{lemmavarphiplus}). Indeed, by \eqref{eq:mto0}, we have
$$ \frac{m^2(|p_t|^2-m^2)^{\frac{1}{2}}r^{\frac{1}{2}}(r-r^{m}_+(|\slashed{p}|))}{2(r+a^{m}(|\slashed{p}|))^{\frac{1}{2}}\big(m^2 r^2-\frac{|\slashed{p}|^2}{2M}r+|\slashed{p}|^2\big)} \isEquivTo{m \to 0} \frac{|p_t|r^{\frac{1}{2}}}{|r+6M|^{\frac{1}{2}}(r-2M)}=\frac{|p_t|}{r^{\frac{1}{2}}|r+6M)|^{\frac{1}{2}} \Omega^2}.$$
\end{remark}
\begin{proof}
In order to lighten the notations, we drop the dependence in $|\slashed{p}|$ of $r_\pm^m$ and $a^m$. We first compute, using $\T_g(|\slashed{p}|)=0$,
\begin{align*}
\T_g \Big(\Big| -m^2 r^2+\frac{|\slashed{p}|^2}{2M}r-|\slashed{p}|^2 \Big|^{-\frac{1}{2}}\Big) &=\frac{p_{r^*}\big(2m^2 r -\frac{|\slashed{p}|^2}{2M}\big)}{2\big|-m^2 r^2+\frac{|\slashed{p}|^2}{2M}r-|\slashed{p}|^2\big|^{\frac{3}{2}}} , \\
\T_g \big(  r^{\frac{3}{2}}p_{r^*} \big) & = \frac{3}{2}r^{\frac{1}{2}}|p_{r^*}|^2-\frac{M}{r^{\frac{5}{2}}}\Big(m^2 r^2-\frac{|\slashed{p}|^2}{M}r+3|\slashed{p}|^2\Big) , \\
\T_g \big( (r+a^{m})^{\frac{1}{2}}(r-r_-^m) \big) & = p_{r^*}\frac{3r-r_-^m+2a^m}{2(r+a^{m})^{\frac{1}{2}}}.
\end{align*}
Consequently, we have 
\begin{align}
&\T_g  \bigg(\frac{r^{\frac{3}{2}}p_{r^*}}{\big|-m^2 r^2+\frac{|\slashed{p}|^2}{2M}r-|\slashed{p}|^2\big|^{\frac{1}{2}}}\bigg) \nonumber \\
&\, =\frac{-|p_{r^*}|^2r^{\frac{1}{2}}}{2\big| -m^2 r^2+\frac{|\slashed{p}|^2}{2M}r-|\slashed{p}|^2 \big|^{\frac{3}{2}}}\Big(m^2 r^2-\frac{|\slashed{p}|^2}{M}r+3|\slashed{p}|^2\Big)-\frac{1}{\big| -m^2 r^2+\frac{|\slashed{p}|^2}{2M}r-|\slashed{p}|^2\big|^{\frac{1}{2}}}\frac{M}{r^{\frac{5}{2}}}\Big(m^2 r^2-\frac{|\slashed{p}|^2}{M}r+3|\slashed{p}|^2\Big) \nonumber .
\end{align}
We then deduce, using the mass-shell relation \eqref{eq:massshell}, that
\begin{align*}
\T_g \bigg(\frac{r^{\frac{3}{2}}p_{r^*}}{\big|m^2 r^2-\frac{|\slashed{p}|^2}{2M}r+|\slashed{p}|^2\big|^{\frac{1}{2}}}\bigg) &=-\frac{m^2 r^2-\frac{|\slashed{p}|^2}{M}r+3|\slashed{p}|^2}{\big| -m^2 r^2+\frac{|\slashed{p}|^2}{2M}r-|\slashed{p}|^2\big|^{\frac{3}{2}}}\cdot \frac{r^{\frac{1}{2}}}{2} \big( |p_t|^2-m^2 \big) \\ 
& =    \frac{m^2 \big| |p_t|^2-m^2 \big|^{\frac{1}{2}}r^{\frac{1}{2}}(r-r^{m}_+)}{2(r+a^{m})^{\frac{1}{2}}\big(m^2 r^2-\frac{|\slashed{p}|^2}{2M}r+|\slashed{p}|^2\big)} \cdot \frac{(r+a^m )^{\frac{1}{2}}(r-r_-^m)}{\big| m^2 r^2-\frac{|\slashed{p}|^2}{2M}r+|\slashed{p}|^2\big|^{\frac{1}{2}}}\big| |p_t|^2-m^2 \big|^{\frac{1}{2}} . 
\end{align*}
We also obtain
\begin{align*}
&\T_g \bigg(\frac{(r+a^{m})^{\frac{1}{2}}(r-r_-^m)}{\big| -m^2 r^2+\frac{|\slashed{p}|^2}{2M}r-|\slashed{p}|^2 \big|^{\frac{1}{2}}}\bigg)=\frac{p_{r^*} }{2(a^{m}+r)^{\frac{1}{2}}\big|-m^2 r^2+\frac{|\slashed{p}|^2}{2M}r-|\slashed{p}|^2\big|^{\frac{3}{2}}} \\
& \qquad \qquad \qquad \qquad \qquad \quad \cdot  \bigg(\Big(2m^2 r-\frac{|\slashed{p}|^2}{2M}\Big)(r+a^{m})(r-r^{m}_-)+(3r+2a^{m}-r^{m}_-)\Big(-m^2 r^2+\frac{|\slashed{p}|^2}{2M}r-|\slashed{p}|^2\Big)\bigg).
\end{align*}
Now, we claim that
$$\Big(2m^2 r-\frac{|\slashed{p}|^2}{2M}\Big)(r+a^{m})(r-r^{m}_-)-(3r+2a^{m}-r^{m}_-)\Big(m^2 r^2-\frac{|\slashed{p}|^2}{2M}r+|\slashed{p}|^2\Big)=-m^2r^3+m^2r^{m}_+r^2.$$ 
For this, note first that the polynomial on the LHS is equal to $-m^2r^3+b_2r^2+b_1r+b_0$, where 
\begin{align*}
b_0 & \coloneqq \frac{|\slashed{p}|^2}{2M}a^mr_-^m-2|\slashed{p}|^2a^m+r_-^m |\slashed{p}|^2, \qquad b_1 \coloneqq -2m^2a^mr_-^m+\frac{|\slashed{p}|^2}{2M}a^m -3 |\slashed{p}|^2, \qquad b_2 \coloneqq -m^2r_-^m+\frac{|\slashed{p}|^2}{M}.
\end{align*} 
By \eqref{eq:defrpm}, we have $b_2=m^2r_+^m$. Then, recalling first the definition of $a^{m}(|\slashed{p}|)$ and using then the relation $m^2|r_-^m|^2-\frac{|\slashed{p}|^2}{M}r_-^m+3|\slashed{p}|^2=0$, we get
\begin{align*}
\frac{2m^2M^2|r_-^m|^2}{|\slashed{p}|^2 a^m} b_0 & = M m^2|r_-^m|^3-4M^2m^2|r_-^m|^2+|\slashed{p}|^2(4M-r_-^m)(r_-^m-3M) =0.
\end{align*} 
Similarly, one can check that
\begin{align*}
 \frac{2m^2M^2|r_-^m|^3}{a^m} b_1 & =-4M^2m^4|r_-^m|^4+|\slashed{p}|^2M m^2|r_-^m|^3+3|\slashed{p}|^4|r_-^m|^2-21 |\slashed{p}|^4 M r_-^m+36|\slashed{p}|^4 M^2 \\
 & = -4M^2 \big(|\slashed{p}|^2 M^{-1}r_-^m-3|\slashed{p}|^2 \big)^2+|\slashed{p}|^2M m^2|r_-^m|^3+3|\slashed{p}|^4|r_-^m|^2-21 |\slashed{p}|^4 M r_-^m+36|\slashed{p}|^4 M^2 \\
 & = |\slashed{p}|^2M m^2|r_-^m|^3-|\slashed{p}|^4|r_-^m|^2+3 |\slashed{p}|^4 M r_-^m =0.
 \end{align*}
We then deduce that
\begin{equation*}
\T_g \bigg(\frac{(r+a^{m})^{\frac{1}{2}}(r-r_-^m)}{\big| -m^2 r^2+\frac{|\slashed{p}|^2}{2M}r-|\slashed{p}|^2 \big|^{\frac{1}{2}}}\bigg)=\frac{m^2 r^{\frac{1}{2}}(r-r^{m}_+)}{2(a^m+r)^{\frac{1}{2}}\big(m^2 r^2-\frac{|\slashed{p}|^2}{2M}r+|\slashed{p}|^2 \big)} \cdot \frac{r^{\frac{3}{2}}p_{r^*}}{\big|m^2r^2-\frac{|\slashed{p}|^2}{2M}r+|\slashed{p}|^2\big|^{\frac{1}{2}}},
\end{equation*}
which allows us to conclude the proof since $\T_g (|p_t|^2-m^2)=\T_g(|\slashed{p}|)=0$.
\end{proof}

In what follows, we denote by $m$ the function $m(x,p)\coloneqq |-g_x^{-1}(p,p)|^{\frac{1}{2}}$, which is well-defined on the subset $\{g_x^{-1}(p,p) \leq 0 \}$ of the cotangent bundle. Recall that, for $(x,p) \in \mathcal{P}$, we denote by $\tau \mapsto \Phi_{\tau}(x,p)$ the flow map of $\T_g$ parametrised by $t^*$ with data $\Phi_{t^*(x)}(x,p)=(x,p)$. Then, according to the previous Proposition \ref{Provarphicotang} and Remark \ref{Rklimitmto0}, the following statement holds.
\begin{corollary}\label{Corvarphicotang}
Let $\mathfrak{s}$ be the function defined as
$$ \mathfrak{s}(x,p)\coloneqq \varphi_-^m (x,p) e^{\alpha (x,p)}, \qquad \alpha(x,p)\coloneqq  \int_{\tau=0}^{t^*(x)} \frac{m^2(|p_t|^2-m^2)^{\frac{1}{2}}r^{\frac{1}{2}}(r-r^{m}_+(|\slashed{p}|))}{2(r+a^{m}(|\slashed{p}|))^{\frac{1}{2}}\big(m^2 r^2-\frac{|\slashed{p}|^2}{2M}r+|\slashed{p}|^2\big)} \circ \Phi_{\tau}(x,p) \dr \tau .  $$
Then, $\mathfrak{s}$ can be extended on $\mathcal{P}$, that is at $m=0$. Moreover, it is conserved along the causal geodesic flow.
\end{corollary}

\subsection{Symplectic gradients}\label{Subsecsympgrad}

We now compute the symplectic gradient of $\varphi_-^m$. For this, remark that we can write
$$ \varphi_-^m = \phi_- \bigg( m^2,\frac{m^2}{|\slashed{p}|^2}, r,p_{r^*},p_t \bigg), \qquad \qquad \phi_- \in C^1 \big( \mathring{O},\R \big) ,$$
where the open set $\mathring{O}$ is the interior of $O$, given by
$$  O \coloneqq \Big\{(x_1,\dots,x_5) \in \R_+\times \Big[0,\frac{16}{M^2} \Big)\times (2M,+\infty) \times \R \times \R_-^* \, \big| \, \rho_- \big(\sqrt{x_1/x_2} \, \big)<x_3<  \rho_+ \big(\sqrt{x_1/x_2}  \,\big) \Big\}  .$$
Moreover, in view of \eqref{eq:equirminus}, $\phi_-$ and $\nabla \phi_-$ can be extended continuously on $O$, so that
$$ \varphi_-=\varphi_-^0= \phi_-(0,0,\cdot), \qquad \! \partial_{r^*} \varphi_-= \Omega^2 \partial_{x^3}\phi_-(0,0, \cdot), \qquad \! \partial_{p_{r^*}} \varphi_-=  \partial_{x^4}\phi_-(0,0, \cdot), \qquad \! \partial_{p_t} \varphi_-=  \partial_{x^5}\phi_-(0,0, \cdot ). $$ 
Since $m^2=2H$, we have
$2 \T_{m^2}=\T_g$. By the chain rule, we then obtain
$$ \T_{\varphi_-^m} = \frac{1}{2}\partial_{x^1} \phi_- \T_g +\frac{1}{2|\slashed{p}|^2} \partial_{x^2} \phi_-  \T_g-2\frac{m^2}{|\slashed{p}|^3}\partial_{x^2} \phi_- \T_{|\slashed{p}|}+ \mathbb{V}_+, \qquad \mathbb{V}_+ \coloneqq \partial_{p_t}\big( \varphi_-^m \big) \partial_t+ \partial_{p_{r^*}}\big( \varphi_-^m \big) \partial_{r^*}-\partial_{r^*}\big( \varphi_-^m \big) \partial_{p_{r^*}} .$$

\begin{remark}\label{Rkabuse}
We abusively wrote $\partial_{r^*}\big( \varphi_-^m \big)$, $\partial_{p_{r^*}}\big( \varphi_-^m \big)$ and $\partial_{p_t}\big( \varphi_-^m \big)$ instead of $\Omega^2 \partial_{x^3}\phi_-$, $\partial_{x^4} \phi_-$ and $\partial_{x^5} \phi_-$.
\end{remark} For convenience, we introduce
$$ \pmb{x}_g(r,p_{r^*},p_t) := \frac{1}{2}\partial_{x^1} \phi_-(0,0,r,p_{r^*},p_t) +\frac{1}{2|\slashed{p}|^2} \partial_{x^2} \phi_- (0,0,r,p_{r^*},p_t).$$
Then, we have $\mathbb{V}_+=\T_{\varphi_-^m}-\pmb{x}_g \T_g$ on the null-shell $\mathcal{P}$ and $\mathbb{V}_+$ is indeed given by \eqref{eq:expTphi}. Moreover, as $\mathfrak{s}$ is a conserved quantity, we have
$$ \big[ \T_g, \T_{\mathfrak{s}} \big]=0, \qquad \qquad \T_{\mathfrak{s}} = e^{\alpha (x,p)} \T_{\varphi_-^m}+\varphi_-^m e^{\alpha (x,p)} \T_{\alpha}.$$
We summarise these properties in the next statement, which also uses Remark \ref{Rklimitmto0}.
\begin{proposition}\label{ProB2}
The vector fields $\T_g$, $\T_{|\slashed{p}|}$, and $\T_{\mathfrak{s}}$ commute with the geodesic spray $\T_g$. We have
$$ \forall \, (x,p) \in \mathcal{P}, \qquad \alpha \vert_{\mathcal{P}} (x,p) =  \int_{s=0}^{t^*(x)} \frac{|p_t|}{r^{\frac{1}{2}}|r+6M|^{\frac{1}{2}}\Omega^2} \circ \Phi_{s}(x,p) \dr s.$$
On the null-shell $\mathcal{P}$, we have
$$ \mathbb{V}_+ =  e^{-\alpha \vert_{\mathcal{P}}(x,p)} \T_{\mathfrak{s}}-\pmb{x}_g \T_g-\varphi_- \T_{\alpha},$$
where
$$  \big[\T_g, e^{-\alpha \vert_{\mathcal{P}}(x,p)} \T_{\mathfrak{s}} \big] = -\frac{|p_t|}{r^{\frac{1}{2}}|r+6M|^{\frac{1}{2}}\Omega^2} e^{-\alpha \vert_{\mathcal{P}}(x,p)}\T_{\mathfrak{s}}, \qquad \qquad \big[\T_g, \pmb{x}_g \T_g \big] = \T_g \big( \pmb{x}_g \big) \T_g.$$
\end{proposition}

We then observe that the vector field $\mathbb{V}_+$ is the sum of 
\begin{itemize}
\item $e^{-\alpha \vert_{\mathcal{P}}(x,p)} \T_{\mathfrak{s}}$, which enjoys a good commutation relation with $\T_g$, 
\item a term collinear to $\T_g$, which is then irrelevant once applied to a solution to the Vlasov equation,
\item a term proportional to $\varphi_-$.
\end{itemize} 
The issue is that $\mathbb{V}_+$ is not tangent to $\mathcal{P}$ so we need to project it on $T \mathcal{P}$ to use it for the study of massless Vlasov fields. However, since the normal to $\mathcal{P}$ is also tangent to it, there is no canonical choice of projection. The next result justifies that any choice should allow to close the estimates.

\begin{proposition}
Let $W$ be a vector field transverse to $T \mathcal{P}$ and
$$ \mathrm{Proj}_{\parallel W} :  \bigcup_{(x,p) \in \mathcal{P}} T_{(x,p)}T^* \mathcal{S} \to T \mathcal{P}$$
be the projection parallel to $W$. Then, for any vector fields $W_1$ and $W_2$ transverse to $T\mathcal{P}$, we have
$$ \mathrm{Proj}_{\parallel W_1} ( \mathbb{V}_+)-\mathrm{Proj}_{\parallel W_2} ( \mathbb{V}_+)=\varphi_-  \big( \, \mathrm{Proj}_{\parallel W_2} ( \T_\alpha)-\mathrm{Proj}_{\parallel W_1} ( \T_\alpha) \, \big) .$$
\end{proposition}
\begin{proof}
Since $\mathfrak{s}$ is conserved along causal geodesics, the vector field $\T_{\mathfrak{s}}$ is tangent to $\mathcal{P}$. As $\T_g\vert_{\mathcal{P}} \in T\mathcal{P}$, we get that $\mathbb{V}_++\varphi_- \T_{\alpha}$ is tangent to the null-shell as well. Thus,
$$ \mathrm{Proj}_{\parallel W} ( \mathbb{V}_+)+\varphi_-  \mathrm{Proj}_{\parallel W} ( \T_\alpha)= \mathbb{V}_++\varphi_- \T_\alpha$$
for any vector field $W$ transverse to $T \mathcal{P}$.
\end{proof}

In this article, we choose to work with $V_+ \! \coloneqq \mathrm{Proj}_{\parallel \partial_{p_t}} ( \mathbb{V}_+)$. As suggested by Proposition \ref{ProprojrescVplus}, for certain well-chosen $W_2$, the difference $V_+-\mathrm{Proj}_{\parallel W_2} ( \mathbb{V}_+)$ is collinear to $\varphi_- p_t \partial_{p_{r^*}}$. Let us now briefly justify that working with $\pmb{V}_{\! +}\!=\Omega^{-1} V_+$ is a relevant choice. The next result can be obtained similarly as Proposition \ref{ProB2}.

\begin{proposition}
Let $\pmb{\varphi}_-^m\coloneqq \Omega^{-1} \varphi_-^m$. Then, the following properties hold.
\begin{itemize}
\item The quantity $\overline{\mathfrak{s}} (x,p) \coloneqq \pmb{\varphi}_-^m (x,p) e^{\overline{\alpha} (x,p)}$, where
$$  \overline{\alpha}(x,p)\coloneqq  \int_{\tau=0}^{t^*(x)} \bigg(\frac{m^2(|p_t|^2-m^2)^{\frac{1}{2}}r^{\frac{1}{2}}(r-r^{m}_+(|\slashed{p}|))}{2(r+a^{m}(|\slashed{p}|))^{\frac{1}{2}}\big(m^2 r^2-\frac{|\slashed{p}|^2}{2M}r+|\slashed{p}|^2\big)\Omega}+  \frac{M}{r^2 \Omega^2}p_{r^*}  \bigg) \circ \Phi_{\tau}(x,p) \dr \tau, $$
can be extended on $\mathcal{P}$, that is at $m=0$. Furthermore, it is conserved along the geodesic flow. 
\item  The function $\pmb{\varphi}_-$ defined on the null-shell in \eqref{eq:defovphi} and $\pmb{\varphi}_-^0$ corresponds. Moreover,
$$  \forall \, (x,p) \in \mathcal{P}, \qquad \overline{\alpha} (x,p)=  \int_{s=0}^{t^*(x)} \pmb{a} \circ \Phi_{s}(x,p) \dr s,$$
where $\pmb{a}$ is the function introduced in Lemma \ref{Lemphiminus}.
\item Since $\overline{\mathfrak{s}}$ is a conserved quantity, its symplectic gradient $\T_{\overline{\mathfrak{s}}}=e^{\overline{\alpha} (x,p)}\T_{\pmb{\varphi}_-^m}+ e^{\overline{\alpha} (x,p)}\pmb{\varphi}_-^m\T_{\overline{\alpha}}$ verifies $[\T_g, \T_{\overline{\mathfrak{s}}}]=0$. Furthermore, we have on $\mathcal{P}$, in the coordinate system $(t,r^*,\theta , \phi, p_t,p_{r^*},p_\theta , p_\phi)$,
$$  \big[\T_g, e^{-\overline{\alpha}\vert_{\mathcal{P}}(x,p)} \T_{\overline{\mathfrak{s}}} \big] = -\pmb{a}(x,p) e^{-\overline{\alpha}\vert_{\mathcal{P}}(x,p)}\T_{\overline{\mathfrak{s}}}$$
and, using similar abuse of notations than in Remark \ref{Rkabuse}, 
$$ \T_{\pmb{\varphi}_-^m}  = \partial_{p_t} \big( \pmb{\varphi}_- \big)\partial_t+\partial_{p_{r^*}} \big( \pmb{\varphi}_- \big) \partial_{r^*}-\partial_{r^*}\big( \pmb{\varphi}_- \big) \partial_{p_{r^*}}+\frac{\pmb{x}_g}{\Omega} \T_g.$$
\end{itemize}
\end{proposition}

In order to compare $\pmb{V}_{\!+}=\Omega^{-1}V_+$ with a projection of $\T_{\pmb{\varphi}_-^m}-\pmb{x}_g \Omega^{-1} \T_g$ on $T \mathcal{P}$, we will work in the coordinate system $(\bar{t},r,\theta , \phi , \overline{p}_t, \overline{p}_r , p_\theta , p_\phi)$ of $T^*\mathcal{S}$ induced by the hyperboloidal coordinates $(\bar{t},r,\theta , \varphi)$, which were introduced in Definition \ref{Defhypcoor}. In particular, we have
$$\bar{t} = t-\int_{3M}^r \frac{\xi(s)}{\Omega^2(s)} \dr s, \qquad\quad \xi(r)=\Big( 1-\frac{3M}{r} \, \Big) \Big(1+\frac{6M}{r} \Big)^{\frac{1}{2}}, \qquad \quad \overline{p}_t=p_t,  \qquad \quad \overline{p}_r=\frac{p_{r^*}}{\Omega^2}+\frac{\xi(r)}{ \Omega^2}p_t.$$
In order to avoid any confusion, we denote by $\overline{\partial}_a$ the derivative with respect to $a$ in this coordinate system. Then, as $ \partial_{|\slashed{p}|} \pmb{\varphi}_-^m\vert_{m=0}=0$ and since $\overline{\partial}_{\overline{p}_t}  \pmb{\varphi}_-^m$, $ \overline{\partial}_{\overline{p}_r}  \pmb{\varphi}_-^m$ and $ \overline{\partial}_r \pmb{\varphi}_-^m$ are continous at $m=0$, we have on the null-shell $\mathcal{P}$, with similar abuse of notations than in Remark \ref{Rkabuse},
$$ \pmb{\varphi}_-=r \overline{p}_r, \qquad \qquad \T_{\pmb{\varphi}_-^m}-\frac{\pmb{x}_g}{\Omega}  \T_g= \overline{\partial}_{\overline{p}_t} \big( \pmb{\varphi}_- \big)\overline{\partial}_{\overline{t}}+\overline{\partial}_{\overline{p}_r} \big( \pmb{\varphi}_- \big) \overline{\partial}_{r}-\overline{\partial}_{r}\big( \pmb{\varphi}_- \big) \overline{\partial}_{\overline{p}_r}=r\overline{\partial}_r-\overline{p}_r \overline{\partial}_{\overline{p}_r}.$$
We finally prove the next result.
\begin{proposition}\label{ProprojrescVplus}
We have
\begin{align*}
 \pmb{V}_{\! +} & =  \mathrm{Proj}_{\parallel \partial_{\overline{p}_t}} \big( \T_{\pmb{\varphi}_-^m} \big)-\frac{\pmb{x}_g}{\Omega}  \T_g - \frac{M }{r \Omega^2} \overline{p}_r \overline{\partial}_{\overline{p}_r}-\frac{r^{\frac{1}{2}}\xi(r)}{|r+6M|^{\frac{1}{2}}\Omega^2} \overline{p}_r \overline{\partial}_{\overline{p}_r} .
 \end{align*}
The sum of the last two terms is a vector field regular up to $\H$ since $\xi(2M)=-1$.
\end{proposition}
\begin{proof}
We parametrise the null-shell $\mathcal{P}$ by the coordinate system $(\overline{t},r,\theta,\phi , \overline{p}_r , p_\theta , p_\phi)$ and we abusively denote by $\overline{\partial}_a$ the derivative with respect to the variable $a$. Then, as $\T_g$ is tangent to $\mathcal{P}$,
$$\mathrm{Proj}_{\parallel \partial_{\overline{p}_t}} \big( \T_{\pmb{\varphi}_-^m} \big) = r \overline{\partial}_r-\overline{p}_r \overline{\partial}_{\overline{p}_r}+\frac{\pmb{x}_g}{\Omega}  \T_g.$$
We now relate $\overline{\partial}_a$ with the derivatives of the coordinate system $(t,r^*,\theta , \phi,p_{r^*},p_\theta , p_\phi)$.  For any $A \in \{ \theta , \varphi \}$, we have
$$  \partial_{t}= \overline{\partial}_{\bar{t}}  , \qquad \quad  \partial_{p_{r^*}}=  \frac{p_t+\xi(r)p_{r^*}}{\Omega^2 p_t} \overline{\partial}_{\overline{p}_r} , \qquad \quad \partial_A = \overline{\partial}_A, \qquad \quad \partial_{p_A}=\overline{\partial}_{p_A} . $$
For the radial derivative, note first that
$$ \partial_r = -\frac{\xi(r)}{ \Omega^2} \overline{\partial}_{\bar{t}}+\overline{\partial}_r-\frac{2M}{r^2\Omega^2} \cdot \frac{p_{r^*}+\xi(r)p_t}{\Omega^2} \overline{\partial}_{\overline{p}_r}+\frac{1}{\Omega^2} \partial_r (\xi(r)p_t)\overline{\partial}_{\overline{p}_r} .$$
Then, using \eqref{preuvecompartialr} and the null-shell relation \eqref{eq:defConsQua}, we have
$$ \partial_r (p_t) =  - \frac{(r-3M)(|p_t|^2-|p_{r^*}|^2)}{r^2 \Omega^2 p_t} , \qquad \qquad \qquad \partial_r (\xi(r) ) = \frac{27M^2}{r^{\frac{5}{2}} |r+6M|^{\frac{1}{2}} } .$$
This implies, 
 $$ \partial_r = -\frac{\xi(r)}{ \Omega^2} \overline{\partial}_{\bar{t}}+\overline{\partial}_r-\frac{2M}{r^2\Omega^2} \overline{p}_r \overline{\partial}_{\overline{p}_r}-\frac{|\xi (r)|^2(|p_t|^2-|p_{r^*}|^2)}{r^{\frac{1}{2}}|r+6M|^{\frac{1}{2}}\Omega^4p_t} \overline{\partial}_{\overline{p}_r} +\frac{27M^2}{r^{\frac{5}{2}} |r+6M|^{\frac{1}{2}} \Omega^2}  \overline{\partial}_{\overline{p}_r} .$$
Recall now from \eqref{eq:defVminus} the expression of $\pmb{V}_{\! +}=\Omega^{-1}V_+$. We have
\begin{align*}
 \pmb{V}_{\! +} &  = r\frac{\xi(r)}{ \Omega^2} \overline{\partial}_{\bar{t}}+r\partial_r - \bigg( \frac{r-3M}{r} \overline{p}_r +\frac{27M^2 p_t}{r^{\frac{3}{2}}|r+6M|^{\frac{1}{2}}} \bigg) \frac{p_t+\xi(r)p_{r^*}}{ \Omega^2p_t}\overline{\partial}_{\overline{p}_r} \\
 & = \big( r \overline{\partial}_r-\overline{p}_r \overline{\partial}_{\overline{p}_r} \big)+\overline{p}_r \overline{\partial}_{\overline{p}_r} -\frac{2M}{r\Omega^2} \cdot \overline{p}_r \overline{\partial}_{\overline{p}_r}-\frac{r^{\frac{1}{2}}|\xi(r)|^2(|p_t|^2-|p_{r^*}|^2)}{|r+6M|^{\frac{1}{2}}\Omega^4p_t} \overline{\partial}_{\overline{p}_r} +\frac{27M^2p_t}{r^{\frac{3}{2}}|r+6M|^{\frac{1}{2}} \Omega^2 } \overline{\partial}_{\overline{p}_r} \\
 & \quad - \bigg( \frac{r-3M}{r} \overline{p}_r+ \frac{27M^2 p_t}{r^{\frac{3}{2}}|r+6M|^{\frac{1}{2}}} \bigg) \frac{p_t+\xi(r)p_{r^*}}{ \Omega^2p_t}\overline{\partial}_{\overline{p}_r} \\
 & =\big( r \overline{\partial}_r-\overline{p}_r \overline{\partial}_{\overline{p}_r} \big)-\frac{M \overline{p}_r}{r \Omega^2} \overline{\partial}_{\overline{p}_r}-\frac{r^{\frac{1}{2}}|\xi(r)|^2(|p_t|^2-|p_{r^*}|^2)}{|r+6M|^{\frac{1}{2}}\Omega^4p_t} \overline{\partial}_{\overline{p}_r}- \bigg( \frac{r^{\frac{1}{2}}|\xi (r)|^2p_{r^*} \overline{p}_r}{|r+6M|^{\frac{1}{2}}\Omega^2p_t} + \frac{27M^2 \xi(r)p_{r^*} }{r^{\frac{3}{2}}|r+6M|^{\frac{1}{2}}\Omega^2} \bigg) \overline{\partial}_{\overline{p}_r}.
 \end{align*}
The result then follows from
\begin{align*}
 \xi(r)|p_t|^2-\xi (r) |p_{r^*}|^2+\frac{27M^2}{r^2}\Omega^2 p_{r^*}p_t &=p_t(\xi (r) p_t+p_{r^*})-p_tp_{r^*}-\xi (r) |p_{r^*}|^2+\frac{27M^2}{r^2}\Omega^2 p_{r^*}p_t \\
&  =p_t(\xi (r) p_t+p_{r^*})-\xi (r) |p_{r^*}|^2-|\xi(r)|^2 p_{r^*}p_t \\
&  =p_t(\xi (r) p_t+p_{r^*})-\xi(r) p_{r^*}(\xi(r)p_t+p_{r^*}) \\
& =\Omega^2 p_t \overline{p}_r-\Omega^2 \xi(r) p_{r^*} \overline{p}_r .
\end{align*}
\end{proof}


\addtocontents{toc}{\protect\setcounter{tocdepth}{0}}

\bibliographystyle{alpha}
\bibliography{Bibliography.bib} 

\end{document}